\documentclass[12pt,svgnames]{article}
\usepackage{amsmath}
\usepackage{amssymb}
\usepackage{amsthm}
\usepackage{enumerate}
\usepackage{tikz}
\usepackage{tikz-cd}
\usepackage{mathrsfs}
\usepackage{quiver}
\usepackage{hyperref}
\usepackage{fullpage}
\usepackage{enumitem}
\usepackage{titlesec}
\usepackage{stmaryrd}
\usepackage[all]{xy}
\hypersetup{
    linkbordercolor = {PaleTurquoise},
    citebordercolor = {PaleGreen},
}

\usepackage{soul}

\titleformat{\section}{\normalsize\bfseries}{\thesection}{1em}{}
\titleformat{\subsection}{\normalsize\bfseries}{\thesubsection}{1em}{}

\numberwithin{equation}{subsection}

\theoremstyle{plain}

\theoremstyle{definition}

\theoremstyle{definition}

\newtheorem{defn}[subsection]{Definition}
\newtheorem{para}[subsection]{}
\newtheorem{egg}[subsection]{Example}
\newtheorem{rmk}[subsection]{Remark}

\newtheorem*{assumption*}{Assumption}

\theoremstyle{plain}

\newtheorem{prop}[subsection]{Proposition}
\newtheorem{theo}[subsection]{Theorem}
\newtheorem{lem}[subsection]{Lemma}
\newtheorem{cor}[subsection]{Corollary}

\newtheorem*{claim*}{Claim}
\newtheorem*{just*}{Justification}
\newtheorem*{lem*}{Lemma}
\newtheorem*{prop*}{Proposition}

\newcommand{\J}{\mathscr{J}}
\newcommand{\V}{\mathscr{V}}

\newcommand{\JSig}{\Sig_{\underJ}}
\newcommand{\ob}{\mathop{\mathsf{ob}}}
\newcommand{\End}{\mathsf{End}}

\newcommand{\tensor}{\otimes}
\newcommand{\Mnd}{\mathsf{Mnd}}

\newcommand{\Lan}{\mathsf{Lan}}

\newcommand{\underJ}{\kern -0.5ex \mathscr{J}}
\newcommand{\Set}{\mathsf{Set}}

\newcommand{\op}{\mathsf{op}}
\newcommand{\Th}{\mathsf{Th}}
\newcommand{\T}{\mathbb{T}}
\newcommand{\C}{\mathscr{C}}

\newcommand{\B}{\mathscr{B}}
\newcommand{\A}{\mathscr{A}}
\newcommand{\U}{\mathcal{U}}
\newcommand{\Mod}{\text{-}\mathsf{Mod}}

\newcommand{\y}{\mathsf{y}}

\newcommand{\Alg}{\text{-}\mathsf{Alg}}

\newcommand{\K}{\mathscr{K}}
\newcommand{\M}{\mathscr{M}}

\newcommand{\N}{\mathbb{N}}

\newcommand{\Monadic}{\mathsf{Monadic}}

\newcommand{\G}{\mathscr{G}}
\newcommand{\Iso}{\mathsf{Iso}}

\newcommand{\Sig}{\mathsf{Sig}}

\newcommand{\scrT}{\mathscr{T}}
\newcommand{\Pow}{\mathscr{P}}

\newcommand{\X}{\mathscr{X}}

\newcommand{\E}{\mathscr{E}}
\newcommand{\MM}{\mathbb{M}}
\newcommand{\All}{\mathsf{All}}

\newcommand{\CAT}{\text{-}\mathsf{CAT}}
\newcommand{\ALG}{\mathsf{Alg}}

\newcommand{\Gph}{\mathsf{Gph}}

\newcommand{\aff}{\mathsf{aff}}
\newcommand{\Preth}{\mathsf{Preth}}
\newcommand{\MOD}{\mathsf{Mod}}

\newcommand{\Eqn}{\mathsf{Eqn}}
\newcommand{\bbV}{\mathbb{V}}
\newcommand{\calA}{\mathcal{A}}
\newcommand{\calE}{\mathcal{E}}

\newcommand{\Var}{\mathsf{Var}}
\newcommand{\Pos}{\mathsf{Pos}}
\newcommand{\llb}{\llbracket}
\newcommand{\rrb}{\rrbracket}
\newcommand{\SF}{\mathsf{SF}}
\newcommand{\ii}{\mathfrak{i}}
\newcommand{\cS}{\mathcal{S}}
\newcommand{\sfE}{\mathsf{E}}
\newcommand{\cP}{\mathcal{P}}
\newcommand{\tstar}{\mathop{\tilde{*}}}
\newcommand{\sfF}{\mathsf{F}}
\newcommand{\SET}{\mathsf{SET}}

\newcommand{\calF}{\mathcal{F}}
\newcommand{\cD}{\mathcal{D}}
\newcommand{\Sat}{\mathsf{Sat}}
\newcommand{\calB}{\mathcal{B}}

\newcommand{\highm}[1]{#1}
\newcommand{\highl}[1]{#1}

\begin{document}

\title{\Large \textbf{Diagrammatic presentations of enriched monads\\ and varieties for a subcategory of arities}}

\author{Rory B. B. Lucyshyn-Wright\let\thefootnote\relax\thanks{We acknowledge the support of the Natural Sciences and Engineering Research Council of Canada (NSERC), [funding reference numbers RGPIN-2019-05274, RGPAS-2019-00087, DGECR-2019-00273].  Cette recherche a \'et\'e financ\'ee par le Conseil de recherches en sciences naturelles et en g\'enie du Canada (CRSNG), [num\'eros de r\'ef\'erence RGPIN-2019-05274, RGPAS-2019-00087, DGECR-2019-00273].} \medskip \\ Jason Parker \let\thefootnote\relax
\\
\small Brandon University, Brandon, Manitoba, Canada}
\date{}

\maketitle

\begin{abstract}
The theory of \textit{presentations} of enriched monads was developed by Kelly, Power, and Lack, following classic work of Lawvere, and has been generalized to apply to \textit{subcategories of arities} in recent work of Bourke-Garner and the authors.  We argue that, while theoretically elegant and structurally fundamental, such presentations of enriched monads can be inconvenient to construct directly in practice, as they do not directly match the definitional procedures used in constructing many categories of enriched algebraic structures via operations and equations.

Retaining the above approach to presentations as a key technical underpinning, we establish a flexible formalism for directly describing enriched algebraic structure borne by an object of a $\V$-category $\C$ in terms of \textit{parametrized $\J$-ary operations} and \textit{diagrammatic equations} for a suitable subcategory of arities $\J \hookrightarrow \C$.  On this basis we introduce the notions of \textit{diagrammatic $\J$-presentation} and \textit{$\J$-ary variety}, and we show that the category of $\J$-ary varieties is dually equivalent to the category of $\J$-ary $\V$-monads.  We establish several examples of diagrammatic $\J$-presentations and $\J$-ary varieties relevant in both mathematics and theoretical computer science, and we define the \textit{sum} and \textit{tensor product} of diagrammatic $\J$-presentations.  We show that both \textit{$\J$-relative monads} and \textit{$\J$-pretheories} give rise to diagrammatic $\J$-presentations that directly describe their algebras.  Using diagrammatic $\J$-presentations as a method of proof, we generalize the \textit{pretheories-monads adjunction} of Bourke and Garner beyond the locally presentable setting.  Lastly, we generalize Birkhoff's Galois connection between classes of algebras and sets of equations to the above setting.
\end{abstract}

\section{Introduction}
Universal algebra begins with the idea of a set $A$ equipped with a family of finitary operations satisfying specified equations.  This is usually formalized by beginning with a \textit{signature} (or \textit{similarity type}), i.e. a set $\Sigma$ (whose elements we call \textit{operation symbols}) equipped with an assignment to each $\sigma \in \Sigma$ a finite cardinal $n_\sigma$ called the \textit{arity} of $\sigma$.  A \textit{$\Sigma$-algebra} is then a set $A$ equipped with a family of operations $\sigma^A:A^{n_\sigma} \rightarrow A$ indexed by the operation symbols $\sigma \in \Sigma$.  Syntactic expressions or \textit{terms} over the signature $\Sigma$ allow the specification of \textit{equations} over $\Sigma$, and by collecting together only those $\Sigma$-algebras that satisfy a specified set of equations $E$ we arrive at the notion of \textit{variety of algebras} that is fundamental to algebra in general.  The $\Sigma$-algebras in the given variety, and their homomorphisms, form a category, and one may then call the pair $P = (\Sigma,E)$ a \textit{presentation} of this category of algebras.  One of the complications inherent in this approach to algebra is that several different presentations may present isomorphic categories of algebras.  This issue was addressed by Lawvere \cite{Law:Phd}, whose \textit{algebraic theories} (or \textit{Lawvere theories}) classify varieties of algebras up to a suitable notion of isomorphism, and with the insights of Linton \cite{Linton:Eq} it was soon realized that Lawvere theories are equivalently given by \textit{finitary monads} on the category of sets, $\Set$, so that each presentation $P = (\Sigma,E)$ presents a finitary monad.

It is known that universal algebraic concepts admit a generalization to a setting in which sets are replaced by the objects of a symmetric monoidal closed category $\V$ or, more generally, a $\V$-enriched category $\C$.  Early contributions to this line of development in enriched category theory include works of Kock \cite{Kock}, Dubuc \cite{Dubucbook,Dubucsemantics}, Borceux and Day \cite{BorceuxDay}, Kelly and Power \cite{KellyPower}, Kelly and Lack \cite{KellyLackstronglyfinitary}.  Generalizing finitary monads, a recent theme in this area has been the study of $\V$-enriched monads for a \textit{subcategory of arities} $\J \hookrightarrow \C$, i.e. a full subcategory that is dense in the enriched sense, and whose objects $J \in \ob\J$ play the role of \textit{arities}.  In the setting of ordinary $\Set$-enriched category theory, this line of development originates with Linton \cite{Lintonoutline} and includes the paper of Berger, Melli\`es, and Weber \cite{BMW}, while enriched monads for subcategories of arities have been studied by the first author \cite{EAT}, by Bourke and Garner \cite{BourkeGarner}, and by the authors \cite{Pres}; also see \cite{LR} for related work in terms of classes of weights.

In the present paper, we develop a versatile framework and methodology for directly describing enriched algebraic structure in terms of operations and equations, relative to a suitable small subcategory of arities $\J \hookrightarrow \C$.  A technical underpinning for this paper is provided by the authors' recent paper \cite{Pres}, which generalizes certain prior contributions of Kelly, Power, and Lack \cite{KellyPower,KellyLackstronglyfinitary,Lackmonadicity} in the study of presentations of finitary enriched monads via signatures.

But here we depart substantially from these prior works, in that we employ a more flexible formalism for describing enriched algebraic structure that more directly matches the definitional procedures that are typically used in constructing specific examples of enriched algebraic categories, thus more directly reflecting mathematical practice.

Indeed, here we formulate enriched algebraic structure borne by an object $A$ of a suitable $\V$-category $\C$ as consisting of a family of \textit{parametrized $\J$-ary operations}, each of which is given by a morphism of the form
$$\C(J,A) \otimes C \rightarrow A$$
in $\C$, where the \textit{arity} $J$ is an object of a given subcategory of arities $\J \hookrightarrow \C$ and $C$ is an object of $\C$ that we call the \textit{parameter} (while $\otimes$ here is the action of $\V$ on $\C$ by tensor).  Such parametrized operations can be described equivalently as morphisms
$$\C(J,A) \rightarrow \C(C,A)$$
in $\V$.  For example, if $\C = \V = \Set$ then a parametrized operation may be written in the form $A^J \times C \rightarrow A$ or $A^J \rightarrow A^C$.

We introduce the notion of \textit{free-form $\J$-signature}, which \highl{is a span $\ob\J \leftarrow \cS \to \ob\C$ (i.e.~a span from $\ob\J$ to $\ob\C$) in which $\cS$ is a \textit{small} set.  Equivalently, a free-form $\J$-signature} is a small set $\cS$, whose elements we call \textit{operation symbols}, equipped with an assignment to each operation symbol $\sigma \in \cS$ an object $J_\sigma$ of $\J$ called the \textit{arity} and an object $C_\sigma$ of $\C$ called the \textit{parameter}.  An \textit{$\cS$-algebra} is then an object $A$ of $\C$ equipped with a family of parametrized $\J$-ary operations
$$\sigma^A\;:\;\C(J_\sigma,A) \longrightarrow \C(C_\sigma,A)\;\;\;\;\;\;(\sigma \in \cS),$$
and we obtain a $\V$-category of $\cS$-algebras, $\cS\Alg$, which we show is \textit{strictly $\J$-ary monadic} over $\C$, i.e., isomorphic to the $\V$-category of Eilenberg-Moore algebras for an associated $\J$-ary $\V$-monad $\T_\cS$.

In order to impose equations on $\cS$-algebras in an intuitive and practically convenient way, we systematically employ Lawvere's notion of \textit{algebraic structure} \cite{Law:Phd} and its enriched analogue developed by Dubuc \cite{Dubucbook}.  Given a free-form $\J$-signature $\cS$, we define a \textit{diagrammatic $\cS$-equation} $\omega \doteq \nu$ to be a family of parallel pairs of the form
$$\omega_A, \nu_A\;:\;\C(J,A) \rightrightarrows \C(C,A),$$
$\V$-natural in $A \in \cS\Alg$, where $J \in \ob\J$ and $C \in \ob\C$.  Such diagrammatic equations correspond bijectively to parallel pairs $t,u:C \rightrightarrows T_\cS J$ (i.e., pairs of Kleisli morphisms for $\T_\cS$ from an object of $\C$ to an object of $\J$), but the formulation in terms of diagrammatic $\cS$-equations is by far the more convenient formalism for constructing specific examples in the enriched context, because, as the nomenclature suggests, we simply construct whatever diagrammatic laws we want an $\cS$-algebra $A$ to satisfy---using whatever constructions are available in the host $\V$-category $\C$---and we then observe that the resulting parallel pairs of composite morphisms are $\V$-natural in $A \in \cS\Alg$ by construction.

We define a \textit{diagrammatic $\J$-presentation} $\cP = (\cS,\sfE)$ to consist of a free-form $\J$-signature $\cS$ together with a small family of diagrammatic $\cS$-equations $\sfE$, and we say that an $\cS$-algebra is a \textit{$\cP$-algebra} if it satisfies the diagrammatic equations in $\sfE$.  We show that the $\V$-category of $\cP$-algebras, $\cP\Alg$, is strictly $\J$-ary monadic over $\C$, so that $\cP$ determines an associated $\J$-ary $\V$-monad $\T_\cP$ on $\C$.  We define a \textit{$\J$-ary variety} to be a $\V$-category of the form $\cP\Alg$, regarded as an object of the slice $\V\CAT/\C$, and we show that the category of $\J$-ary varieties is dually equivalent to the category of $\J$-ary $\V$-monads on $\C$, thus generalizing several results in the literature.  We establish several examples of diagrammatic $\J$-presentations and $\J$-ary varieties, as well as certain constructions on diagrammatic $\J$-presentations including the \textit{sum} and \textit{tensor product}.  We further illustrate the versatility of diagrammatic $\J$-presentations by showing that the algebras of \textit{$\J$-relative $\V$-monads} \cite{Altenkirchmonads} and of \textit{$\J$-pretheories} \cite{BourkeGarner} may be described as algebras for diagrammatic $\J$-presentations; thus we establish establish new and more general results on the monadicity of these categories of structures.  We pursue further applications and ramifications of diagrammatic $\J$-presentations and $\J$-ary varieties, including a generalization of the \textit{pretheories-monads adjunction} of Bourke and Garner \cite{BourkeGarner} beyond the locally presentable setting, as well as a generalization of Birkhoff's Galois connection \cite{Birkhoff:Lattice} between classes of algebras and sets of equations to the present enriched setting of free-form $\J$-signatures $\cS$ and (parametrized) diagrammatic $\cS$-equations.

We now contrast the above concepts with the more traditional formalism for presentations of $\J$-ary $\V$-monads used in the authors' paper \cite{Pres}, which provides a fundamental technical basis for this work and generalizes methods developed by Kelly, Power, and Lack \cite{KellyPower,KellyLackstronglyfinitary}.  With this more traditional formalism, one defines a $\J$-signature $\Sigma$ to be a family of objects $\Sigma J$ of $\C$, indexed by the objects $J \in \ob\J$, thinking of $\Sigma J$ as the \textit{object of internal operation symbols} of arity $J$.  The theoretical importance of this approach lies in the result that the category of $\J$-ary $\V$-monads on $\C$ is \textit{monadic} over the category of these `traditional' $\J$-signatures \cite{Pres}, generalizing a result for finitary monads proved by Lack \cite{Lackmonadicity}.  Thus every $\J$-ary $\V$-monad has a canonical \textit{presentation} as a coequalizer, leading to a more traditional concept of presentation, namely a parallel pair of morphisms $\Gamma \rightrightarrows \T_\Sigma$ from a given (traditional) $\J$-signature $\Gamma$ into the $\J$-signature underlying the free $\J$-ary $\V$-monad $\T_\Sigma$ on $\Sigma$; see \cite{Pres}.  This notion of \mbox{$\J$-presentation} is theoretically fundamental and yet practically inconvenient for capturing many specific examples of categories of enriched algebraic structures, because it forces the user of the formalism to (1) explicitly construct each object of internal operation symbols $\Sigma J$ $(J \in \ob\J)$, often as a coproduct, and (2) similarly construct the objects $\Gamma J$ $(J \in \ob\J)$ and the morphisms $\Gamma J \rightrightarrows T_\Sigma J$, a process that is often tedious and just as often does not resemble the manner in which the given enriched algebraic structures are defined in practice.  Thus, we argue that \textit{diagrammatic} $\J$-presentations provide an essential practical tool for enriched algebra, enabling familiar definitional procedures to be directly employed in describing enriched algebraic structures, and thus providing an indispensable `user-friendly interface' to the fundamental theory of presentations via monadicity of $\J$-ary monads.

The diagrammatic $\J$-presentations that we introduce here should also be compared and contrasted with the \textit{term equational systems} of Fiore and Hur \cite{termequationalsystems} and the \textit{monadic equational systems} of Fiore \cite{Fiore:MES}.  Indeed, these works involve arbitrary sets of parallel pairs of morphisms $C \rightrightarrows TB$ for a $\V$-monad $\T$ on a $\V$-category $\C$, and the interpretation of such morphisms as parametrized operations on $\T$-algebras.  But these works do not employ a given subcategory of arities $\J \hookrightarrow \C$ and, moreover, they are not founded on the monadicity of any particular class of $\V$-monads over any class of signatures, whereas the present paper begins with a subcategory of arities $\J \hookrightarrow \C$ and is founded upon the monadicity of $\J$-ary $\V$-monads over $\J$-signatures.

We now outline the structure of the paper.  In \S\ref{background}, we review elements of the authors' paper \cite{Pres}, which provides a theoretical underpinning for this work, and we discuss our assumptions on the given closed category $\V$, the $\V$-category $\C$, and the subcategory of arities $\J \hookrightarrow \C$; these relatively mild assumptions allow the theory in this paper to be applied to wide classes of closed categories $\V$ and $\V$-categories $\C$ that need not be locally presentable, including the \textit{locally bounded closed categories} of \cite{Kelly} and the \textit{locally bounded $\V$-categories} of \cite{locbd}.  In \S\ref{sec:param_ops_interp}, we delineate those aspects of parametrized operations and diagrammatic equations that can be formulated for an arbitrary $\V$-monad, without reference to the subcategory of arities.

In \S\ref{sec:sys_eqns} we define \textit{systems of diagrammatic $\J$-ary equations} on $\J$-ary monadic $\V$-categories $\A$ over $\C$, and we show that the full sub-$\V$-category of $\A$ described by such a system of equations is $\J$-ary monadic over $\C$.

In \S\ref{freeform}, we define the fundamental concepts of free-form $\J$-signature, diagrammatic $\J$-presentation, and $\J$-ary variety, and we show not only that the $\V$-category of algebras for a diagrammatic $\J$-presentation is $\J$-ary monadic over $\C$, but also that there is a dual equivalence between $\J$-ary varieties and $\J$-ary $\V$-monads on $\C$.

In \S\ref{presentationexamples} we construct several examples of diagrammatic $\J$-presentations and diagrammatic systems of $\J$-ary equations that present various $\V$-categories of enriched mathematical structures, including the following: internal $R$-modules and $R$-affine spaces for an internal rig (or semiring) $R$ in a complete and cocomplete cartesian closed category $\V$ (\ref{Rmodules}, \ref{internalaffinespaces}); monoidal categories, regarded as $\mathsf{Cat}$-enriched algebraic structures in $\mathsf{Cat}$ (\ref{monoidalcategories}); internal categories in a locally presentable closed category $\V$ (\ref{internalcategoriesexample}), regarded as $\V$-enriched algebraic structures in the $\V$-category of internal graphs in $\V$; the global state algebras of Plotkin and Power \cite{PlotkinPowernotions} (\ref{PPexamples}); and some of the parametrized algebraic theories of Staton \cite{Statoninstances} (\ref{Statonexamples}).  The latter two examples illustrate the applicability of diagrammatic $\J$-presentations for algebraic computational effects.

In \S\ref{sumtensor} we further illustrate the versatility of diagrammatic $\J$-presentations by constructing the sum and tensor product (when $\C = \V$) of a given pair of diagrammatic $\J$-presentations; the tensor product generalizes Freyd's tensor product of finitary algebraic presentations \cite{Freydtensor}, and it can be expressed very conveniently in this formalism, since we simply write out the diagrammatic laws that express that the operations for the first presentation must \textit{commute} with those for the second.

In \S\ref{standardizedpresentations}, we show that every \textit{$\J$-relative $\V$-monad} $H$ (cf. \cite{Altenkirchmonads}) gives rise to a diagrammatic $\J$-presentation whose algebras are precisely those of $H$.  We thus establish general results on the monadicity of such algebras, deducing as a corollary that for any small full subcategory $\K$ of a locally presentable $\V$-category $\C$ over a locally presentable closed category $\V$, the $\V$-category of algebras for any $\K$-relative $\V$-monad is monadic by way of an accessible $\V$-monad.

In \S\ref{monadpretheory}, as one of the further main contributions of the paper, we use diagrammatic $\J$-presentations as a method of proof in order to generalize the \textit{pretheories-monads adjunction} of Bourke and Garner \cite{BourkeGarner} beyond the locally presentable setting.  We achieve this by showing that the algebras for a pretheory are precisely the algebras for an associated diagrammatic $\J$-presentation and so are the objects of a strictly $\J$-ary monadic $\V$-category over $\C$.

In \S\ref{Galoisconnection}, we generalize Birkhoff's well-known Galois connection between classes of algebras and sets of equations \cite{Birkhoff:Lattice} to the present setting of free-form $\J$-signatures $\cS$ and (parametrized) diagrammatic $\cS$-equations, thus enabling the consideration of \textit{the $\J$-ary variety over $\cS$ generated by a class of $\cS$-algebras}, and we prove characterization theorems for $\J$-ary varieties with reference to the resulting Galois connection.

\section{Background and given data}  
\label{background}

We assume that the reader has familiarity with enriched category theory, which is exposited (for example) in \cite{Kelly, Dubucbook} and \cite[Chapter 6]{Borceux2}. We use the term ``set" throughout to refer to a not-necessarily-small set (which we sometimes emphasize by prepending the phrase ``possibly large" in parentheses); when we wish to specify that a set is \emph{small}, we always use the terminology ``small set''. We have the category $\Set$ of small sets and the category $\SET$ of sets. We now recall the needed background, definitions, and results from our previous work \cite{Pres}, and in \ref{runningexamples} we recall several examples.

\begin{para}
\label{Vcategoryassumptions}
Throughout, we let $\V$ be a locally small, symmetric monoidal closed category that is complete and cocomplete. We also assume that $\V$ is a \textbf{closed factegory} (see \cite[6.1.2]{Pres}), meaning that $\V$ is equipped with a $\V$-enriched factorization system $(\E, \M)$ (see \cite{enrichedfact}) that need not be proper. 

Throughout, we let $\C$ be a \textbf{$\V$-factegory} \cite[6.1.5]{Pres}, i.e., a $\V$-category $\C$ equipped with an enriched factorization system $(\E_\C, \M_\C)$ such that each $\C(C, -) : \C \to \V$ ($C \in \ob\C$) sends $\M_\C$-morphisms to $\M$-morphisms.  We assume that the $\V$-category $\C$ is cocomplete and cotensored and has arbitrary conical cointersections (i.e. wide pushouts) of $\E_\C$-morphisms. Also, in order to invoke certain theorems in \cite{Pres}, we suppose either that $(\E_\C, \M_\C)$ is proper or that $\C$ is $\E_\C$-cowellpowered. \qed 
\end{para}

\begin{para}
\label{subcategoryaritiesassumptions}
{
We assume that the $\V$-category $\C$ is equipped with a small \textbf{subcategory of arities} $j : \J \hookrightarrow \C$, i.e. a small, full, and dense sub-$\V$-category. We suppose that $j : \J \hookrightarrow \C$ is \textbf{eleutheric}, which means that every $\V$-functor $H : \J \to \C$ has a left Kan extension along $j$ that is preserved by each $\C(J, -) : \C \to \V$ ($J \in \ob\J$), or equivalently that $j : \J \hookrightarrow \C$ presents $\C$ as a free $\Phi$-cocompletion of $\J$ for a class of small weights $\Phi$ (see \cite[3.8]{Pres}).  In particular, $j$ then presents $\C$ as a free $\Phi_{\underJ}$-cocompletion of $\J$ for the class $\Phi_{\underJ}$ consisting of the weights $\C(j-,C):\J^\op \rightarrow \V$ with $C \in \ob\C$ \cite[3.6]{Pres}.

We also assume that $j : \J \hookrightarrow \C$ is \textbf{bounded}, meaning that ($\J$ is small and) there is a regular cardinal $\alpha$ for which each $\C(J, -) : \C \to \V$ ($J \in \ob\J$) \emph{preserves the $\E$-tightness of small $\alpha$-filtered $\M$-cocones} (see \cite[6.1.6, 6.1.10]{Pres}).  In the special case where both $(\E_\C, \M_\C)$ and $(\E, \M)$ are the trivial factorization system $(\Iso, \All)$, this reduces to the requirement that each $\C(J, -) : \C \to \V$ ($J \in \ob\J$) preserve small conical $\alpha$-filtered colimits.  See \ref{runningexamples} for examples of bounded and eleutheric subcategories of arities. \qed  
}
\end{para}

\begin{para}
\label{Jary}
A $\V$-endofunctor $H : \C \to \C$ is \textbf{$\J$-ary} (or \textbf{$j$-ary}) if it preserves left Kan extensions along $j : \J \hookrightarrow \C$, equivalently, if $H$ is $\Phi_{\underJ}$-cocontinuous for the class of weights $\Phi_{\underJ}$ of \ref{subcategoryaritiesassumptions}.  Moreover, if $\Phi$ is any class of small weights such that $j$ presents $\C$ as a free $\Phi$-cocompletion of $\J$ (e.g.~$\Phi_{\underJ}$), then a $\V$-functor $H : \C \to \C$ is $\J$-ary iff $H$ is $\Phi$-cocontinuous \cite[4.2]{Pres}. We let $\End_{\underJ}(\C)$ be the (ordinary) category of $\J$-ary $\V$-endofunctors on $\C$, which is a (non-symmetric) strict monoidal category under composition. A \textbf{$\J$-ary $\V$-monad} is a $\V$-monad $\T$ on $\C$ whose underlying $\V$-endofunctor $T$ is $\J$-ary.  We write $\Mnd_{\underJ}(\C)$ to denote the category of $\J$-ary $\V$-monads, i.e., the category of monoids in the monoidal category $\End_{\underJ}(\C)$. Since $j : \J \hookrightarrow \C$ is eleutheric, we have an equivalence of categories $\V\CAT(\J, \C) \simeq \End_{\underJ}(\C)$ given by precomposition with $j$ and left Kan extension along $j$ (see \cite[4.5]{Pres}). 

A \textbf{$\V$-category over $\C$} is an object of the (strict) slice category $\V\CAT/\C$, i.e., a $\V$-category $\A$ equipped with a $\V$-functor $G:\A \rightarrow \C$.  We denote such an object of $\V\CAT/\C$ either by $(\A,G)$ or simply by $\A$.  Given a $\V$-monad $\T$ on $\C$, we write $\T\Alg$ to denote the $\V$-category of $\T$-algebras, and we regard $\T\Alg$ as a $\V$-category over $\C$ by way of the forgetful $\V$-functor $U^\T:\T\Alg \rightarrow \C$.

We say that a $\V$-functor $G:\A \rightarrow \C$ is \textbf{strictly $\J$-monadic} (or \textbf{strictly $\J$-ary monadic}), or that $\A$ is a \textbf{strictly $\J$-monadic $\V$-category over $\C$}, if $\A \cong \T\Alg$ in $\V\CAT/\C$ for some $\J$-ary $\V$-monad $\T$.  It then follows that $G$ has a left adjoint $F$ such that the $\V$-monad induced by the $\V$-adjunction $F \dashv G$ is precisely $\T$. Writing $\J\text{-}\Monadic^!$ for the full subcategory of $\V\CAT/\C$ consisting of the strictly $\J$-monadic $\V$-categories over $\C$, we have an equivalence $\J\text{-}\Monadic^! \simeq \Mnd_{\underJ}(\C)^\op$ by \cite[\S 9.3]{Pres}\qed
\end{para}

\begin{egg}
\label{runningexamples}
The following examples satisfy the blanket assumptions of  \ref{Vcategoryassumptions} and \ref{subcategoryaritiesassumptions}; see  \cite[4.7, 12.3]{Pres}:

\begin{enumerate}[leftmargin=*]
\item the subcategory of arities $\C_\alpha \hookrightarrow \C$ given by a skeleton of the full sub-$\V$-category of enriched $\alpha$-presentable objects in a locally $\alpha$-presentable $\V$-category $\C$ over a locally $\alpha$-presentable closed category $\V$ \cite{Kellystr}; in this class of examples, $\C_\alpha$-ary $\V$-monads are precisely \textit{$\alpha$-ary} or \textit{$\alpha$-accessible} $\V$-monads, i.e. $\V$-monads that preserve conical $\alpha$-filtered colimits, which correspond to the $\alpha$-ary variant of the enriched Lawvere theories of \cite{NishizawaPower};

\item the subcategory of arities $\SF(\V) \hookrightarrow \V$ consisting of the finite copowers of the unit object in a symmetric monoidal closed $\pi$-category  \cite{BorceuxDay} (e.g. a complete and cocomplete cartesian closed category); in this class of examples, $\SF(\V)$-ary $\V$-monads correspond to the enriched algebraic theories of Borceux and Day \cite{BorceuxDay}, and, in the case where $\V$ is cartesian closed, $\SF(\V)$-ary $\V$-monads are precisely the \textit{strongly finitary $\V$-monads} of Lack and Kelly \cite{KellyLackstronglyfinitary}, since in this case $\SF(\V)$ is the free $\V$-category on the category of finite cardinals (with all maps between them);

\item the subcategory of arities $\{I\} \hookrightarrow \V$ consisting of the unit object in a complete and cocomplete symmetric monoidal closed category $\V$; in this class of examples, $\{I\}$-ary $\V$-monads correspond to monoids in $(\V,\otimes,I)$;

\item the subcategory of arities $\y_{\scriptscriptstyle\X} : \X^\op \hookrightarrow [\X, \V]$ consisting of the representables, for a small $\V$-category $\X$ and a complete and cocomplete closed category $\V$; in the special case where $\X$ is the discrete $\V$-category on a set $X$, $\y_{\scriptscriptstyle\X}$-ary $\V$-monads on $[\X,\V] = \V^X$ correspond to $\V$-categories with object set $X$;

\item the subcategory of arities $\y_\Phi : \scrT^\op \hookrightarrow \Phi\Mod(\scrT)$ consisting of the representables in the $\V$-category of models of a $\Phi$-theory $\scrT$, where $\Phi$ is a locally small class of small weights satisfying Axiom A of \cite{LR} and $\V$ is a locally bounded and $\E$-cowellpowered closed category; in this class of examples, $\y_\Phi$-ary $\V$-monads are precisely the \textit{$\Phi$-accessible $\V$-monads} on $\Phi\Mod(\scrT)$ in the sense defined by Lack and Rosick\'{y} in \cite{LR};

\item any small and eleutheric subcategory of arities in a locally bounded $\V$-category $\C$ over a locally bounded closed category $\V$ \cite{locbd}.
\end{enumerate}
In examples 1 through 4, both $\V$ and $\C$ carry the $(\Iso,\All)$ factorization system, while in examples 5 and 6 both $\V$ and $\C$ carry associated proper factorization systems as discussed in \cite{locbd}. \qed
\end{egg}

\begin{para}\label{sys_ar}
In the case where $\C = \V$, the subcategory of arities $\J \hookrightarrow \V$ is said to be a \textbf{system of arities} if $\J$ is closed under the monoidal product $\tensor$ and contains the unit object $I$ of $\V$ \cite[3.8]{EAT}.  For an eleutheric system of arities $\J \hookrightarrow \V$, it was shown in \cite[11.8]{EAT} that there is an equivalence $\Mnd_{\underJ}(\V) \simeq \Th_{\underJ}$ between $\J$-ary $\V$-monads on $\V$ and \emph{$\J$-theories} (see \cite[4.1]{EAT}) that respects semantics, in the sense that if the $\J$-ary $\V$-monad $\T$ and the $\J$-theory $\scrT$ correspond under this equivalence, then the $\V$-category $\scrT\Alg$ of \textit{$\scrT$-algebras} \cite[5.1]{EAT} is equivalent to $\T\Alg$, while the $\V$-category $\scrT\Alg^!$ of \emph{normal} $\scrT$-algebras \cite[5.10]{EAT} is isomorphic to $\T\Alg$ in  $\V\CAT/\V$ (see \cite[11.14]{EAT}). The subcategories of arities $\SF(\V) \hookrightarrow \V$ and $\{I\} \hookrightarrow \V$ of \ref{runningexamples} are systems of arities \cite[3.6, 3.7]{EAT}, as is the subcategory of arities $\V_\alpha \hookrightarrow \V$ in a locally $\alpha$-presentable closed category $\V$ \cite[3.4]{EAT}. \qed 
\end{para}

\begin{para}
\label{signatures}
{
A \textbf{$\J$-signature} (in $\C$) is a $\V$-functor $\Sigma : \ob\J \to \C$, where we write $\ob\J$ to denote the discrete $\V$-category on $\ob\J$; hence $\Sigma$ is just an $\ob\J$-indexed family of objects of $\C$. We then have the ordinary category $\JSig(\C) = \V\CAT(\ob\J,\C)$ of $\J$-signatures \cite[7.1]{Pres}.

If $\Sigma$ is a $\J$-signature, then a \textbf{$\Sigma$-algebra} is an object $A$ of $\C$ equipped with structural morphisms $\alpha_J:\C(J, A) \tensor \Sigma J \to A$ ($J \in \ob\J$), or equivalently $\alpha_J:\C(J, A) \to \C(\Sigma J, A)$, or equivalently $\alpha_J:\Sigma J \to [\C(J, A), A]$, where we denote the tensor and cotensor of each object $C$ of $\C$ by each object $X$ of $\V$ by $X \otimes C$ and $[X,C]$, respectively. We write such a $\Sigma$-algebra either as $(A,\alpha)$ or simply as $A$. Given $\Sigma$-algebras $A = (A,\alpha)$ and $B = (B,\beta)$, a \textbf{$\Sigma$-homomorphism} $f:A \rightarrow B$ is a morphism in $\C$ that makes the following diagram commute for each $J \in \ob\J$:
\[\begin{tikzcd}
	\C(J, A) \tensor \Sigma J && \C(J, B) \tensor \Sigma J \\
	A && B
	\arrow["\alpha_J"', from=1-1, to=2-1]
	\arrow["{\C(J, f) \tensor \Sigma J}", from=1-1, to=1-3]
	\arrow["\beta_J", from=1-3, to=2-3]
	\arrow["f", from=2-1, to=2-3]
\end{tikzcd}\]
We then have a $\V$-category $\Sigma\Alg$ of $\Sigma$-algebras with a faithful $\V$-functor $U^\Sigma : \Sigma\Alg \to \C$ (see \cite[7.4]{Pres}).

The forgetful functor $\U : \Mnd_{\underJ}(\C) \to \JSig(\C)$ given by $\T \mapsto (TJ)_{J \in \ob\J}$ has a left adjoint by \cite[7.9]{Pres}, and the free $\J$-ary $\V$-monad $\T_\Sigma$ on a $\J$-signature $\Sigma$ moreover satisfies $\T_\Sigma\Alg \cong \Sigma\Alg$ in $\V\CAT/\C$ by \cite[7.9]{Pres}. The forgetful functor $\U : \Mnd_{\underJ}(\C) \to \JSig(\C)$ is also monadic by \cite[8.2]{Pres}. \qed   
}
\end{para}

\begin{para}
\label{algebraiccolimits}
If $\MM : \K \to \Mnd_{\underJ}(\C)$ is a small diagram of $\J$-ary $\V$-monads on $\C$, then $\MM$ has a colimit $\T_\MM$ that is \textbf{algebraic} (see \cite[9.3.8]{Pres}), meaning that it is sent to a limit by the semantics functor $\ALG : \Mnd_{\underJ}(\C)^\op \to \V\CAT/\C$ (given on objects by $\T \mapsto \T\Alg$). In particular, $\Mnd_{\underJ}(\C)$ has \textbf{algebraic coequalizers}, meaning that each parallel pair $\T \rightrightarrows \T'$ of $\J$-ary $\V$-monads has a coequalizer in $\Mnd_{\underJ}(\C)$ that is sent to an equalizer by the semantics functor $\ALG$. \qed 
\end{para} 

\begin{para}
\label{presentations}
Given a $\J$-ary $\V$-monad $\T$ on $\C$, we write simply $\T$ to denote the $\J$-signature underlying $\T$ (\ref{signatures}).  A \textbf{system of $\J$-ary equations over $\T$} (see \cite[10.1.2]{Pres}) is a parallel pair $E = (t, u : \Gamma \rightrightarrows \T)$ of $\J$-signature morphisms, where $\Gamma$ is a $\J$-signature. A \textbf{system of $\J$-ary equations} is then a pair $(\T, E)$ consisting of a $\J$-ary $\V$-monad $\T$ and a system of $\J$-ary equations $E$ over $\T$. A \textbf{$\J$-presentation} is a pair $P = (\Sigma, E)$ consisting of a $\J$-signature $\Sigma$ and a system of $\J$-ary equations $E = (t, u : \Gamma \rightrightarrows \T_\Sigma)$ over the free $\J$-ary $\V$-monad $\T_\Sigma$ on $\Sigma$.  As we also consider other notions of signature and presentation relative to $\J$ in this paper, we refer to $\J$-signatures, $\J$-presentations, and systems of $\J$-ary equations also as \textbf{traditional $\J$-signatures}, \textbf{traditional $\J$-presentations}, and \textbf{traditional systems of $\J$-ary equations} for clarity.

If $(\T,E)$ is a system of $\J$-ary equations, where $E = (t, u : \Gamma \rightrightarrows \T)$,  then we write $\T/E$ to denote the algebraic coequalizer of the parallel pair $t^\sharp, u^\sharp : \T_\Gamma \rightrightarrows \T$ in $\Mnd_{\underJ}(\C)$ induced by $t, u$ via the adjunction between $\Mnd_{\underJ}(\C)$ and $\JSig(\C)$, and we call the $\J$-ary $\V$-monad $\T/E$ the \textbf{quotient of $\T$ by $E$} \cite[10.1.6]{Pres}.  In particular, if $P = (\Sigma,E)$ is a $\J$-presentation, then $(\T_\Sigma,E)$ is a system of $\J$-ary equations, and we call $\T_P := \T_\Sigma / E$ the \textbf{$\J$-ary $\V$-monad presented by $P$}.  We then say that a given $\J$-ary $\V$-monad is \textbf{presented by $P$} if it is isomorphic to $\T_P$.  Every $\J$-ary $\V$-monad $\T$ has a canonical $\J$-presentation $P_\T$ (see \cite[10.1.9]{Pres}), so that $\T\Alg \cong P_\T\Alg$ in $\V\CAT/\C$. \qed
\end{para}

\begin{para}
\label{algebrasforpresentation}
Let $\T = (T, \eta, \mu)$ be a $\J$-ary $\V$-monad on $\C$. Given any object $C$ of $\C$, there is a $\J$-ary $\V$-monad $\langle C, C\rangle$ on $\C$ with $\langle C, C\rangle J = [\C(J, C), C]$ for each $J \in \ob\J$ (see \cite[10.2.9]{Pres}).  Given a $\T$-algebra $(A,a)$, which we also write simply as $A$, there is a canonical $\V$-monad morphism $\ii^A:\T \to \langle A, A\rangle$, which we call the \textbf{$\J$-ary interpretation morphism} for $A$ \cite[10.2.10]{Pres}. For each $J \in \ob\J$ the component $\ii^A_J:TJ \to \langle A, A\rangle J = [\C(J, A), A]$ is the transpose of the composite $\C(J,A) \xrightarrow{T_{JA}} \C(TJ,TA) \xrightarrow{\C(TJ,a)} \C(TJ,A)$ \cite[10.2.1]{Pres}.

Given a (traditional) system of $\J$-ary equations $(\T,E)$, where $E = (t, u : \Gamma \rightrightarrows \T)$, a \textbf{$(\T,E)$-algebra} is, by definition, a $\T$-algebra $A$ such that $\ii^A_J \circ t_J = \ii^A_J \circ u_J:\Gamma J \rightarrow [\C(J,A),A]$ for each $J \in \ob\J$ (equivalently, such that $\ii^A \circ t^\sharp = \ii^A \circ u^\sharp:\T_\Gamma \rightarrow \langle A,A\rangle$ with the notation of \ref{presentations}).  We let $(\T,E)\Alg$ be the full sub-$\V$-category of $\T\Alg$ consisting of the $(\T,E)$-algebras, and we regard $(\T,E)\Alg$ as a $\V$-category over $\C$ by means of the $\V$-functor  $(\T,E)\Alg \rightarrow \C$ obtained as a restriction of $U^\T$.  Writing $\T/E$ for the quotient of $\T$ by $E$ (\ref{presentations}), we have an isomorphism $(\T/E)\Alg \cong (\T,E)\Alg$ in $\V\CAT/\C$ by \cite[10.2.13]{Pres}, under which each $(\T/E)$-algebra $(A,a)$ corresponds to the $(\T,E)$-algebra $(A,a \circ q^E_A)$ where $q^E:\T \rightarrow \T/E$ is the regular epimorphism that presents $\T/E$ as a quotient. 

Given a $\J$-presentation $P = (\Sigma,E)$, where $E = (t,u:\Gamma \rightrightarrows \T_\Sigma)$, a \textbf{$P$-algebra} is a $\Sigma$-algebra $A$ such that $\ii^A_J \circ t_J = \ii^A_J \circ u_J$ for all $J \in \ob\J$, where we write $\ii^A:\T_\Sigma \rightarrow \langle A,A\rangle$ for the interpretation morphism obtained by regarding $A$ equivalently as a $\T_\Sigma$-algebra via the isomorphism $\Sigma\Alg \cong \T_\Sigma\Alg$ in $\V\CAT/\C$ (\ref{signatures}).  Hence, a $P$-algebra is precisely a $\Sigma$-algebra whose corresponding $\T_\Sigma$-algebra is a $(\T_\Sigma,E)$-algebra for the system of $\J$-ary equations $(\T_\Sigma,E)$.  We write $P\Alg$ for the full sub-$\V$-category of $\Sigma\Alg$ consisting of $P$-algebras, and we regard $P\Alg$ as a $\V$-category over $\C$ by means of the $\V$-functor $P\Alg \rightarrow \C$ obtained as a restriction of $U^\Sigma$.  Hence $P\Alg \cong (\T_\Sigma,E)\Alg \cong \T_P\Alg$ in $\V\CAT/\C$ \cite[10.2.14]{Pres} where $\T_P = \T_\Sigma / E$ is the $\J$-ary $\V$-monad presented by $P$ (\ref{presentations}). \qed
\end{para}
 
 \section{Parametrized operations and interpretation relative to a $\V$-adjunction}\label{sec:param_ops_interp}
 
Everything in the present section applies when $\C$ is an arbitrary tensored and cotensored $\V$-category, so that the full force of the assumptions in \S \ref{background} is not required.
 
\begin{defn}\label{param_op}
Given objects $A,J,C \in \ob\C$, a \textbf{parametrized operation} on $A$ with \textbf{arity} $J$ and \textbf{parameter object} (or \textbf{parameter}) $C$ is a morphism $\omega:\C(J,A) \rightarrow \C(C,A)$ in $\V$.  Such a parametrized operation is equivalently given by a morphism $\omega:\C(J,A) \otimes C \rightarrow A$ in $\C$, or by a morphism $\omega:C \rightarrow [\C(J,A),A]$ in $\C$. When $\C = \V$, we may also write $\omega : C \tensor \V(J, A) \to A$. \qed
\end{defn}

\noindent Later we shall restrict attention to the case where the arity $J$ is an object of the specified subcategory of arities $\J \hookrightarrow \C$, while crucially we shall \textit{not} require the parameter $C$ to be an object of $\J$.  It is for this reason that we deliberately choose to call the object $C$ a \emph{parameter} rather than an \emph{(output) arity}. The following generalizes \ref{param_op} (which is recovered in the case where $\A$ is the unit $\V$-category):
 
\begin{defn}\label{defn:param_ops_on_vfunctor}
Let $G:\A \rightarrow \C$ be a $\V$-functor, so that we may regard $\A$ as a $\V$-category over $\C$ (\ref{Jary}).  Given objects $J,C \in \ob\C$, a \textbf{pa\-ra\-me\-tri\-zed operation on $G$} (or \textbf{on $\A$}) with \textbf{arity} $J$ and \textbf{parameter} $C$ is a $\V$-natural transformation $\omega:\C(J,G-) \rightarrow \C(C,G-)$.  A \textbf{diagrammatic equation on $G$} (or \textbf{on $\A$}), written as $\omega \doteq \nu$, is a parallel pair of $\V$-natural transformations $\omega,\nu:\C(J,G-) \rightrightarrows \C(C,G-)$ for a specified pair of objects $J,C$ of $\C$; an object $A \in \ob\A$ \textbf{satisfies} the equation $\omega \doteq \nu$ if $\omega_A = \nu_A$. \qed
\end{defn}

\noindent As a variation on \cite[II.3]{Dubucbook}, if $G:\A \rightarrow \C$ is a $\V$-functor then we write $\mathsf{Str}_0(G)$ to denote the (ordinary) category whose objects are those of $\C$ and whose morphisms $\omega:X \rightarrow Y$ are the parametrized operations $\omega:\C(X,G-) \rightarrow \C(Y,G-)$ with arity $X$ and parameter $Y$.  The following is a consequence of \cite[Proposition II.1.5 and p. 81]{Dubucbook}:

\begin{prop}\label{interp_iso}
Let $F \dashv G:\A \to \C$ be a $\V$-adjunction with unit $\eta$ and counit $\varepsilon$, and let $\T$ be the induced $\V$-monad on $\C$.  Then the category $\mathsf{Str}_0(G)$ underlies a $\V$-category $\mathsf{Str}(G)$ (called the \textnormal{$\V$-structure} of $G$) that is isomorphic to the opposite of the Kleisli $\V$-category $\C_\T$ for $\T$.  Under the latter isomorphism $\mathsf{Str}(G) \cong \C_\T^\op$, if $t:C \rightarrow J$ is a morphism in $\C_\T$, given by a morphism $t:C \rightarrow TJ$ in $\C$ with $C,J \in \ob\C$, then the corresponding morphism in $\mathsf{Str}(G)$ is the $\V$-natural transformation $\C(J,G-) \rightarrow \C(C,G-)$ consisting of the composite morphisms
\begin{equation}\label{eq:interp_comp}\C(J,GA) \xrightarrow{T_{J,GA}} \C(TJ,TGA) \xrightarrow{\C(t,a_A)} \C(C,GA)\end{equation}
associated to objects $A$ of $\A$, where we write $a_A:TGA \to GA$ for the $\T$-algebra structure on $GA$ defined as $a_A = G\varepsilon_A$.  In the opposite direction, if $\omega:\C(J,G-) \rightarrow \C(C,G-)$ is a parametrized operation on $G$, then the corresponding morphism $C \rightarrow J$ in $\C_\T$ is the morphism $C \rightarrow TJ$ in $\C_0$ obtained as the composite $I \xrightarrow{\eta_J} \C(J,TJ) \xrightarrow{\omega_{FJ}} \C(C,TJ)$ in $\V$.\qed
\end{prop}

\begin{defn}\label{param_tterm}
Given any $\V$-monad $\T$ on $\C$, a \textbf{parametrized $\T$-term (with arity $J$ and parameter $C$)} is a morphism $t : C \to TJ$ for a pair of objects $J,C \in \ob\C$.  A \textbf{parametrized $\T$-equation}, denoted by $t \doteq u:C \rightrightarrows TJ$, is a parallel pair $t, u : C \rightrightarrows TJ$ of parametrized $\T$-terms for a specified pair of objects $J,C \in \ob\C$.\footnote{Fiore and Hur \cite{termequationalsystems} use the terms \textit{generalised term} and \textit{generalised equation} to refer to these concepts.} \qed
\end{defn}

\begin{defn}\label{defn:interp_in_talg}
Let $G:\A \to \C$ be a right adjoint $\V$-functor, and let $\T$ be the induced $\V$-monad on $\C$. Given a parametrized $\T$-term $t:C \rightarrow TJ$ and an object $A$ of $\A$, we denote the composite morphism in \eqref{eq:interp_comp} by
\begin{equation}\label{eq:interp_in_A}\llb t \rrb_A\;:\;\C(J,GA) \longrightarrow \C(C,GA)\end{equation}
and we call $\llb t \rrb_A$ the \textbf{interpretation of $t$ in $A$}.  Hence $\llb t\rrb_A$ is a parametrized operation on $GA$ with arity $J$ and parameter $C$.  In view of \ref{interp_iso}, the resulting family of morphisms \eqref{eq:interp_in_A} is $\V$-natural in $A \in \A$ and so is a parametrized operation on $G$ that we write as
$$\llb t\rrb\;:\;\C(J,G-) \longrightarrow \C(C,G-)$$
and call simply the \textbf{interpretation of $t$ (in $\A$)}.  Note that $\llb t\rrb$ is the morphism in $\mathsf{Str}(G)$ that corresponds to $t$ under the isomorphism $\mathsf{Str}(G) \cong \C_\T^\op$ of \ref{interp_iso}.  Given a parametrized $\T$-equation $t \doteq u:C \rightrightarrows TJ$, we say that an object $A$ of $\A$ \textbf{satisfies} $t \doteq u$, and we write $A \models t \doteq u$, if $\llb t \rrb_A = \llb u \rrb_A:\C(J,GA) \rightarrow \C(C,GA)$.  In other words, $A$ satisfies $t \doteq u$ iff $A$ satisfies the diagrammatic equation $\llb t \rrb \doteq \llb u \rrb:\C(J,G-) \rightrightarrows \C(C,G-)$ consisting of the interpretations of $t$ and $u$ in $\A$.\qed
\end{defn}

\begin{egg}\label{interp_in_talg}
Given a $\V$-monad $\T$ on $\C$, we may apply \ref{interp_iso} and \ref{defn:interp_in_talg} to the right adjoint $\V$-functor $U^\T:\T\Alg \rightarrow \C$.  In particular, given a $\T$-algebra $A = (A,a)$, we may consider the interpretation $\llb t \rrb_A = \C(t,a) \circ T_{JA}:\C(J,A) \to \C(C,A)$ of a parametrized $\T$-term $t:C \to TJ$ in $A$, and we may ask whether $A$ satisfies a given parametrized $\T$-equation $t \doteq u$, i.e. whether $\llb t\rrb_A = \llb u\rrb_A$.
\end{egg}

\begin{prop}\label{teqns_that_hold_in_every_talg}
Let $t \doteq u:C \rightrightarrows TJ$ be a parametrized $\T$-equation, where $\T = (T,\eta,\mu)$ is a $\V$-monad on $\C$.  The following are equivalent: (1) every $\T$-algebra $A$ satisfies $t \doteq u$; (2) the free $\T$-algebra $(TJ,\mu_J)$ satisfies $t \doteq u$; (3) $t = u$.
\end{prop}
\begin{proof}
(1) holds iff $\llb t \rrb = \llb u\rrb$, and the result now readily follows from \ref{interp_iso}.
\end{proof}

\begin{para}\label{mnd_morph_alg}
Let $\lambda:\T \rightarrow \mathbb{S}$ be a morphism of $\V$-monads on $\C$, and let $A = (A,a)$ be an $\mathbb{S}$-algebra.  Then $(A,a \circ \lambda_A)$ is a $\T$-algebra, which we also denote simply by $A$, by abuse of notation.  In particular, if $t:C \rightarrow TJ$ is a parametrized $\T$-term then we write $\llb t \rrb_A$ to denote the interpretation of $t$ in the $\T$-algebra $(A,a \circ \lambda_A)$, and if $t \doteq u$ is a parametrized $\T$-equation then we say that that $A$ \textit{satisfies} $t \doteq u$ if $(A,a \circ \lambda_A)$ satisfies $t \doteq u$.
\end{para}

\begin{prop}\label{mnd_morph_interp}
Let $\lambda:\T \rightarrow \mathbb{S}$ be a morphism of $\V$-monads on $\C$, and let $A$ be an $\mathbb{S}$-algebra.  With the notation of \ref{mnd_morph_alg}, $\llb t\rrb_A = \llb \lambda_J \circ t\rrb_A$ for each parametrized $\T$-term $t:C \rightarrow TJ$.  Hence, with the terminology of \ref{mnd_morph_alg}, $A$ satisfies a parametrized $\T$-equation $t \doteq u:C \rightrightarrows TJ$ if and only if $A$ satisfies the parametrized $\mathbb{S}$-equation $\lambda_J \circ t \doteq \lambda_J \circ u:C \rightrightarrows SJ$.
\end{prop}
\begin{proof}
For the first claim, we may use \ref{interp_in_talg} and the $\V$-naturality of $\lambda$ to compute that $\llb t \rrb_A = \C(t,a \circ \lambda_A) \circ T_{JA} = \C(t,a) \circ \C(TJ,\lambda_A) \circ T_{JA} = \C(t,a) \circ \C(\lambda_J,SA) \circ S_{JA} = \C(\lambda_J \circ t,a) \circ S_{JA} = \llb \lambda_J \circ t\rrb_A$.  The second claim follows from the first.
\end{proof}

\begin{cor}\label{teqn_equalized_by_mnd_morph}
Let $\lambda:\T \rightarrow \mathbb{S}$ be a morphism of $\V$-monads on $\C$, and let $t \doteq u:C \rightrightarrows TJ$ be a parametrized $\T$-equation.  Then the following are equivalent, where we employ the terminology of \ref{mnd_morph_alg}: (1) every $\mathbb{S}$-algebra $A$ satisfies $t \doteq u$; (2) the free $\mathbb{S}$-algebra on $J$ satisfies $t \doteq u$; (3) $\lambda_J \circ t = \lambda_J \circ u$.
\end{cor}
\begin{proof}
This follows immediately from \ref{mnd_morph_interp} and \ref{teqns_that_hold_in_every_talg}. 
\end{proof}

\section{Systems of diagrammatic equations on \texorpdfstring{$\J$}{J}-monadic \texorpdfstring{$\V$}{V}-categories over \texorpdfstring{$\C$}{C}}\label{sec:sys_eqns}

\begin{para}
\label{param_Jary_operation}
A \textbf{parametrized $\J$-ary operation} is a parametrized operation whose arity $J$ is an object of the given subcategory of arities $\J$.  Moreover, we add the word \textit{$\J$-ary} to the various terms in \ref{param_op}, \ref{defn:param_ops_on_vfunctor}, \ref{param_tterm} to signify that the arity $J$ is an object of $\J$, thus arriving at the notions of \textbf{diagrammatic $\J$-ary equation}, \textbf{parametrized $\J$-ary $\T$-term}, and \textbf{parametrized $\J$-ary $\T$-equation}.\qed
\end{para}

\begin{defn}
\label{diagrammatic_Jary_eqn_on_g}
Let $G:\A \rightarrow \C$ be a strictly $\J$-monadic $\V$-functor, so that $\A$ may be regarded as a strictly $\J$-monadic $\V$-category over $\C$.  A \textbf{system of diagrammatic $\J$-ary equations on $G$} (or \textbf{on $\A$}) is a small family $\sfE = (\omega_\delta \doteq \nu_\delta)_{\delta \in \cD}$ of diagrammatic $\J$-ary equations $\omega_\delta, \nu_\delta : \C(J_\delta, G-) \rightrightarrows \C(C_\delta, G-)$ on $G$, indexed by a small set $\cD$. An \textbf{$\sfE$-model} (or a \textbf{model of $\sfE$}) is then an object $A$ of $\A$ that satisfies every equation $\omega_\delta \doteq \nu_\delta$ ($\delta \in \cD$). We write $\A_\sfE$ for the full sub-$\V$-category of $\A$ consisting of the $\sfE$-models, and we regard $\A_\sfE$ as a $\V$-category over $\C$ by restricting $G$ to $\A_\sfE$. \qed
\end{defn}

\noindent We shall show in \ref{Jary_eqn_monadic_prop} that $\A_\sfE$ is strictly $\J$-monadic over $\C$ in the situation of \ref{diagrammatic_Jary_eqn_on_g}.  But in order to do this, we first consider the monad-based counterpart of the concept in \ref{diagrammatic_Jary_eqn_on_g}: 

\begin{defn}\label{defn_ff_sys_jary_eqns}
Given a $\J$-ary $\V$-monad $\T$, a \textbf{system of parametrized $\J$-ary equations} over $\T$ is a small family $\sfE = (t_\delta \doteq u_\delta)_{\delta \in \cD}$ of parametrized $\J$-ary $\T$-equations $t_\delta \doteq u_\delta:C_\delta \rightrightarrows TJ_\delta$ (indexed by a small set $\cD$).  A \textbf{system of parametrized $\J$-ary equations} is a pair $(\T,\sfE)$ consisting of a $\J$-ary $\V$-monad $\T$ and a system of parametrized $\J$-ary equations $\sfE$ over $\T$.  A \textbf{$(\T,\sfE)$-algebra} is a $\T$-algebra $A$ that satisfies every equation in $\sfE$. We write $(\T,\sfE)\Alg$ to denote the full sub-$\V$-category of $\T\Alg$ consisting of the $(\T,\sfE)$-algebras, and we regard $(\T,\sfE)\Alg$ as a $\V$-category over $\C$ by means of the $\V$-functor $(\T,\sfE)\Alg \rightarrow \C$ obtained as a restriction of $U^\T$.\qed
\end{defn}

\begin{para}\label{bij_diag_ff}
Let $G:\A \rightarrow \C$ be a strictly $\J$-monadic $\V$-functor, so that $\A$ is a strictly $\J$-monadic $\V$-category over $\C$, and let $\T$ be the induced $\J$-ary $\V$-monad on $\C$.  In view of \ref{interp_iso}, systems of diagrammatic $\J$-ary equations $\sfE$ on $\A$ are in bijective correspondence with systems of parametrized $\J$-ary equations $\sfE$ over $\T$, so that we may identify these concepts. With this identification, $\A_\sfE \cong (\T,\sfE)\Alg$ in $\V\CAT/\C$, because $\A \cong \T\Alg$ and it is immediate from \ref{interp_iso} and \ref{interp_in_talg} that an object $A$ of $\A$ satisfies a given diagrammatic equation $\omega \doteq \nu$ on $\A$ iff the associated $\T$-algebra $(GA,a_A)$ satisfies the corresponding parametrized $\T$-equation, with the notation of \ref{interp_iso}. Hence, in order to show that $\A_\sfE$ is strictly $\J$-monadic over $\C$, it suffices to show that $(\T,\sfE)\Alg$ is so, and for this we shall need the following material. \qed
\end{para}

\begin{para}\label{interp_morph_vs_interp_of_t}
Given a $\J$-ary $\V$-monad $\T$ on $\C$ and an object $J$ of $\J$, the identity morphism $1_{TJ}$ may be regarded as a parametrized $\J$-ary $\T$-term with arity $J$ and parameter $TJ$, so for each $\T$-algebra $A = (A,a)$ we obtain a parametrized $\J$-ary operation $\llb 1_{TJ}\rrb_A = \C(TJ,a) \circ T_{JA}:\C(J,A) \to \C(TJ,A)$, which may be described also as the transpose of the interpretation morphism $\ii^A_J:TJ \rightarrow [\C(J,A),A]$ of \ref{algebrasforpresentation}.  Given any parametrized $\J$-ary $\T$-term $t:C \to TJ$, the interpretation $\llb t\rrb_A = \C(t,a) \circ T_{JA}:\C(J,A) \rightarrow \C(C,A)$ is therefore the transpose of the composite $\ii^A_J \circ t:C \rightarrow [\C(J,A),A]$.  Hence, given any parametrized $\J$-ary $\T$-equation $t \doteq u:C \rightrightarrows TJ$, $A \vDash t \doteq u$ iff $\ii^A_J \circ t = \ii^A_J \circ u$.  It follows that if $(f_\lambda:C_\lambda \rightarrow C)_{\lambda \in \Lambda}$ is a jointly epimorphic family of morphisms in $\C$, then $A \vDash t \doteq u$ iff $A \vDash t \circ f_\lambda \doteq u \circ f_\lambda$ for all $\lambda \in \Lambda$. \qed
\end{para}

\begin{para}\label{sys_jary_eqs_as_ff}
Given a traditional system of $\J$-ary equations $E = (t,u:\Gamma \rightrightarrows \T)$ over a $\J$-ary $\V$-monad $\T$ (\ref{presentations}), we may regard $E$ as a system of parametrized $\J$-ary equations $\sfE = (t_J \doteq u_J:\Gamma J \rightrightarrows TJ)_{J \in \ob\J}$ over $\T$, and it follows immediately from \ref{interp_morph_vs_interp_of_t} that $(\T,\sfE)\Alg = (\T,E)\Alg$ as $\V$-categories over $\C$.  In the other direction, every system of parametrized $\J$-ary equations determines a traditional system of $\J$-ary equations, as follows:
\end{para}

\begin{prop}\label{free_form_to_trad}
Let $\sfE = \left(t_\delta \doteq u_\delta : C_\delta \rightrightarrows TJ_\delta\right)_{\delta \:\in\: \cD}$ be a system of parametrized $\J$-ary equations over a $\J$-ary $\V$-monad $\T$.  Then there is a traditional system of $\J$-ary equations $E = \left(\tilde{t}, \tilde{u} : \Gamma \rightrightarrows \T\right)$ over $\T$ such that $(\T,\sfE)\Alg = (\T,E)\Alg$ as $\V$-categories over $\C$.
\end{prop}
\begin{proof}
Define a traditional $\J$-signature $\Gamma$ by declaring that $\Gamma J = \coprod_{\delta \in \cD,\:J_\delta = J} C_\delta$ for each $J \in \ob\J$, and define $\tilde{t}_J,\tilde{u}_J:\Gamma J \rightrightarrows TJ$ to be the unique morphisms such that $\tilde{t}_J \circ \iota_\delta = t_\delta$ and $\tilde{u}_J \circ \iota_\delta = u_\delta$ for all $\delta \in \cD$ with $J_\delta = J$, where $\iota_\delta:C_\delta \to \Gamma J$ is the coproduct insertion.  The result now follows readily in view of \ref{interp_morph_vs_interp_of_t}, since the $\iota_\delta$ are jointly epimorphic.
\end{proof}

\begin{theo}\label{quot_by_ff_sys_eqns}
Let $(\T,\sfE)$ be a system of parametrized $\J$-ary equations.  Then there is a $\J$-ary $\V$-monad $\T/\sfE$ such that $(\T/\sfE)\Alg \cong (\T,\sfE)\Alg$ in $\V\CAT/\C$.
\end{theo}
\begin{proof}
Taking $E$ as in \ref{free_form_to_trad} and letting $\T/\sfE := \T/E$ with the notation of \ref{algebrasforpresentation}, this follows from  \ref{free_form_to_trad} and \ref{algebrasforpresentation}. 
\end{proof}

\begin{theo}
\label{Jary_eqn_monadic_prop}
Let $\sfE$ be a system of diagrammatic $\J$-ary equations on a strictly $\J$-monadic $\V$-category $\A$ over $\C$. Then $\A_\sfE$ is a strictly $\J$-monadic $\V$-category over $\C$.
\end{theo}
\begin{proof}
This follows from \ref{quot_by_ff_sys_eqns}, by the discussion in \ref{bij_diag_ff}.
\end{proof}

\begin{rmk}\label{ae_jary_monadic}
In the situation of \ref{Jary_eqn_monadic_prop}, if we write $\T$ for the $\J$-ary $\V$-monad induced by $\A$ (i.e., by the associated $\V$-adjunction $F \dashv G$, \ref{Jary}), then the $\J$-ary $\V$-monad induced by $\A_\sfE$ is $\T/\sfE$ (in view of \ref{Jary}, \ref{bij_diag_ff}, and the proof of \ref{Jary_eqn_monadic_prop}), and there is an associated regular epimorphism $q:\T \rightarrow \T/\sfE$ in $\Mnd_{\underJ}(\C)$ (by \ref{algebrasforpresentation} and the proof of \ref{quot_by_ff_sys_eqns}). 
\end{rmk}

\section{Free-form and diagrammatic $\J$-presentations}
\label{freeform}

We now study objects equipped with parametrized operations whose arities and parameters are specified by a \textit{signature} of the following kind:

\begin{defn}
\label{freeformsignature}
A \textbf{free-form $\J$-signature} is a small set $\cS$ equipped with an assignment to each $\sigma \in \cS$ a pair of objects $J_\sigma,C_\sigma$ with $J_\sigma \in \ob\J$ and $C_\sigma \in \ob\C$.  We call each element $\sigma \in \cS$ an \textbf{operation symbol} with \textbf{arity} $J_\sigma$ and \textbf{parameter} $C_\sigma$.  We denote such a free-form $\J$-signature simply by $\cS$, and we then call $\cS$ a \textbf{free-form $\J$-signature with arities $J_\sigma$ and parameters $C_\sigma$ $(\sigma \in \cS)$}. \qed
\end{defn}

\begin{defn}
\label{ff_signature_algebra}
Given a free-form $\J$-signature $\cS$, with arities $J_\sigma$ and parameters $C_\sigma$ $(\sigma \in \cS)$, an \textbf{$\cS$-algebra} is an object $A$ of $\C$ equipped with parametrized operations $\sigma^A : \C(J_\sigma, A) \to \C(C_\sigma, A)$ for all the operation symbols $\sigma \in \cS$. An $\cS$-algebra is equivalently an object $A$ of $\C$ with a family of morphisms $\sigma^A:\C(J_\sigma,A) \otimes C_\sigma \rightarrow A$ in $\C$ $(\sigma \in \cS)$, or equivalently, a family of morphisms $\sigma^A:C_\sigma \rightarrow [\C(J_\sigma,A),A]$ in $\C$.  In the special case where $\C = \V$, we may also write $\sigma^A : C_\sigma \tensor \V(J_\sigma, A) \to A$.  Given $\cS$-algebras $A$ and $B$, an \textbf{$\cS$-homomorphism} $f:A \to B$ is a morphism in $\C$ that makes the following diagram commute for each $\sigma \in \cS$:
\[\begin{tikzcd}
	{\C(J_\sigma, A)} && {\C(C_\sigma, A)} \\
	\\
	{\C(J_\sigma, B)} && {\C(C_\sigma, B)}.
	\arrow["{\sigma^A}", from=1-1, to=1-3]
	\arrow["{\C(J_\sigma, f)}"', from=1-1, to=3-1]
	\arrow["{\sigma^B}"', from=3-1, to=3-3]
	\arrow["{\C(C_\sigma, f)}", from=1-3, to=3-3]
\end{tikzcd}\]
We thus have an ordinary category $\cS\Alg_0$ of $\cS$-algebras and $\cS$-homomorphisms, which underlies a $\V$-category $\cS\Alg$ in which each hom-object $\cS\Alg(A,B)$ is defined as the \textit{pairwise equalizer} \cite[2.1]{Commutants} of the following $\cS$-indexed family of parallel pairs in $\V$ $(\sigma \in \cS)$:
$$\C(A,B) \xrightarrow{\C(J_\sigma,-)_{AB}} \V(\C(J_\sigma,A),\C(J_\sigma,B)) \xrightarrow{\V(1,\sigma^B)} \V(\C(J_\sigma,A),\C(C_\sigma,B)),$$
$$\C(A,B) \xrightarrow{\C(C_\sigma,-)_{AB}} \V(\C(C_\sigma,A),\C(C_\sigma,B)) \xrightarrow{\V(\sigma^A,1)} \V(\C(J_\sigma,A),\C(C_\sigma,B))\;.$$
Composition in $\cS\Alg$ is defined in the unique way that enables the resulting subobjects $U^\cS_{AB}:\cS\Alg(A,B) \hookrightarrow \C(A,B)$ $(A,B \in \cS\Alg)$ to serve as the structural morphisms of a faithful $\V$-functor $U^\cS:\cS\Alg \rightarrow \C$ that sends each $\cS$-algebra $A$ to its carrier object $A \in \ob\C$. \qed   
\end{defn}

\begin{defn}
\label{Xparam_Salg}
Let $\cS$ be a free-form $\J$-signature. Given a $\V$-category $\X$, an \textbf{$\X$-pa\-ra\-me\-tri\-zed $\cS$-algebra} is a $\V$-functor $A : \X \to \C$ equipped with parametrized operations $\sigma^A : \C(J_\sigma, A-) \to \C(C_\sigma, A-)$ \eqref{defn:param_ops_on_vfunctor} for all $\sigma \in \cS$.  Writing $\cS\Alg[\X]$ to denote the (large) set of all $\X$-parametrized $\cS$-algebras, we obtain an evident functor $\cS\Alg[-] : \V\CAT^\op \to \mathsf{SET}$ given on objects by $\X \mapsto \cS\Alg[\X]$. \qed 
\end{defn}

\begin{para}
\label{SAlg_representing_object}
Let $\cS$ be a free-form $\J$-signature.  It readily follows from the definition of $U^\cS : \cS\Alg \to \C$ \eqref{ff_signature_algebra} that for each operation symbol $\sigma \in \cS$, the parametrized operations $\sigma^A : \C(J_\sigma, A) \to \C(C_\sigma, A)$ for $\cS$-algebras $A$ constitute a parametrized operation $\sigma^{U^\cS} : \C\left(J_\sigma, U^\cS-\right) \to \C\left(C_\sigma, U^\cS-\right)$ on the $\V$-functor $U^\cS$. Thus, the $\V$-functor $U^\cS : \cS\Alg \to \C$ carries the structure of an $\cS\Alg$-parametrized $\cS$-algebra. It is now straightforward to verify the following result, which characterizes $\cS\Alg$ uniquely up to isomorphism:

\begin{prop}
\label{universal_property_SAlg}
Let $\cS$ be a free-form $\J$-signature. Then $\cS\Alg$ is a representing object for the functor $\cS\Alg[-] : \V\CAT^\op \to \SET$, and the counit of this representation (in the sense of \cite[\S 1.10]{Kelly}) is $U^\cS : \cS\Alg \to \C$. Thus, there is a bijective correspondence, natural in $\X \in \V\CAT$, between $\V$-functors $\X \to \cS\Alg$ and $\X$-parametrized $\cS$-algebras. \qed 
\end{prop}   
\end{para}

\begin{para}\label{trads_det_ffs}
Every traditional $\J$-signature $\Sigma$ (\ref{presentations}) determines a (rather special) free-form $\J$-signature $\cS_\Sigma$ that is defined by declaring that the set of operation symbols underlying $\cS_\Sigma$ is $\ob\J$ and that each operation symbol $J \in \ob\J$ has arity $J$ and parameter $\Sigma J$.  With this notation, it is straightforward to verify that $\Sigma\Alg \cong \cS_\Sigma\Alg$ in $\V\CAT/\C$.  In fact, if we formulate $\Sigma$-algebras equivalently in terms of structure morphisms of the form $\alpha_J:\C(J,A) \rightarrow \C(\Sigma J,A)$ as mentioned in \ref{algebrasforpresentation}, then in fact $\Sigma\Alg = \cS_\Sigma\Alg$ as $\V$-categories over $\C$.\qed
\end{para} 

\begin{para}\label{sig_det_by_ffsig}
In the opposite direction from \ref{trads_det_ffs}, given a free-form $\J$-signature $\cS$ with arities $J_\sigma$ and parameters $C_\sigma$ $(\sigma \in \cS)$, we can construct a traditional $\J$-signature $\Sigma_\cS : \ob\J \to \C$ as follows: for each $J \in \ob\J$, we set $\Sigma_\cS J := \coprod_{\sigma \in \cS,  J_\sigma = J} C_\sigma$. It then follows straightforwardly that $\Sigma_\cS$-algebras (in the sense of \ref{signatures}) are in bijective correspondence with $\cS$-algebras, and that we have an isomorphism of $\V$-categories $\cS\Alg \cong \Sigma_\cS\Alg$ in $\V\CAT/\C$. To prove the latter claim, note that $\Sigma_\cS\Alg = \cS_{\Sigma_\cS}\Alg$ by \ref{trads_det_ffs}, so that by \ref{universal_property_SAlg} it suffices to show that we have a bijective correspondence, natural in $\X \in \V\CAT$, between $\X$-parametrized $\cS$-algebras and $\X$-parametrized $\cS_{\Sigma_\cS}$-algebras, and this readily follows from the definitions. (Note that $\X$-parametrized $\cS_{\Sigma_\cS}$-algebras can reasonably be called \textit{$\X$-parametrized $\Sigma_\cS$-algebras}.) \qed
\end{para}

\begin{theo}\label{salg_jmonadic}
Given a free-form $\J$-signature $\cS$, the $\V$-category of $\cS$-algebras $\cS\Alg$ is strictly $\J$-monadic over $\C$, i.e., there is a $\J$-ary $\V$-monad $\T_\cS$ on $\C$ such that $\cS\Alg \cong \T_\cS\Alg$ in $\V\CAT/\C$.
\end{theo}
\begin{proof}
Letting $\T_\cS := \T_{\Sigma_\cS}$ be the free $\J$-ary $\V$-monad on the traditional $\J$-signature $\Sigma_\cS$ (\ref{signatures}, \ref{sig_det_by_ffsig}), we find that $\cS\Alg \cong \Sigma_{\cS}\Alg \cong \T_\cS\Alg$ in $\V\CAT/\C$ by \ref{sig_det_by_ffsig} and \ref{signatures}.
\end{proof}

\begin{defn}\label{jary_vmnd_gend_by_ffsig}
We call the $\V$-monad $\T_\cS = \T_{\Sigma_\cS}$ in \ref{salg_jmonadic} the \textbf{$\J$-ary $\V$-monad generated by $\cS$}. \qed
\end{defn}

\begin{defn}\label{diagrammatic_seqn}
Let $\cS$ be a free-form $\J$-signature.  A \textbf{$\V$-natural $\cS$-operation} is a parametrized $\J$-ary operation on the $\V$-functor $U^\cS:\cS\Alg \rightarrow \C$.  Hence, a $\V$-natural $\cS$-operation $\omega$ with arity $J \in \ob\J$ and parameter $C \in \ob\C$ is precisely a family of morphisms
$$\omega_A\;:\;\C(J,A) \longrightarrow \C(C,A)\;\;\;\text{in $\V$}\;\;\;\;\;\;(A \in \cS\Alg)$$
that is $\V$-natural in $A \in \cS\Alg$, where we write simply $A$ to denote $U^\cS A$.  A \textbf{diagrammatic $\cS$-equation}, denoted by $\omega \doteq \nu$, is a diagrammatic $\J$-ary equation on $U^\cS$, i.e. a parallel pair of $\V$-natural $\cS$-operations
$$\omega_A,\nu_A\;:\;\C(J,A) \rightrightarrows \C(C,A)\;\;\;\;\;\;(A \in \cS\Alg)$$
with a specified arity $J \in \ob\J$ and a specified parameter $C \in \ob\C$.  In view of \ref{param_op}, a diagrammatic $\cS$-equation may be written equivalently in the form
$$\omega_A,\nu_A\;:\;\C(J,A) \otimes C \rightrightarrows A\;\;\;\;\;\;(A \in \cS\Alg).$$
With the terminology of \ref{defn:param_ops_on_vfunctor}, an $\cS$-algebra $A$ \textbf{satisfies} the diagrammatic $\cS$-equation $\omega \doteq \nu$ if $\omega_A = \nu_A$.  A \textbf{system of diagrammatic $\cS$-equations} is a system of diagrammatic $\J$-ary equations on $U^\cS$, i.e., a small family $\sfE = (\omega_\delta \doteq \nu_\delta)_{\delta \:\in\: \cD}$ of diagrammatic $\cS$-equations $\omega_\delta \doteq \nu_\delta:\C(J_\delta,U^\cS-) \rightrightarrows \C(C_\delta,U^\cS-)$, indexed by a small set $\cD$.  The index set $\cD$ then carries the structure of a free-form $\J$-signature with arities $J_\delta$ and parameters $C_\delta$ $(\delta \in \cD)$. \qed
\end{defn}

\begin{para}
\label{signature_operation_natural}
By \ref{SAlg_representing_object}, if $\cS$ is a free-form $\J$-signature, then each operation symbol $\sigma \in \cS$ determines a $\V$-natural $\cS$-operation $\sigma^A : \C(J_\sigma, A) \to \C(C_\sigma, A)$ $(A \in \cS\Alg$). We shall use this fact very often in the examples of \S\ref{presentationexamples}. \qed 
\end{para}

\begin{defn}
A \textbf{diagrammatic $\J$-presentation} is a pair $\cP = (\cS,\sfE)$ consisting of a free-form $\J$-signature $\cS$ and a system of diagrammatic $\cS$-equations $\sfE$.  A \textbf{$\cP$-algebra} is an $\cS$-algebra $A$ that satisfies every equation in $\sfE$. \qed
\end{defn}

\begin{defn}
\label{param_Seqn}
Let $\cS$ be a free-form $\J$-signature.  A \textbf{parametrized $\cS$-term} is a parametrized $\J$-ary $\T_\cS$-term $t:C \rightarrow T_\cS J$, where $\T_\cS$ is the $\J$-ary $\V$-monad generated by $\cS$.  A \textbf{parametrized $\cS$-equation} is a parametrized $\J$-ary $\T_\cS$-equation $t \doteq u$, i.e. a parallel pair $t,u:C \rightrightarrows T_\cS J$ with $J \in \ob\J$ and $C \in \ob\C$. \qed
\end{defn}

\begin{para}\label{interp_in_salg}
Applying \ref{defn:interp_in_talg} to the $\V$-functor $U^\cS:\cS\Alg \rightarrow \C$, we may consider the interpretation $\llb t \rrb_A:\C(J,A) \to \C(C,A)$ of a parametrized $\cS$-term $t:C \rightarrow T_\cS J$ in an $\cS$-algebra $A$, and we may ask whether $A$ satisfies a parametrized $\cS$-equation $t \doteq u:C \rightrightarrows T_\cS J$, in which case we write $A \models t \doteq u$. \qed
\end{para}

\begin{defn}
\label{ff_Jpres}
A \textbf{free-form $\J$-presentation} is a pair $\cP = (\cS,\sfE)$ consisting of a free-form $\J$-signature $\cS$ and a system of parametrized $\J$-ary equations $\sfE$ over the $\J$-ary $\V$-monad $\T_\cS$ generated by $\cS$ (i.e., a small family of parametrized $\cS$-equations).  A \textbf{$\cP$-algebra} is an $\cS$-algebra $A$ that satisfies every equation in $\sfE$. \qed
\end{defn}

\begin{rmk}
\label{bijection_diag_ff}
By \ref{bij_diag_ff}, there is a bijection between diagrammatic $\J$-presentations and free-form $\J$-presentations, so that we may identify these concepts. \qed
\end{rmk}

\begin{defn}\label{algs_for_free_form_pres}
Given a free-form or diagrammatic $\J$-presentation $\mathcal{P} = (\cS,\sfE)$, we write $\cP\Alg$ to denote the full sub-$\V$-category of $\cS\Alg$ consisting of the $\cP$-algebras, and we regard $\cP\Alg$ as a $\V$-category over $\C$ by means of the $\V$-functor $U^\cP:\cP\Alg \rightarrow \C$ obtained as a restriction of $U^\cS$. \qed
\end{defn}

\begin{rmk}\label{palg_arbitrary_j}
The notions of free-form $\J$-signature $\cS$ and diagrammatic $\J$-presentation $\cP$ can be defined relative to an arbitrary full sub-$\V$-category $\J \hookrightarrow \C$, without the assumptions in \ref{Vcategoryassumptions} and \ref{subcategoryaritiesassumptions}, as can the notion of $\cP$-algebra and the $\V$-category of $\cP$-algebras, $\cP\Alg$.  We shall make use of this observation in \S\ref{standardizedpresentations}. \qed
\end{rmk}

\begin{rmk}
Given a diagrammatic $\J$-presentation $\cP = (\cS,\sfE)$, $\cP$-algebras are precisely $\sfE$-models for the system of diagrammatic $\J$-ary equations $\sfE$ on $\cS\Alg$, and moreover $\cP\Alg = \cS\Alg_\sfE$ as $\V$-categories over $\C$, so Theorems \ref{Jary_eqn_monadic_prop} and \ref{salg_jmonadic} (together with \ref{ae_jary_monadic}) entail the following:
\end{rmk}

\begin{theo}\label{vmnd_pres_by_ff_jpres}
Let $\cP = (\cS,\sfE)$ be a free-form or diagrammatic $\J$-presentation.  Then the $\V$-category of $\cP$-algebras $\cP\Alg$ is strictly $\J$-monadic over $\C$, i.e., there is a $\J$-ary $\V$-monad $\T_\cP$ such that $\T_\cP\Alg \cong \cP\Alg$ in $\V\CAT/\C$. Explicitly, $\T_\cP = \T_\cS/\sfE$.\qed
\end{theo}

\begin{defn}
Given a free-form or diagrammatic $\J$-presentation $\cP$, we call the $\V$-monad $\T_\cP = \T_\cS/\sf{E}$ in \ref{vmnd_pres_by_ff_jpres} the \textbf{$\J$-ary $\V$-monad presented by $\cP$}. \qed
\end{defn}

\begin{defn}
\label{Xparam_Palg}
Let $\cP = (\cS, \sfE)$ be a diagrammatic $\J$-presentation and $\X$ a $\V$-category.  An \textbf{$\X$-parametrized $\cP$-algebra} is an $\X$-parametrized $\cS$-algebra $A : \X \to \C$ whose corresponding $\V$-functor $\underline{A}:\X \rightarrow \cS\Alg$ (\ref{universal_property_SAlg}) factors through $\cP\Alg \hookrightarrow \cS\Alg$ (equivalently, has the property that the $\cS$-algebra $\underline{A}X$ is a $\cP$-algebra for each object $X$ of $\X$).  In particular, by \ref{universal_property_SAlg}, the $\V$-functor $U^\cP : \cP\Alg \to \C$ \eqref{algs_for_free_form_pres} carries the structure of a $\cP\Alg$-parametrized $\cP$-algebra that has the following universal property, which characterizes $\cP\Alg$ uniquely up to isomorphism:
\end{defn}

\begin{prop}
\label{universal_property_PAlg}
Let $\cP$ be a diagrammatic $\J$-presentation. Then $\cP\Alg$ is a representing object for the functor $\cP\Alg[-] : \V\CAT^\op \to \SET$ that sends each $\V$-category $\X$ to the (large) set $\cP\Alg[\X]$ of all $\X$-parametrized $\cP$-algebras.  The counit of this representation is $U^\cP : \cP\Alg \to \C$.  Thus, there is a bijective correspondence, natural in $\X \in \V\CAT$, between $\V$-functors $\X \to \cP\Alg$ and $\X$-parametrized $\cP$-algebras. \qed 
\end{prop}

\begin{prop}\label{trad_pres_det_ff_pres}
Let $\A$ be a strictly $\J$-monadic $\V$-category over $\C$.  Then $\A \cong \cP\Alg$ in $\V\CAT/\C$ for some diagrammatic (or, equivalently, free-form) $\J$-presentation $\cP$.
\end{prop}
\begin{proof}
By \ref{presentations}, there is a traditional $\J$-presentation $P = (\Sigma,E)$ with $\A \cong P\Alg$ in $\V\CAT/\C$.  Here, $E$ is a traditional system of $\J$-ary equations over the $\V$-monad $\T_\Sigma$ associated to the strictly $\J$-monadic $\V$-category $\Sigma\Alg$ over $\C$ (\ref{signatures}, \ref{presentations}).  By \ref{sys_jary_eqs_as_ff}, $E$ may be regarded as a system of parametrized $\J$-ary equations $\sfE$ over $\T_\Sigma$ and so, equivalently, as a system of diagrammatic $\J$-ary equations on $\Sigma\Alg$ (\ref{bij_diag_ff}).  But $\Sigma\Alg = \cS_\Sigma\Alg$ as $\V$-categories over $\C$, by \ref{trads_det_ffs}, and the result follows by taking $\cP = (\cS_\Sigma,\sfE)$.
\end{proof}

\begin{defn}
\label{Jaryvariety}
{
A $\V$-category over $\C$ is a \textbf{$\J$-ary variety} if it is of the form $\cP\Alg$ for some diagrammatic (or free-form) $\J$-presentation $\cP$. The (ordinary) category of $\J$-ary varieties $\Var_{\underJ}(\C)$ is the full subcategory of $\V\CAT/\C$ whose objects are the $\J$-ary varieties. \qed 
}
\end{defn}

\begin{theo}
\label{varietyequivalence}
The category $\Var_{\underJ}(\C)$ of $\J$-ary varieties is dually equivalent to the category of $\J$-ary $\V$-monads on $\C$:
\[ \Var_{\underJ}(\C) \simeq \Mnd_{\underJ}(\C)^\op. \]
\end{theo} 

\begin{proof}
By \ref{vmnd_pres_by_ff_jpres} and \ref{trad_pres_det_ff_pres}, the repletion of $\Var_{\underJ}(\C)$ in $\V\CAT/\C$ is the full subcategory $\J\text{-}\Monadic^!$ of $\V\CAT/\C$ consisting of the strictly $\J$-monadic $\V$-categories over $\C$ (\ref{Jary}).  Hence $\Var_{\underJ}(\C) \simeq \J\text{-}\Monadic^!$, but $\J\text{-}\Monadic^! \simeq \Mnd_{\underJ}(\C)^\op$ by \ref{Jary}.   
\end{proof}

\begin{rmk}
\label{varietyequivalencermk}
Theorem \ref{varietyequivalence} generalizes certain results in the literature. If we take $\V = \C = \Set$ and $\J = \SF(\Set)$ (i.e. the finite cardinals), then \ref{varietyequivalence} specializes to the classical result that ordinary varieties (in the sense of universal algebra) are dually equivalent to finitary monads on $\Set$. If we take $\V = \C = \Pos$ (the cartesian closed category of posets and monotone maps) and $\J = \SF(\Pos)$, then \ref{varietyequivalence} specializes to the result \cite[Theorem 4.6]{categoricalviewordered} that ($\SF(\Pos)$-ary) varieties of ordered algebras are dually equivalent to strongly finitary (i.e. $\SF(\Pos)$-ary) enriched monads on $\Pos$. If we take $\V = \Set$ and $\C = \Pos$, and we let $\J = \Pos_f$ be a skeleton of the full subcategory of finite (equivalently, finitely presentable) posets, then \ref{varietyequivalence} specializes to the result \cite[Corollary 4.5]{finitarymonadsposets} that finitary (i.e. $\Pos_f$-ary) varieties of ordered algebras are dually equivalent to finitary (unenriched) monads on $\Pos$. Finally, if we take $\V = \C = \Pos$ and $\J = \Pos_f$, then \ref{varietyequivalence} specializes to the result \cite[Corollary 4.7]{finitarymonadsposets} that finitary (i.e. $\Pos_f$-ary) varieties of \emph{coherent} ordered algebras are dually equivalent to finitary enriched monads on $\Pos$.\footnote{The varieties of \cite{categoricalviewordered, finitarymonadsposets} are technically defined differently than ours, in that the former varieties are defined using syntactic inequations; but, in view of \ref{trads_det_ffs} and \ref{sig_det_by_ffsig}, it follows from \cite[p. 2]{categoricalviewordered} and \cite[Remark 2.7, Example 3.19(9)]{finitarymonadsposets} (cf. also the discussion in \cite[\S 5]{Robinson}) that the two definitions are equivalent.} \qed
\end{rmk}

\section{Examples of diagrammatic $\J$-presentations and $\J$-ary varieties}
\label{presentationexamples}

In the present section, we develop several examples of diagrammatic $\J$-presentations and $\J$-ary varieties. We very often (implicitly) use the fact (\ref{signature_operation_natural}) that every operation symbol of a free-form $\J$-signature $\cS$ canonically induces a $\V$-natural $\cS$-operation.

In our first two examples, we employ the system of arities $\SF(\V) \hookrightarrow \V$ consisting of the finite copowers of the terminal object in a complete and cocomplete cartesian closed category $\V$ (\ref{runningexamples}, \ref{sys_ar}). For each finite cardinal $n \in \N$, we write simply $n$ to denote the $n$th copower $n \cdot 1$ of the terminal object $1$ of $\V$, so that for each object $A$ of $\V$, the internal hom $\V(n, A)$ may be identified with the (conical) $n$th power $A^n$. Hence, a parametrized operation on $A$ with arity $n$ and parameter $C \in \ob\V$ is a morphism $\omega:C \times A^n \rightarrow A$ in $\V$, and in the case where $C = 1$ we recover the notion of $n$-ary operation $\omega:A^n \rightarrow A$ on $A$ in the usual sense.

\begin{egg}
\label{Rmodules}
{
Let $\left(R, +^R, \cdot^R, 0^R, 1^R\right)$ be an \emph{internal rig} (or \textit{internal unital semiring}) in a complete and cocomplete cartesian closed category $\V$ (see e.g. \cite[2.7]{functional}). We now provide a diagrammatic $\SF(\V)$-presentation whose algebras will be the \emph{(left) $R$-modules} in $\V$. Recall (see \cite[2.7]{functional}) that an $R$\emph{-module} in $\V$ is an object $M$ of $\V$ equipped with morphisms $+^M : M \times M \to M$, $0^M : 1 \to M$ and $\bullet^M : R \times M \to M$ such that $\left(M, +^M, 0^M\right)$ is a commutative monoid in $\V$, $\bullet^M$ is an associative, unital action of the monoid $\left(R, \cdot^R, 1^R\right)$ on $M$, and $\bullet^M$ is a \emph{bimorphism of commutative monoids} from $\left(R, +^R, 0^R\right), \left(M, +^M, 0^M\right)$ to $\left(M, +^M, 0^M\right)$. Explicitly, we require the satisfaction of the following equations, where we use the convenient notation employed in \cite[2.6]{functional} for expressing equations involving algebraic operations in categories with finite products; in particular, given objects $X_1,...,X_n$ of $\V$, we write expressions of the form $(x_1,x_2,...,x_n):X_1 \times X_2 \times ... \times X_n$ to mean that $x_i$ denotes the $i$th projection morphism $x_i = \pi_i:X_1 \times X_2 \times ... \times X_n  \rightarrow X_i$ for each $i \in \{1,...,n\}$.
\begin{enumerate}[leftmargin=*]
\item $m_1 +^M \left(m_2 +^M m_3\right) = \left(m_1 +^M m_2\right) +^M m_3 : M \times M \times M \to M$, where $(m_1, m_2, m_3) : M \times M \times M$; 
\item $m +^M n = n +^M m : M \times M \to M$, where $(m, n) : M \times M$;
\item $m +^M 0^M = m : M \to M$ and $0^M +^M m = m : M \to M$, where $m : M$;
\item $1^R \bullet^M m = m : M \to M$, where $m : M$;
\item $r \bullet^M \left(s \bullet^M m\right) = \left(r \cdot^R s\right) \bullet^M m : R \times R \times M \to M$, where $(r, s, m) : R \times R \times M$;
\item $0^R \bullet^M m = 0^M : M \to M$ and $r \bullet^M 0^M = 0^M : R \to M$, where $m : M$ and $r : R$;
\item $\left(r +^R s\right) \bullet^M m = \left(r \bullet^M m\right) +^M \left(s \bullet^M m\right) : R \times R \times M \to M$, where $(r, s, m) : R \times R \times M$; and
\item $r \bullet^M \left(m +^M n\right) = \left(r \bullet^M m\right) +^M \left(r \bullet^M n\right) : R \times M \times M \to M$, where $(r, m, n) : R \times M \times M$. 
\end{enumerate}
For example, equation (8) expresses the commutativity of the following diagram
% https://q.uiver.app/?q=WzAsNixbMCwwLCJSIFxcdGltZXMgTV4yIl0sWzIsMCwiUl4yIFxcdGltZXMgTV4yIl0sWzMsMCwiKFIgXFx0aW1lcyBNKSBcXHRpbWVzIChSIFxcdGltZXMgTSkiXSxbNSwwLCJNXjIiXSxbNSwxLCJNIl0sWzAsMSwiUiBcXHRpbWVzIE0iXSxbMCwxLCJcXERlbHRhX1IgXFx0aW1lcyAxIl0sWzEsMiwiXFxzaW0iXSxbMiwzLCJcXGJ1bGxldF5NIFxcdGltZXMgXFxidWxsZXReTSJdLFszLDQsIiteTSJdLFswLDUsIjEgXFx0aW1lcyArXk0iLDJdLFs1LDQsIlxcYnVsbGV0Xk0iLDJdXQ==
\[\begin{tikzcd}
	{R \times M^2} && {R^2 \times M^2} & {(R \times M) \times (R \times M)} && {M^2} \\
	{R \times M} &&&&& M,
	\arrow["{\Delta_R \times 1}", from=1-1, to=1-3]
	\arrow["\sim", from=1-3, to=1-4]
	\arrow["{\bullet^M \times \bullet^M}", from=1-4, to=1-6]
	\arrow["{+^M}", from=1-6, to=2-6]
	\arrow["{1 \times +^M}"', from=1-1, to=2-1]
	\arrow["{\bullet^M}"', from=2-1, to=2-6]
\end{tikzcd}\]
where $\Delta_R$ is the diagonal and the isomorphism $\sim$ exchanges the middle two factors. 

Our free-form $\SF(\V)$-signature $\cS$ for $R$-modules will thus have one operation symbol $+$ of arity $2$ and parameter $1$, one operation symbol $0$ of arity $0$ and parameter $1$, and one operation symbol $\bullet$ of arity $1$ and parameter $R$. As in the example of equation (8) above, each of the listed equations describes a diagrammatic $\cS$-equation $C \times M^n \rightrightarrows M$ ($M \in \cS\Alg$), for some arity $n \in \N$ and parameter $C \in \ob\V$. We thus obtain a diagrammatic $\SF(\V)$-presentation $\cP = (\cS, \sfE)$ such that $\cP$-algebras are precisely $R$-modules in $\V$, and we may therefore regard $R\Mod := \cP\Alg$ as the \emph{$\V$-category of (left) $R$-modules in $\V$}. By \ref{vmnd_pres_by_ff_jpres}, this diagrammatic $\SF(\V)$-presentation $\cP$ presents a strongly finitary $\V$-monad $\T_{\cP}$ on $\V$ with $\T_\cP\Alg \cong R\Mod$ in $\V\CAT/\V$. \qed
}
\end{egg} 

\begin{egg}
\label{monoidalcategories}
Letting $\V$ be the complete and cocomplete cartesian closed category $\mathsf{Cat}$ of small categories, we now discuss a presentation of small monoidal categories in terms of the concepts of \S \ref{sec:sys_eqns} and \S \ref{freeform}, relative to the subcategory of arities $\SF(\mathsf{Cat}) \hookrightarrow \mathsf{Cat}$.  A small monoidal category is precisely a pseudomonoid in the monoidal $2$-category $\mathsf{Cat}$ (see \cite[2.4]{McCrudden}), i.e., a small category $\mathbb{A}$ equipped with the following data, where we write  $\tilde{\mathbf{2}}$ to denote the category with exactly two objects and an isomorphism between these objects:
\begin{itemize}[leftmargin=*]
\item Functors $m : \mathbb{A}^2 \to \mathbb{A}$ and $i:1 \to \mathbb{A}$ (where $1$ denotes the terminal category);

\item A natural isomorphism % https://q.uiver.app/?q=WzAsMixbMCwwLCJcXGJiQ14zIl0sWzIsMCwiXFxiYkMiXSxbMCwxLCJcXGFscGhhXzEiLDIseyJvZmZzZXQiOjN9XSxbMCwxLCJcXGFscGhhXzAiLDAseyJvZmZzZXQiOi0zfV0sWzMsMiwiXFxhbHBoYSIsMix7InNob3J0ZW4iOnsic291cmNlIjoyMCwidGFyZ2V0IjoyMH19XV0=
\begin{tikzcd}
	{\mathbb{A}^3} && \mathbb{A}
	\arrow[""{name=0, anchor=center, inner sep=0}, "{\alpha_1}"', shift right=3, from=1-1, to=1-3]
	\arrow[""{name=1, anchor=center, inner sep=0}, "{\alpha_0}", shift left=3, from=1-1, to=1-3]
	\arrow["\alpha"', shorten <=2pt, shorten >=2pt, Rightarrow, from=1, to=0]
\end{tikzcd} (the \emph{associator}) between the functors $\alpha_0 = m \circ (m \times 1)$ and $\alpha_1 = m \circ (1 \times m)$; 

\item A natural isomorphism
\begin{tikzcd}
	{\mathbb{A}} && \mathbb{A}
	\arrow[""{name=0, anchor=center, inner sep=0}, "{\lambda_1}"', shift right=3, from=1-1, to=1-3]
	\arrow[""{name=1, anchor=center, inner sep=0}, "{\lambda_0}", shift left=3, from=1-1, to=1-3]
	\arrow["\lambda"', shorten <=2pt, shorten >=2pt, Rightarrow, from=1, to=0]
\end{tikzcd}
(the \emph{left unitor}) between functors $\lambda_0, \lambda_1 : \mathbb{A} \rightrightarrows \mathbb{A}$, where $\lambda_0$ is the composite $\mathbb{A} \xrightarrow{\sim} 1 \times \mathbb{A} \xrightarrow{i \times 1} \mathbb{A} \times \mathbb{A} \xrightarrow{m} \mathbb{A}$ and $\lambda_1 = 1_\mathbb{A} : \mathbb{A} \to \mathbb{A}$;

\item A natural isomorphism
\begin{tikzcd}
	{\mathbb{A}} && \mathbb{A}
	\arrow[""{name=0, anchor=center, inner sep=0}, "{\rho_1}"', shift right=3, from=1-1, to=1-3]
	\arrow[""{name=1, anchor=center, inner sep=0}, "{\rho_0}", shift left=3, from=1-1, to=1-3]
	\arrow["\rho"', shorten <=2pt, shorten >=2pt, Rightarrow, from=1, to=0]
\end{tikzcd}
(the \emph{right unitor}) between functors $\rho_0, \rho_1 : \mathbb{A} \rightrightarrows \mathbb{A}$, where $\rho_0 = 1_\mathbb{A} : \mathbb{A} \to \mathbb{A}$ and $\rho_1$ is the composite $\mathbb{A} \xrightarrow{\sim} \mathbb{A} \times 1 \xrightarrow{1 \times i} \mathbb{A} \times \mathbb{A} \xrightarrow{m} \mathbb{A}$;
\end{itemize}
such that the following monoidal coherence laws hold:
\begin{itemize}[leftmargin=*]
\item The \textit{pentagon identity} requires that the composite natural isomorphism
\[ m(m \times 1)(m \times 1 \times 1) \xrightarrow{\alpha \circ 1} m(1 \times m)(m \times 1 \times 1) = m(m \times 1)(1 \times 1 \times m) \xrightarrow{\alpha \circ 1} m(1 \times m)(1 \times 1 \times m) \] between functors $\mathbb{A}^4 \rightrightarrows \mathbb{A}$ be equal to the composite natural isomorphism
\[ m(m \times 1)(m \times 1 \times 1) \xrightarrow{1 \circ (\alpha \times 1)} m(m \times 1)(1 \times m \times 1) \xrightarrow{\alpha \circ 1} m(1 \times m)(1 \times m \times 1) \xrightarrow{1 \circ (1 \times \alpha)} m(1 \times m)(1 \times 1 \times m). \] 
\item The following automorphism of $m:\mathbb{A}^2 \rightarrow \mathbb{A}$ is required to be equal to the identity on $m$:
\[ m = m \circ 1 \xrightarrow{1 \circ (\rho \times 1)} m(m \times 1)(1 \times i \times 1) \xrightarrow{\alpha \circ 1} m(1 \times m)(1 \times i \times 1) \xrightarrow{1 \circ (1 \times \lambda)} m \circ 1 = m\;.\]
\end{itemize}
The associator $\alpha$ clearly corresponds to a functor $\tilde{\mathbf{2}} \to \mathbb{A}^{\mathbb{A}^3}$, and hence to a parametrized operation $\mathbb{A}^3 \times \tilde{\mathbf{2}} \to \mathbb{A}$ with arity $3$ and parameter $\tilde{\mathbf{2}}$; and the left and right unitors $\lambda, \rho$ clearly correspond to functors $\tilde{\mathbf{2}} \to \mathbb{A}^{\mathbb{A}}$, and hence to parametrized operations $\mathbb{A} \times \tilde{\mathbf{2}} \to \mathbb{A}$ with arity $1$ and parameter $\tilde{\mathbf{2}}$. 

We thus define a free-form $\SF(\mathsf{Cat})$-signature $\cS$ with one operation symbol $m$ of arity $2$ and parameter $1$, one operation symbol $i$ of arity $0$ and parameter $1$, one operation symbol $\alpha$ of arity $3$ and parameter $\tilde{\mathbf{2}}$, and two operation symbols $\lambda, \rho$ of arity $1$ and parameter $\tilde{\mathbf{2}}$. As an intermediate step, we next provide a diagrammatic $\SF(\mathsf{Cat})$-presentation $\cP = (\cS, \sfE)$ for which a $\cP$-algebra will be a small category $\mathbb{A}$ equipped with functors $m^\mathbb{A} : \mathbb{A}^2 \to \mathbb{A}$, $i^\mathbb{A} : \mathbb{A}^0 = 1 \to \mathbb{A}$ and natural isomorphisms $\alpha^\mathbb{A} : \alpha_0^\mathbb{A} \xrightarrow{\sim} \alpha_1^\mathbb{A}$, $\lambda^\mathbb{A} : \lambda_0^\mathbb{A} \xrightarrow{\sim} \lambda_1^\mathbb{A}$, and $\rho^\mathbb{A} : \rho_0^\mathbb{A} \xrightarrow{\sim} \rho_1^\mathbb{A}$ with the intended domains and codomains (but not necessarily satisfying the monoidal coherence laws). So let $\mathbb{A}$ be an $\cS$-algebra. We want diagrammatic $\cS$-equations expressing that the domain and codomain $\alpha_0^\mathbb{A}, \alpha_1^\mathbb{A} : \mathbb{A}^3 \rightrightarrows \mathbb{A}$ of the associator $\alpha^\mathbb{A}$ are the composite functors in the following diagram:
% https://q.uiver.app/?q=WzAsNCxbMCwwLCJcXGJiQyBcXHRpbWVzIFxcYmJDIFxcdGltZXMgXFxiYkMiXSxbMiwwLCJcXGJiQyBcXHRpbWVzIFxcYmJDIl0sWzIsMSwiXFxiYkMiXSxbNCwwLCJcXGJiQyBcXHRpbWVzIFxcYmJDIFxcdGltZXMgXFxiYkMiXSxbMCwxLCJtIFxcdGltZXMgMSJdLFsxLDIsIm0iXSxbMCwyLCJcXGFscGhhXzAiLDJdLFszLDEsIjEgXFx0aW1lcyBtIiwyXSxbMywyLCJcXGFscGhhXzEiXV0=
\[\begin{tikzcd}
	{\mathbb{A} \times \mathbb{A} \times \mathbb{A}} && {\mathbb{A} \times \mathbb{A}} && {\mathbb{A} \times \mathbb{A} \times \mathbb{A}} \\
	&& \mathbb{A}
	\arrow["{m^\mathbb{A} \times 1}", from=1-1, to=1-3]
	\arrow["m^\mathbb{A}", from=1-3, to=2-3]
	\arrow["{\alpha_0^\mathbb{A}}"', from=1-1, to=2-3]
	\arrow["{1 \times m^\mathbb{A}}"', from=1-5, to=1-3]
	\arrow["{\alpha_1^\mathbb{A}}", from=1-5, to=2-3]
\end{tikzcd}\] 
In view of \ref{signature_operation_natural}, the two triangles in this diagram constitute diagrammatic $\cS$-equations $m^\mathbb{A} \circ \left(m^\mathbb{A} \times 1\right) \doteq \alpha_0^\mathbb{A},\; m^\mathbb{A} \circ \left(1 \times m^\mathbb{A}\right) \doteq  \alpha_1^\mathbb{A}:\mathbb{A}^3 \rightrightarrows \mathbb{A}^1$ ($\mathbb{A} \in \cS\Alg$). In an exactly analogous way, we obtain diagrammatic $\cS$-equations to ensure that $\lambda_0^\mathbb{A}, \lambda_1^\mathbb{A}, \rho_0^\mathbb{A}, \rho_1^\mathbb{A}$ all have their intended interpretations. We thus obtain a diagrammatic $\SF(\mathsf{Cat})$-presentation $\cP = (\cS, \sfE)$ for which a $\cP$-algebra is a small category $\mathbb{A}$ equipped with functors $m^\mathbb{A} : \mathbb{A}^2 \to \mathbb{A}$, $i^\mathbb{A} : \mathbb{A}^0 \to \mathbb{A}$ and natural isomorphisms $\alpha^\mathbb{A} : m^\mathbb{A} \circ \left(m^\mathbb{A} \times 1\right) \xrightarrow{\sim} m^\mathbb{A} \circ \left(1 \times m^\mathbb{A}\right)$, $\lambda^\mathbb{A} : m^\mathbb{A} \circ \left(i^\mathbb{A} \times 1\right) \xrightarrow{\sim} 1$, and $\rho^\mathbb{A} : 1 \xrightarrow{\sim} m^\mathbb{A} \circ \left(1 \times i^\mathbb{A}\right)$.

We now define a further system of diagrammatic $\SF(\mathsf{Cat})$-ary equations $\sfF$ on $\cP\Alg$ whose models \eqref{diagrammatic_Jary_eqn_on_g} will be monoidal categories. Given a $\cP$-algebra $\mathbb{A}$, we want $\mathbb{A}$ to be an $\sfF$-model iff $\mathbb{A}$ is a monoidal category, i.e. iff $\mathbb{A}$ satisfies the monoidal coherence laws. The pentagon identity requires the equality of two natural isomorphisms between functors $\mathbb{A}^4 \rightrightarrows \mathbb{A}$, and it can be expressed by a diagrammatic $\SF(\mathsf{Cat})$-ary equation $\mathbb{A}^4 \rightrightarrows \mathbb{A}^{\tilde{\mathbf{2}}}$ ($\mathbb{A} \in \cP\Alg$) on $\cP\Alg$. The second monoidal coherence law requires the equality of two natural isomorphisms between functors $\mathbb{A}^2 \rightrightarrows \mathbb{A}$, and it can be expressed by a diagrammatic $\SF(\mathsf{Cat})$-ary equation $\mathbb{A}^2 \rightrightarrows \mathbb{A}^{\tilde{\mathbf{2}}}$ ($\mathbb{A} \in \cP\Alg$) on $\cP\Alg$. We thus obtain a system of diagrammatic $\SF(\mathsf{Cat})$-ary equations $\mathsf{F}$ on $\cP\Alg$ for which $\cP\Alg_{\mathsf{F}} = \mathsf{MonCat}_{\mathsf{strict}}$, the 2-category of small monoidal categories and strict monoidal functors.  Hence, $\mathsf{MonCat}_{\mathsf{strict}}$ is a strictly $\SF(\mathsf{Cat})$-monadic $2$-category over $\mathsf{Cat}$, by \ref{Jary_eqn_monadic_prop} and \ref{vmnd_pres_by_ff_jpres}, so $\mathsf{MonCat}_{\mathsf{strict}}$ is isomorphic to the $2$-category of algebras for a strongly finitary $2$-monad on $\mathsf{Cat}$, as noted in \cite{KellyLackstronglyfinitary}. \qed 
\end{egg}

\begin{egg}
\label{internalcategoriesexample}
{
We now employ diagrammatic $\J$-presentations to prove an extension of the well-known result that (small) categories are finitary monadic over graphs \cite{Burroni}.  \highl{Indeed, whereas Wolff} \cite{WolffVcat} \highl{extended the latter result by proving the monadicity of small $\V$-categories over small $\V$-graphs under only the assumption that $\V$ is cocomplete, we now prove instead that \textit{internal categories} in $\V$ form a $\V$-category that is $\alpha$-ary monadic over \textit{internal graphs} in $\V$ as soon as $\V$ is locally $\alpha$-presentable as a closed category, which we assume for the remainder of Example} \ref{internalcategoriesexample}. \highl{To this end,} let $\B_\rightrightarrows$ be the free $\V$-category on the ordinary category consisting of a single parallel pair $s, t : 1 \rightrightarrows 0$, and consider the presheaf $\V$-category $\C = \Gph(\V) = \left[\B_\rightrightarrows, \V\right]$, the $\V$-category of \emph{internal graphs in} $\V$, which is a locally $\alpha$-presentable $\V$-category by \cite[3.1, 7.4]{Kellystr}. An object $G$ of $\C$ can be regarded as a parallel pair $s_G, t_G : G_1 \rightrightarrows G_0$ in $\V$, i.e. as an internal graph in $\V$. Let $\J = \C_\alpha$ be a skeleton of the full sub-$\V$-category of $\C$ spanned by the enriched $\alpha$-presentable objects, so that $\J$ is a bounded and eleutheric subcategory of arities by \ref{runningexamples}. In particular, $\J$ contains the representables $[i] := \B_\rightrightarrows(i, -) : \B_\rightrightarrows \to \V$ ($i \in \{0, 1\}$) and is closed under conical finite colimits. In the $\V = \Set$ case, the (ordinary) graph $[0]$ consists of a single vertex and no edges, while the graph $[1] = (\,\cdot \rightarrow \cdot\,)$ consists of a single edge between two distinct vertices. We have the two morphisms $\iota_1 := \B_\rightrightarrows(s, -) : [0] \to [1]$ and $\iota_2 := \B_\rightrightarrows(t, -) : [0] \to [1]$; in the $\V = \Set$ case, $\iota_1$ sends the single vertex of $[0]$ to the source of the unique edge of $[1]$, while $\iota_2$ sends the single vertex of $[0]$ to the target of this edge. For every internal graph $G : \B_\rightrightarrows \to \V$, the enriched Yoneda lemma enables us to identify $\C([i], G)$ with $G_i$ ($i \in \{0, 1\}$), so that we may identify $\C(\iota_1, G)$ with $s_G$ and $\C(\iota_2, G)$ with $t_G$. We shall also need the arity $[2] \in \ob \C_\alpha$ defined via the following pushout in $\C = \Gph(\V)$, which in the $\V = \Set$ case is the graph $[2] = (\,\cdot \rightarrow \cdot \rightarrow \cdot\,)$.
% https://q.uiver.app/?q=WzAsNCxbMCwwLCJcXEFfXFxyaWdodHJpZ2h0YXJyb3dzKDAsIC0pIl0sWzIsMCwiXFxBX1xccmlnaHRyaWdodGFycm93cygxLCAtKSJdLFswLDIsIlxcQV9cXHJpZ2h0cmlnaHRhcnJvd3MoMSwgLSkiXSxbMiwyLCJbMl0iXSxbMCwxLCJcXEFfXFxyaWdodHJpZ2h0YXJyb3dzKHMsIC0pIl0sWzAsMiwiXFxBX1xccmlnaHRyaWdodGFycm93cyh0LCAtKSIsMl0sWzIsMywicV8xIiwyXSxbMSwzLCJxXzIiXSxbMywwLCIiLDEseyJzdHlsZSI6eyJuYW1lIjoiY29ybmVyIn19XV0=
\[\begin{tikzcd}
	{[0]} && {[1]} \\
	\\
	{[1]} && {[2]}.
	\arrow["{\iota_1}", from=1-1, to=1-3]
	\arrow["{\iota_2}"', from=1-1, to=3-1]
	\arrow["{k_1}"', from=3-1, to=3-3]
	\arrow["{k_2}", from=1-3, to=3-3]
\end{tikzcd}\]
For every internal graph $G$ in $\V$ we have a pullback square in $\V$ as on the left below, 
% https://q.uiver.app/?q=WzAsOCxbMCwwLCJcXEMoWzJdLCBHKSJdLFsyLDAsIlxcQyhbMV0sIEcpIl0sWzIsMiwiXFxDKFswXSwgRykiXSxbMCwyLCJcXEMoWzFdLCBHKSJdLFs0LDAsIkdfMSBcXHRpbWVzX3tHXzB9IEdfMSJdLFs2LDAsIkdfMSJdLFs0LDIsIkdfMSJdLFs2LDIsIkdfMCJdLFswLDEsIlxcQyhxXzIsIEcpIl0sWzEsMiwiXFxDKFxcaW90YV8xLCBHKSJdLFswLDMsIlxcQyhxXzEsIEcpIiwyXSxbMywyLCJcXEMoXFxpb3RhXzIsIEcpIiwyXSxbNCw1LCJwXzJeRyJdLFs0LDYsInBfMV5HIiwyXSxbNSw3LCJzXkciXSxbNiw3LCJ0XkciLDJdXQ==
\[\begin{tikzcd}
	{\C([2], G)} && {\C([1], G)} && {G_2 := G_1 \times_{G_0} G_1} && {G_1} \\
	\\
	{\C([1], G)} && {\C([0], G)} && {G_1} && {G_0}
	\arrow["{\C(k_2, G)}", from=1-1, to=1-3]
	\arrow["{\C(\iota_1, G)}", from=1-3, to=3-3]
	\arrow["{\C(k_1, G)}"', from=1-1, to=3-1]
	\arrow["{\C(\iota_2, G)}"', from=3-1, to=3-3]
	\arrow["{\pi_2}", from=1-5, to=1-7]
	\arrow["{\pi_1}"', from=1-5, to=3-5]
	\arrow["{s_G}", from=1-7, to=3-7]
	\arrow["{t_G}"', from=3-5, to=3-7]
\end{tikzcd}\] 
which we may identify with the pullback square on the right.        

We now provide a presentation of  \emph{internal categories} in $\V$ in terms of the concepts of \S \ref{sec:sys_eqns} and \S \ref{freeform}; for a standard reference on internal categories, see e.g. \cite[\S 8.1]{Borceux1}. An internal category in $\V$ is an internal graph $A = \left(s_A, t_A : A_1 \rightrightarrows A_0\right)$ in $\V$ equipped with morphisms $e^A : A_0 \to A_1$ and $c^A : A_2 \to A_1$ (\textit{identity} and \textit{composition}) satisfying certain axioms.  The identity and composition operations $e^A$ and $c^A$ may be written equivalently as $e^A : \C([0], A) \to \C([1], A)$ and $c^A : \C([2], A) \to \C([1], A)$ and so are parametrized $\J$-ary operations on the object $A$ of $\C = \Gph(\V)$, where $\J = \ob\C_\alpha$. We thus require a free-form $\J$-signature $\cS$ with one operation symbol $e$ of arity $[0]$ and parameter $[1]$, and one operation symbol $c$ of arity $[2]$ and parameter $[1]$. We now construct a diagrammatic $\J$-presentation $\cP = (\cS, \sfE)$ that will play a preliminary role, in which $\sfE$ consists of four diagrammatic $\cS$-equations, expressing the interaction of the source and target morphisms with $e^A, c^A$ in an $\cS$-algebra $A$. In particular, we shall have a diagrammatic $\cS$-equation $\C\left([2], A\right) \rightrightarrows \C\left([0], A\right)$ ($A \in \cS\Alg$) given by the diagram on the left below, which may be written as the more familiar diagram on the right
\[\begin{tikzcd}
	{\C([2], A)} && {\C([1], A)} && A_2 && {A_1} \\
	\\
	{\C([1], A)} && {\C([0], A)} && {A_1} && {A_0},
	\arrow["{c^A}", from=1-1, to=1-3]
	\arrow["{s_A}", from=1-3, to=3-3]
	\arrow["{\C(k_1, A)}"', from=1-1, to=3-1]
	\arrow["{s_A}"', from=3-1, to=3-3]
	\arrow["{c^A}", from=1-5, to=1-7]
	\arrow["{\pi_1}"', from=1-5, to=3-5]
	\arrow["{s_A}", from=1-7, to=3-7]
	\arrow["{s_A}"', from=3-5, to=3-7]
\end{tikzcd}\] 
and we shall also have a diagrammatic $\cS$-equation $t_A \circ c^A \doteq t_A \circ \pi_2:\C\left([2], A\right) \rightrightarrows \C\left([0], A\right)$ ($A \in \cS\Alg$). We shall also have two diagrammatic $\cS$-equations $s_A \circ e^A \doteq 1_{A_0},\;t_A \circ e^A \doteq 1_{A_0}:\C\left([0], A\right) \rightrightarrows \C\left([0], A\right)$ ($A \in \cS\Alg$). We thus obtain a diagrammatic $\J$-presentation $\cP = (\cS, \sfE)$ for which $\cP$-algebras will be internal graphs $A = \left(s_A, t_A : A_1 \rightrightarrows A_0\right)$ in $\V$
equipped with identity and composition operations $e^A, c^A$ satisfying (just) these four diagrammatic laws. 

We now define a system of diagrammatic $\J$-ary equations $\mathsf{F}$ on $\cP\Alg$ whose models \eqref{diagrammatic_Jary_eqn_on_g} will be internal categories in $\V$. First, given an internal graph $G$, we need to define an \textit{object of paths of length three} in $G$. This object, which we write as $G_3$, can be defined as the (vertex of the) pullback of $s_G \circ \pi_1 : G_2 \to G_0$ along $t_G : G_1 \to G_0$, or as the pullback of $s_G : G_1 \to G_0$ along $t_G \circ \pi_2 : G_2 \to G_0$, since it is well known that these pullbacks $G_1 \times_{G_0} \left(G_1 \times_{G_0} G_1\right)$ and $\left(G_1 \times_{G_0} G_1\right) \times_{G_0} G_1$ are canonically isomorphic. By a dual process, we can define a further arity $[3] \in \ob\C_\alpha$ by taking the pushout of $k_1 \circ \iota_1 : [0] \to [2]$ along $\iota_2 : [0] \to [1]$, or equivalently the pushout of $\iota_1 : [0] \to [1]$ along $k_2 \circ \iota_2 : [0] \to [2]$. In the $\V = \Set$ case, $[3] = (\,\cdot \rightarrow \cdot \rightarrow \cdot \rightarrow \cdot\,)$. For any internal graph $G$ in $\V$, we may identify $\C([3], G)$ with $G_3 \cong G_1 \times_{G_0} \left(G_1 \times_{G_0} G_1\right) \cong \left(G_1 \times_{G_0} G_1\right) \times_{G_0} G_1$.       

Given a $\cP$-algebra $A$, we want $A$ to be an $\mathsf{F}$-model iff $A$ is an internal category in $\V$, i.e. iff composition in $A$ is associative and unital. The associativity of composition is expressed by the following diagram
% https://q.uiver.app/?q=WzAsNixbMCwwLCJcXEMoWzNdLCBBKSJdLFsyLDAsIkFfMSBcXHRpbWVzX3tBXzB9IEFfMSBcXHRpbWVzX3tBXzB9IEFfMSJdLFs0LDAsIkFfMSBcXHRpbWVzX3tBXzB9IEFfMSJdLFswLDEsIkFfMSBcXHRpbWVzX3tBXzB9IEFfMSBcXHRpbWVzX3tBXzB9IEFfMSJdLFsyLDEsIkFfMSBcXHRpbWVzX3tBXzB9IEFfMSJdLFs0LDEsIkFfMSJdLFswLDEsIlxcc2ltIl0sWzEsMiwiMSBcXHRpbWVzX3tBXzB9IGNeQSJdLFswLDMsIlxcd3IiLDJdLFszLDQsImNeQSBcXHRpbWVzX3tBXzB9IDEiLDJdLFs0LDUsImNeQSIsMl0sWzIsNSwiY15BIl1d
\[\begin{tikzcd}
	A_3 && {A_1 \times_{A_0} \left(A_1 \times_{A_0} A_1\right)} && {A_1 \times_{A_0} A_1} \\
	{\left(A_1 \times_{A_0} A_1\right) \times_{A_0} A_1} && {A_1 \times_{A_0} A_1} && {A_1},
	\arrow["\sim", from=1-1, to=1-3]
	\arrow["{1 \times_{A_0} c^A}", from=1-3, to=1-5]
	\arrow["\wr"', from=1-1, to=2-1]
	\arrow["{c^A \times_{A_0} 1}"', from=2-1, to=2-3]
	\arrow["{c^A}"', from=2-3, to=2-5]
	\arrow["{c^A}", from=1-5, to=2-5]
\end{tikzcd}\]
which is a diagrammatic $\J$-ary equation $\C\left([3], A\right) \rightrightarrows \C\left([1], A\right)$ ($A \in \cP\Alg$) on $\cP\Alg$. The identity axioms for an internal category are expressed by the following diagram
      % https://q.uiver.app/?q=WzAsNyxbMCwwLCJcXEMoWzFdLCBHKSJdLFsxLDAsIkdfMSJdLFszLDAsIkdfMSBcXHRpbWVzX3tHXzB9IEdfMSJdLFszLDEsIkdfMSJdLFszLDIsIlxcQyhbMV0sIEcpIl0sWzYsMCwiXFxDKFsxXSwgRykiXSxbNSwwLCJHXzEiXSxbMCwxLCJcXHNpbSJdLFsxLDIsIlxcbGVmdFxcbGFuZ2xlIDEsIGVeRyBcXGNpcmMgc15HXFxyaWdodFxccmFuZ2xlIl0sWzIsMywiYyJdLFszLDQsIlxcd3IiXSxbMCw0LCIxIiwyXSxbNSw2LCJcXHNpbSIsMl0sWzYsMiwiXFxsZWZ0XFxsYW5nbGUgZV5HIFxcY2lyYyB0XkcsIDFcXHJpZ2h0XFxyYW5nbGUiLDJdLFs1LDQsIjEiXV0=
\[\begin{tikzcd}
	& {A_1} && {A_1 \times_{A_0} A_1} && {A_1} & \\
	&&& {A_1}
	\arrow["{( 1, e^A \circ t_A)}", from=1-2, to=1-4]
	\arrow["c^A", from=1-4, to=2-4]
	\arrow["1"', from=1-2, to=2-4]
	\arrow["{(e^A \circ s_A, 1)}"', from=1-6, to=1-4]
	\arrow["1", from=1-6, to=2-4]
\end{tikzcd}\]
which may be regarded as a pair of diagrammatic $\J$-ary equations $\C\left([1], A\right) \rightrightarrows \C\left([1], A\right)$ ($A \in \cP\Alg$) on $\cP\Alg$. We thus obtain a system of diagrammatic $\J$-ary equations $\mathsf{F}$ on $\cP\Alg$ for which $\mathsf{F}$-models are internal categories in $\V$, so that we may regard $\mathsf{Cat}(\V) := \cP\Alg_\mathsf{F}$ as the \emph{$\V$-category of internal categories in $\V$}. The following theorem now follows from \ref{Jary_eqn_monadic_prop} and \ref{vmnd_pres_by_ff_jpres}:  
}
\end{egg}

\begin{theo}
\label{internal_cat_thm}
Let $\V$ be locally $\alpha$-presentable as a closed category. There is an $\alpha$-accessible $\V$-monad $\T$ on $\Gph(\V) = \left[\B_\rightrightarrows, \V\right]$ such that $\T\Alg \cong \mathsf{Cat}(\V)$ in $\V\CAT/\Gph(\V)$, so that the $\V$-category $\mathsf{Cat}(\V)$ of internal categories in $\V$ is strictly $\alpha$-ary monadic over $\Gph(\V)$. \qed
\end{theo}

\begin{egg}
\label{PPexamples}
{
We now provide a diagrammatic $\J$-presentation for Plotkin and Power's \emph{global state} algebras \cite{PlotkinPowernotions}, which they employed in treating the \textit{global state monad} within their program of \textit{algebraic computational effects}. Suppose that $\V$ is cartesian closed, and let $j : \J \hookrightarrow \V$ be any bounded and eleutheric \emph{system} of arities (so that $\J$ contains the terminal object $1$ and is closed under finite products, \ref{sys_ar}). For example, we may take $\J = \SF(\V)$, or if $\V$ is locally $\alpha$-presentable as a cartesian closed category, then we may take $\J = \V_\alpha$, a skeleton of the full sub-$\V$-category of $\V$ consisting of the $\alpha$-presentable objects. Let $V \in \ob\J$ be a fixed object (the object of \textit{values}), let $L$ be a finite set (the set of \textit{locations}), and let $L_2 := \{(\ell, \ell') \in L \times L \mid \ell \neq \ell'\}$. By an abuse of notation, we write $L := L \cdot 1 \in \ob\V$ (the coproduct of $L$ copies of $1$) and $L_2 := L_2 \cdot 1 \in \ob\V$. We define a free-form $\J$-signature $\cS$ with an operation symbol $\ell$ (``lookup'') of arity $V$ and parameter $L$, and an operation symbol $u$ (``update'') of arity $1$ and parameter $L \times V$. An $\cS$-algebra is thus an object $A$ of $\V$ equipped with a \textit{lookup} operation $\ell^A : A^V \to A^L$ and an \textit{update} operation $u^A : A \to A^{L \times V}$, as in \cite[Definition 1]{PlotkinPowernotions}.  

Each of the seven commutative diagrams of \cite[Definition 1]{PlotkinPowernotions} for global state may be regarded as a diagrammatic $\cS$-equation $A^J \rightrightarrows A^C$ ($A \in \cS\Alg$) for some arity $J \in \ob\J$ and parameter $C \in \ob\V$. Specifically, the first commutative diagram can be expressed as a diagrammatic $\cS$-equation $A \rightrightarrows A^L$ ($A \in \cS\Alg$); the second as a diagrammatic $\cS$-equation $A^{V \times V} \rightrightarrows A^L$ ($A \in \cS\Alg$); the third as a diagrammatic $\cS$-equation $A \rightrightarrows A^{L \times V \times V}$ ($A \in \cS\Alg$); the fourth as a diagrammatic $\cS$-equation $A^V \rightrightarrows A^{L \times V}$ ($A \in \cS\Alg$); the fifth as a diagrammatic $\cS$-equation $A^{V \times V} \rightrightarrows A^{L_2}$ ($A \in \cS\Alg$); the sixth as a diagrammatic $\cS$-equation $A \rightrightarrows A^{L_2 \times V \times V}$ ($A \in \cS\Alg$); and the last as a diagrammatic $\cS$-equation $A^V \rightrightarrows A^{L_2 \times V}$ ($A \in \cS\Alg$). We thus obtain a diagrammatic $\J$-presentation $\cP = (\cS, \sfE)$ for which $\cP$-algebras are (the $\V$-based analogue of) the \emph{global state} algebras of \cite[Definition 1]{PlotkinPowernotions}, and by \ref{vmnd_pres_by_ff_jpres} we find that $\cP$ presents a $\J$-ary $\V$-monad $\T_\cP$ on $\V$ with $\T_\cP\Alg \cong \cP\Alg$ in $\V\CAT/\V$. \qed 
} 
\end{egg} 

\begin{egg}
\label{Statonexamples}
{
We now provide diagrammatic $\J$-presentations corresponding to some of Staton's \emph{parametrized algebraic theories} \cite{Statoninstances}, namely the parametrized theory of reading many bits \cite[\S IV]{Statoninstances} and the parametrized theory of restriction \cite[\S V.B]{Statoninstances}, and we thereby obtain $\J$-ary $\V$-monads determined by these parametrized algebraic theories. Suppose that $\V$ is cartesian closed, and let us first define a diagrammatic $\SF(\V)$-presentation for the theory of reading many bits, parametrized by an object $\alpha \in \ob\V$. The free-form $\SF(\V)$-signature $\cS$ will have one operation symbol $?$ of arity $2$ and parameter $\alpha$, so that an $\cS$-algebra is an object $A$ of $\V$ equipped with a morphism $?^A : A^2 \to A^\alpha$, or equivalently $?^A : A \times A \times \alpha \to A$. The desired diagrammatic $\SF(\V)$-presentation $\cP = (\cS, \sfE)$ will have diagrammatic $\cS$-equations corresponding to the equations of \cite{Statoninstances} annotated as (idem-$?_a$), (dup-$?_a$), and $(?_a/?_b)$ (for $a, b : \alpha$). The first equation (idem-$?_a$) is written there as $x \equiv x \ ?_a \ x$ for $x : A$ and $a : \alpha$, and can be expressed as the commutativity of the diagram
% https://q.uiver.app/?q=WzAsMyxbMCwwLCJBIFxcdGltZXMgUyJdLFsyLDAsIkEgXFx0aW1lcyBBIFxcdGltZXMgUyJdLFs0LDAsIkEiXSxbMCwxLCJcXERlbHRhX0EgXFx0aW1lcyAxIl0sWzEsMiwiP15BIl0sWzAsMiwiXFxwaV8xIiwyLHsib2Zmc2V0IjoyLCJjdXJ2ZSI6MX1dXQ==
\[\begin{tikzcd}
	{A \times \alpha} && {A \times A \times \alpha} && A,
	\arrow["{\Delta_A \times 1}", from=1-1, to=1-3]
	\arrow["{?^A}", from=1-3, to=1-5]
	\arrow["{\pi_1}"', shift right=2, curve={height=6pt}, from=1-1, to=1-5]
\end{tikzcd}\]
which is a diagrammatic $\cS$-equation $A \times \alpha \rightrightarrows A$ ($A \in \cS\Alg$). The second equation (dup-$?_a$) is written as $(u \ ?_a \ v) \ ?_a \ (x \ ?_a \ y) \equiv u \ ?_a \ y$ for $u, v, x, y : A$ and $a : \alpha$. Writing the product projections as $u, v, x, y : A^4 \times \alpha \to A$ and $a : A^4 \times \alpha \to \alpha$, this equation can be expressed as the commutativity of the diagram
% https://q.uiver.app/?q=WzAsNSxbMywwXSxbNSwwXSxbMCwwLCJBXjQgXFx0aW1lcyBcXGFscGhhIl0sWzQsMCwiQSBcXHRpbWVzIEEgXFx0aW1lcyBcXGFscGhhIl0sWzYsMCwiQSJdLFsyLDMsIlxcbGVmdFxcbGFuZ2xlID9eQSBcXGNpcmMgXFxsYW5nbGUgdSwgdiwgYVxccmFuZ2xlLCA/XkEgXFxjaXJjIFxcbGFuZ2xlIHgsIHksIGEgXFxyYW5nbGUsIGFcXHJpZ2h0XFxyYW5nbGUiXSxbMiw0LCI/XkEgXFxjaXJjIFxcbGFuZ2xlIHUsIHksIGFcXHJhbmdsZSIsMix7ImN1cnZlIjozfV0sWzMsNCwiP15BIl1d
\[\begin{tikzcd}
	{A^4 \times \alpha} &&& {} & {A \times A \times \alpha} & {} & A,
	\arrow["{\left\langle ?^A \circ \langle u, v, a\rangle, ?^A \circ \langle x, y, a \rangle, a\right\rangle}", from=1-1, to=1-5]
	\arrow["{?^A \circ \langle u, y, a\rangle}"', curve={height=18pt}, from=1-1, to=1-7]
	\arrow["{?^A}", from=1-5, to=1-7]
\end{tikzcd}\]
which is a diagrammatic $\cS$-equation $A^4 \times \alpha \rightrightarrows A$ ($A \in \cS\Alg$). The third equation $(?_a/?_b)$ is written as $(u \ ?_b \ v) \ ?_a \ (x \ ?_b \ y) \equiv (u \ ?_a \ x) \ ?_b \ (v \ ?_a \ y)$ for $u, v, x, y : A$ and $a, b : \alpha$. Writing the product projections as $u, v, x, y : A^4 \times \alpha^2 \to A$ and $a, b : A^4 \times \alpha^2 \to \alpha$, this equation can be expressed as the commutativity of the diagram
% https://q.uiver.app/?q=WzAsNixbMCwwLCJBXjQgXFx0aW1lcyBTXjIiXSxbMywwLCIoQV4yIFxcdGltZXMgUykgXFx0aW1lcyAoQV4yIFxcdGltZXMgUykgXFx0aW1lcyBTIl0sWzUsMCwiQSBcXHRpbWVzIEEgXFx0aW1lcyBTIl0sWzUsMSwiQSJdLFswLDEsIihBXjIgXFx0aW1lcyBTKSBcXHRpbWVzIChBXjIgXFx0aW1lcyBTKSBcXHRpbWVzIFMiXSxbMywxLCJBIFxcdGltZXMgQSBcXHRpbWVzIFMiXSxbMCwxLCJcXGxhbmdsZSBcXGxhbmdsZSB1LCB2LCB0IFxccmFuZ2xlLCBcXGxhbmdsZSB4LCB5LCB0XFxyYW5nbGUsIHNcXHJhbmdsZSJdLFsxLDIsIj9eQSBcXHRpbWVzID9eQSBcXHRpbWVzIDEiXSxbMiwzLCI/XkEiLDJdLFswLDQsIlxcbGFuZ2xlIFxcbGFuZ2xlIHUsIHgsIHNcXHJhbmdsZSwgXFxsYW5nbGUgdiwgeSwgc1xccmFuZ2xlLCB0XFxyYW5nbGUiXSxbNCw1LCI/XkEgXFx0aW1lcyA/XkEgXFx0aW1lcyAxIiwyXSxbNSwzLCI/XkEiLDJdXQ==
\[\begin{tikzcd}
	{A^4 \times \alpha^2} &&& {(A^2 \times \alpha) \times (A^2 \times \alpha) \times \alpha} && {A \times A \times \alpha} \\
	{(A^2 \times \alpha) \times (A^2 \times \alpha) \times \alpha} &&& {A \times A \times \alpha} && A,
	\arrow["{\langle \langle u, v, b \rangle, \langle x, y, b\rangle, a\rangle}", from=1-1, to=1-4]
	\arrow["{?^A \times ?^A \times 1}", from=1-4, to=1-6]
	\arrow["{?^A}"', from=1-6, to=2-6]
	\arrow["{\langle \langle u, x, a\rangle, \langle v, y, a\rangle, b\rangle}", from=1-1, to=2-1]
	\arrow["{?^A \times ?^A \times 1}"', from=2-1, to=2-4]
	\arrow["{?^A}"', from=2-4, to=2-6]
\end{tikzcd}\]
which is a diagrammatic $\cS$-equation $A^4 \times \alpha^2 \rightrightarrows A$ ($A \in \cS\Alg$). We thus obtain a diagrammatic $\SF(\V)$-presentation $\cP = (\cS, \sfE)$ for this $\alpha$-parametrized theory ($\alpha \in \ob\V$) of reading many bits, and \ref{vmnd_pres_by_ff_jpres} entails that $\cP$ presents a strongly finitary $\V$-monad $\T_\cP$ on $\V$ with $\T_\cP\Alg \cong \cP\Alg$ in $\V\CAT/\V$.  

Continuing to suppose that $\V$ is cartesian closed, now let $\J \hookrightarrow \V$ be any bounded and eleutheric \emph{system} of arities; for example, as in Example \ref{PPexamples}, we may take $\J = \SF(\V)$, or if $\V$ is locally $\alpha$-presentable as a cartesian closed category then we may take $\J = \V_\alpha$. Let $J \in \ob\J$ be a fixed arity. We define a diagrammatic $\J$-presentation for the $J$-parametrized theory of restriction \cite[\S V.B]{Statoninstances}. The free-form $\J$-signature $\cS$ will have a single operation symbol $\nu$ of arity $J$ and parameter $1$, so that an $\cS$-algebra is an object $A$ of $\V$ equipped with a morphism $\nu^A : A^J \to A$. The desired diagrammatic $\J$-presentation $\cP = (\cS, \sfE)$ will have diagrammatic $\cS$-equations corresponding to the equations of \cite{Statoninstances} annotated as (idem-$\nu$) and $(\nu/\nu)$. The first equation (idem-$\nu$) is written there as $\nu(\lambda j. x) \equiv x$ for $x : A$, and can be expressed as the commutativity of the diagram
% https://q.uiver.app/?q=WzAsMyxbMCwwLCJBIl0sWzIsMCwiQV5KIl0sWzQsMCwiQSJdLFswLDEsIlxcRGVsdGFfQV5KIl0sWzEsMiwiXFxudV5BIl0sWzAsMiwiMSIsMix7ImN1cnZlIjoyfV1d
\[\begin{tikzcd}
	A && {A^J} && A
	\arrow["{\Delta_A}", from=1-1, to=1-3]
	\arrow["{\nu^A}", from=1-3, to=1-5]
	\arrow["1"', curve={height=12pt}, from=1-1, to=1-5]
\end{tikzcd}\]
(where $\Delta_A : A \to A^J$ is the exponential transpose of the first projection $\pi_1 : A \times J \to A$), which is a diagrammatic $\cS$-equation $A \rightrightarrows A$ ($A \in \cS\Alg$). The second equation $(\nu/\nu)$ is written as $\nu\left(\lambda j. \nu\left(\lambda k. f(j, k)\right)\right) \equiv \nu\left(\lambda j. \nu\left(\lambda k. f(k, j)\right)\right)$ for $f : A^{J \times J}$, and it can be expressed as a diagrammatic $\cS$-equation of the form $A^{J \times J} \rightrightarrows A$ ($A \in \cS\Alg$) via the diagram
% https://q.uiver.app/?q=WzAsOCxbMCwwLCJBXntKIFxcdGltZXMgSn0iXSxbMSwwLCJcXGxlZnQoQV5KXFxyaWdodCleSiJdLFszLDAsIkFeSiJdLFs0LDAsIkEiXSxbMCwxLCJBXntKIFxcdGltZXMgSn0iXSxbMSwxLCJcXGxlZnQoQV5KXFxyaWdodCleSiJdLFszLDEsIkFeSiJdLFs0LDEsIkEiXSxbMCwxLCJcXHNpbSJdLFsxLDIsIlxcbGVmdChcXG51XkFcXHJpZ2h0KV5KIl0sWzIsMywiXFxudV5BIl0sWzAsNCwiQV5zIiwyXSxbNCw1LCJcXHNpbSIsMl0sWzUsNiwiXFxsZWZ0KFxcbnVeQVxccmlnaHQpXkoiLDJdLFs2LDcsIlxcbnVeQSIsMl0sWzMsNywiIiwwLHsib2Zmc2V0IjoxLCJzdHlsZSI6eyJib2R5Ijp7Im5hbWUiOiJub25lIn0sImhlYWQiOnsibmFtZSI6Im5vbmUifX19XSxbMyw3LCJcXHBhcmFsbGVsIiwxLHsic3R5bGUiOnsiYm9keSI6eyJuYW1lIjoibm9uZSJ9LCJoZWFkIjp7Im5hbWUiOiJub25lIn19fV1d
\[\begin{tikzcd}
	{A^{J \times J}} & {\left(A^J\right)^J} && {A^J} & A \\
	{A^{J \times J}} & {\left(A^J\right)^J} && {A^J} & A
	\arrow["\kappa_{JJ}", from=1-1, to=1-2]
	\arrow["\sim"', from=1-1, to=1-2]
	\arrow["{\left(\nu^A\right)^J}", from=1-2, to=1-4]
	\arrow["{\nu^A}", from=1-4, to=1-5]
	\arrow["{A^s}"', from=1-1, to=2-1]
	\arrow["\kappa_{JJ}", from=2-1, to=2-2]
	\arrow["\sim"', from=2-1, to=2-2]
	\arrow["{\left(\nu^A\right)^J}"', from=2-2, to=2-4]
	\arrow["{\nu^A}"', from=2-4, to=2-5]
	\arrow[shift right=1, draw=none, from=1-5, to=2-5]
	\arrow["\parallel"{description}, draw=none, from=1-5, to=2-5]
\end{tikzcd}\]
where $\kappa_{XY}:A^{X \times Y} \xrightarrow{\sim} (A^Y)^X$ denotes the canonical isomorphism $(X,Y \in \V)$, and $s : J \times J \xrightarrow{\sim} J \times J$ is the `twist' isomorphism. We thus obtain a diagrammatic $\J$-presentation $\cP = (\cS, \sfE)$ for this $J$-parametrized theory of restriction, and \ref{vmnd_pres_by_ff_jpres} entails that $\cP$ presents a $\J$-ary $\V$-monad $\T_\cP$ on $\V$ such that $\T_\cP\Alg \cong \cP\Alg$ in $\V\CAT/\V$. By combining the $J$-parametrized theory of restriction with the $\alpha$-parametrized theory of reading many bits, for $\alpha = J$, and adding a single additional equation that requires that $\nu$ \textit{commutes with} $?$ in the sense of \S\ref{sumtensor} below, we obtain a diagrammatic $\J$-presentation for Staton's parametrized theory of \textit{instantiating and reading bits} \cite[\S V.B]{Statoninstances}, assuming that $\J$ contains $\SF(\V)$.\qed        
}
\end{egg}

\section{Constructions on diagrammatic $\J$-presentations: Sum and tensor}\label{sumtensor}

We can define the \emph{sum} and \emph{tensor} of two diagrammatic (or free-form) $\J$-presentations $\cP_1$ and $\cP_2$ as follows. For each $i \in \{1, 2\}$, let us write $\cP_i = (\cS_i,\sfE_i)$, where $\cS_i$ is a free-form $\J$-signature with arities $J_{i\sigma}$ and parameters $C_{i\sigma}$ $(\sigma \in \cS_i)$.

To define the sum $\cP_1 + \cP_2$, we first define a free-form $\J$-signature $\cS_1 + \cS_2$ by equipping the disjoint union of the (small) sets of operation symbols $\cS_1$ and $\cS_2$ with the associated arities $J_{i\sigma}$ and parameters $C_{i\sigma}$ $(i \in \{1,2\}, \sigma \in \cS_i)$.  An $(\cS_1 + \cS_2)$-algebra is then an object $A$ of $\C$ carrying the structure of both an $\cS_1$-algebra and an $\cS_2$-algebra, and we readily deduce by \ref{universal_property_SAlg} that moreover $(\cS_1 + \cS_2)\Alg \cong \cS_1\Alg \times_\C \cS_2\Alg$ in $\V\CAT/\C$. For each $i \in \{1,2\}$, $\sfE_i$ is a family of diagrammatic $\cS_i$-equations $\omega^{i\delta} \doteq \nu^{i\delta}$ $(\delta \in \cD_i)$, each of which consists of morphisms $\omega^{i\delta}_A,\nu^{i\delta}_ A:\C(J_{i\delta},A) \rightrightarrows \C(C_{i\delta},A)$ that are $\V$-natural in $A \in \cS_i\Alg$ (where we assume without loss of generality that the sets $\cS_i$ and $\cD_i$ are disjoint). Given $i \in \{1,2\}$, if we denote the $\cS_i$-algebra underlying each $(\cS_1+\cS_2)$-algebra $A$ also by $A$, then for each $\delta \in \cD_i$ we obtain a parallel pair of $\V$-natural families $\omega^{i\delta}_A,\nu^{i\delta}_A:\C(J_{i\delta},A) \rightrightarrows \C(C_{i\delta},A)$ $(A \in (\cS_1+\cS_2)\Alg)$, so that we obtain a diagrammatic $(\cS_1+\cS_2)$-equation that we may write also as $\omega^{i\delta} \doteq \nu^{i\delta}$.  Collectively, these equations $\omega^{i\delta} \doteq \nu^{i\delta}$ $(i \in \{1,2\}, \delta \in \cD_i)$ constitute a system of diagrammatic $(\cS_1 + \cS_2)$-equations that we may write as $\sfE_1 + \sfE_2$.

Thus we obtain a diagrammatic $\J$-presentation $\cP_1 + \cP_2 := (\cS_1+\cS_2,\sfE_1+\sfE_2)$, for which a $(\cP_1+\cP_2)$-algebra is an object $A$ of $\C$ equipped with the structure of both a $\cP_1$-algebra and a $\cP_2$-algebra, and moreover $(\cP_1 + \cP_2)\Alg \cong \cP_1\Alg \times_\C \cP_2\Alg$ in $\V\CAT/\C$ as a consequence of the above isomorphism $(\cS_1 + \cS_2)\Alg \cong \cS_1\Alg \times_\C \cS_2\Alg$. 

Now supposing that $\C = \V$ and that $\J \hookrightarrow \V$ is a \emph{system} of arities (see \ref{sys_ar}), we next add further diagrammatic $(\cS_1+\cS_2)$-equations to the sum $\cP_1 + \cP_2$ to obtain the \emph{tensor product} $\cP_1 \tensor \cP_2$ of the diagrammatic $\J$-presentations $\cP_1, \cP_2$. To do this, we employ the following terminology: given any two parametrized operations $\omega_i:[J_i,A] \rightarrow [C_i,A]$ $(i \in \{1,2\})$ on an object $A$ of $\V$, where $J_i,C_i \in \ob\V$ and we write the internal hom in $\V$ as $[-,-]$, we obtain a pair of parametrized operations $$\omega_1 * \omega_2,\;\;\omega_1 \tstar \omega_2\;\;:\;\;[J_1 \otimes J_2,A] \rightrightarrows [C_1 \otimes C_2,A]$$
that are called the \textit{first and second Kronecker products} of $\omega_1$ and $\omega_2$ (cf. \cite{Commutants}) and are defined as the following composites, respectively: 
\[ [J_1 \tensor J_2, A] \cong [J_2, [J_1, A]] \xrightarrow{\left[1, \omega_1\right]} [J_2, [C_1, A]] \cong [C_1, [J_2, A]] \xrightarrow{\left[1, \omega_2\right]} [C_1, [C_2, A]] \cong [C_1 \tensor C_2, A], \]
\[ [J_1 \tensor J_2, A] \cong [J_1, [J_2, A]] \xrightarrow{\left[1, \omega_2\right]} [J_1, [C_2, A]] \cong [C_2, [J_1, A]] \xrightarrow{\left[1, \omega_1\right]} [C_2, [C_1, A]] \cong [C_1 \tensor C_2, A]. \] 
We say that $\omega_1$ \textit{commutes with} $\omega_2$ if $\omega_1 * \omega_2 = \omega_1 \tstar \omega_2$.

Given operation symbols $\sigma_1 \in \cS_1$ and $\sigma_2 \in \cS_2$, we can consider the parametrized operations $\sigma_i^A:[J_{i\sigma_i},A] \to [C_{i\sigma_i},A]$ $(i \in \{1,2\})$ carried by each $(\cS_1+\cS_2)$-algebra $A$.  The first and second Kronecker products of $\sigma_1^A$ and $\sigma_2^A$ are therefore morphisms
\begin{equation}\label{eq:kr_prods_for_s1s2alg}\sigma_1^A * \sigma_2^A,\;\;\sigma_1^A \tstar \sigma_2^A\;\;:\;\;[J_{1\sigma_1} \otimes J_{2\sigma_2},A] \rightrightarrows [C_{1\sigma_1} \otimes C_{2\sigma_2},A]\end{equation}
that are $\V$-natural in $A \in (\cS_1+\cS_2)\Alg$, in view of their definitions above, since $\sigma_1^A$ and $\sigma_2^A$ are $\V$-natural in $A \in (\cS_1+\cS_2)\Alg$ by \ref{signature_operation_natural}.  Hence, the morphisms in \eqref{eq:kr_prods_for_s1s2alg} constitute $\V$-natural $(\cS_1+\cS_2)$-operations $\sigma_1 * \sigma_2$ and $\sigma_1 \tstar \sigma_2$, so we obtain a diagrammatic $(\cS_1+\cS_2)$-equation $\sigma_1 * \sigma_2 \doteq \sigma_1 \tstar \sigma_2$.  Adjoining these equations $\sigma_1 * \sigma_2 \doteq \sigma_1 \tstar \sigma_2$ $(\sigma_1 \in \cS_1, \sigma_2 \in \cS_2)$ to $\sfE_1+\sfE_2$, we obtain a system of diagrammatic $(\cS_1+\cS_2)$-equations $\sfE_1 \otimes \sfE_2$.

Thus we obtain a diagrammatic $\J$-presentation $\cP_1 \otimes \cP_2 = (\cS_1+\cS_2,\sfE_1\otimes\sfE_2)$, for which a $(\cP_1 \otimes \cP_2)$-algebra is an object $A$ of $\V$ equipped with both the structure of a $\cP_1$-algebra and a $\cP_2$-algebra such that $\sigma_1^A$ commutes with $\sigma_2^A$ for each pair of operation symbols $\sigma_1 \in \cS_1$ and $\sigma_2 \in \cS_2$, and $(\cP_1 \otimes \cP_2)\Alg$ is isomorphic (in $\V\CAT/\V$) to the full sub-$\V$-category of $\cP_1\Alg \times_\V \cP_2\Alg$ consisting of these objects.  For example, by bringing together the diagrammatic $\J$-presentations $\cP_1$ and $\cP_2$ for \textit{reading many bits} and \textit{restriction} (\ref{Statonexamples}) and forming the tensor product $\cP_1 \otimes \cP_2$, we arrive at a diagrammatic $\J$-presentation for Staton's parametrized theory of \textit{instantiating and reading bits}; the diagrammatic equation expressing that $\nu$ commutes with $?$ corresponds to Staton's syntactic equation $(\nu/?)$ \cite[\S V.B]{Statoninstances}.

\section{$\J$-ary varieties of algebras for $\K$-relative monads}
\label{standardizedpresentations}

The notion of \textit{abstract clone} of Philip Hall provides an equivalent way of encoding the data for a Lawvere theory in terms of a family of sets $H(n)$ indexed by the finite cardinals $n$, with some further structure; see \cite{Cohn}.  Variants and generalizations of this concept include \textit{Kleisli triples} \cite{Manesbook}, \textit{enriched clones} \cite{Statonpredicate, Fioreclones}, and \textit{relative monads} \cite{Altenkirchmonads}.  By transporting the notion of relative monad to the setting of a subcategory of arities $\K \hookrightarrow \C$ enriched in $\V$, we obtain a notion of \textit{$\K$-relative $\V$-monad} (\ref{relativemonad}).  In the special case where $\K$ is eleutheric, $\K$-relative $\V$-monads provide an equivalent way of describing $\K$-ary $\V$-monads in terms of families of objects $H = (HK)_{K \in\,\ob\K}$ of $\C$ equipped with further structure (\ref{rel_mnds_case_k_equals_j}).  But $\K$-relative $\V$-monads are of interest also in cases where $\K$ is not eleutheric, and in the present section we do \textit{not} impose such an assumption on $\K$.  Instead, we consider the case where $\K$ is merely \textit{contained in} a bounded and eleutheric subcategory of arities $\J \hookrightarrow \C$ (under the assumptions of \ref{Vcategoryassumptions} and \ref{subcategoryaritiesassumptions}), and in this case we show that for each $\K$-relative $\V$-monad $H$, the $\V$-category of $H$-algebras $H\Alg$ is precisely the $\V$-category of $\cP_H$-algebras for a diagrammatic $\J$-presentation $\cP_H$, so that $H\Alg$ is a $\J$-ary variety and hence is $\J$-ary monadic over $\C$.  For example, this result is applicable to \textit{any} small subcategory of arities $\K$ in a locally presentable $\V$-category $\C$ over a locally presentable closed category $\V$, since in this case $\K$ is contained in the bounded and eleutheric subcategory of arities $\C_\alpha \hookrightarrow \C$ for a suitable $\alpha$ (\ref{runningexamples}, \ref{lpres}).  But the presentations $\cP_H$ obtained in this way are important even in the case where $\K = \J$, because by applying this construction to the $\J$-relative $\V$-monad corresponding to a $\J$-ary $\V$-monad $\T$ we obtain a diagrammatic $\J$-presentation of $\T$ that is more economical than the canonical $\J$-presentation of $\T$ obtained by way of the monadicity of $\J$-ary $\V$-monads (\ref{signatures}, \ref{presentations}); examples are discussed in \ref{internalaffinespaces}.

\begin{defn}
\label{relativemonad}
Given a full sub-$\V$-category $\K \hookrightarrow \C$, a \textbf{$\K$-relative $\V$-monad (on $\C$)} is a triple $(H, e, m)$ consisting of a $\V$-functor $H : \K \to \C$ and $\V$-natural families of morphisms $e_J : J \to HJ$ $(J \in \K)$ and $m_{JK} : \C(J, HK) \to \C(HJ, HK)$ ($J, K \in \K$) making the following diagrams commute for all $J, K, L \in \ob\K$:
% https://q.uiver.app/?q=WzAsMTEsWzAsMCwiSSJdLFsyLDAsIlxcQyhKLCBUSikiXSxbNCwwLCJcXEMoVEosIFRKKSJdLFswLDEsIlxcQyhKLCBUSykiXSxbMiwxLCJcXEMoVEosIFRLKSJdLFs0LDEsIlxcQyhKLCBUSykiXSxbMCwyLCJcXEMoSiwgVEspIFxcb3RpbWVzIFxcQyhLLCBUTCkiXSxbNCwyLCJcXEMoVEosIFRLKSBcXG90aW1lcyBcXEMoVEssIFRMKSJdLFswLDMsIlxcQyhKLCBUSykgXFxvdGltZXMgXFxDKFRLLCBUTCkiXSxbMCw0LCJcXEMoSiwgVEwpIl0sWzQsNCwiXFxDKFRKLCBUTCkiXSxbMCwxLCJcXHdpZGVoYXR7XFxldGFfSn0iXSxbMSwyLCJcXG11X3tKSn0iXSxbMCwyLCJ1X3tUSn0iLDIseyJjdXJ2ZSI6Mn1dLFszLDQsIlxcbXVfe0pLfSJdLFs0LDUsIlxcQ1xcbGVmdChcXGV0YV9KLCAxXFxyaWdodCkiXSxbMyw1LCIxIiwyLHsiY3VydmUiOjJ9XSxbNiw3LCJcXG11X3tKS30gXFxvdGltZXMgXFxtdV97Skx9IiwyXSxbNiw4LCIxIFxcb3RpbWVzIFxcbXVfe0tMfSJdLFs4LDksImMiXSxbNywxMCwiYyIsMl0sWzksMTAsIlxcbXVfe0pMfSIsMl1d
\[\begin{tikzcd}
	I & {\C(J, HJ)} & {\C(HJ, HJ)} & &  \\
	{\C(J, HK)} & {\C(HJ, HK)}& {\C(J, HK)} & &\\
	{\C(J, HK) \otimes \C(K, HL)} && {\C(HJ, HK) \otimes \C(HK, HL)} \\
	{\C(J, HK) \otimes \C(HK, HL)} & {\C(J, HL)} & {\C(HJ, HL)},	
	\arrow["{e_J}", from=1-1, to=1-2]
	\arrow["{m_{JJ}}", from=1-2, to=1-3]
	\arrow["{1_{HJ}}"', curve={height=14pt}, from=1-1, to=1-3]
	\arrow["{m_{JK}}", from=2-1, to=2-2]
	\arrow["{\C\left(e_J, 1\right)}", from=2-2, to=2-3]
	\arrow["1"', curve={height=14pt}, from=2-1, to=2-3]
	\arrow["{m_{JK} \otimes m_{KL}}", from=3-1, to=3-3]
	\arrow["{1 \otimes m_{KL}}", from=3-1, to=4-1]
	\arrow["c", from=4-1, to=4-2]
	\arrow["c", from=3-3, to=4-3]
	\arrow["{m_{JL}}", from=4-2, to=4-3]
\end{tikzcd}\]
where each $c$ is a composition morphism of $\C$. \qed 
\end{defn}

\begin{rmk}
\label{relative_mnd_rmk}
In Definition \ref{relativemonad}, it actually suffices to just assume that we have a family of objects $(HJ)_{J \in \ob\K}$ of $\C$ and families of morphisms $e_J : J \to HJ$ $(J \in \ob\K)$ and $m_{JK} : \C(J, HK) \to \C(HJ, HK)$ $(J, K \in \ob\K)$ making the diagrams of \ref{relativemonad} commute, because then (as in the ordinary context \cite{Altenkirchmonads}) it follows straightforwardly that the object assignment $J \mapsto HJ$ ($J \in \ob\K$) extends to a $\V$-functor $H : \K \to \C$, with respect to which the families of morphisms $e_J$ and $m_{JK}$ are $\V$-natural. \qed 
\end{rmk}

\begin{para}\label{rel_mnds_case_k_equals_j}
In the special case where $\K$ is a subcategory of arities $\K=\J \hookrightarrow \C$ satisfying the assumptions of \ref{Vcategoryassumptions} and \ref{subcategoryaritiesassumptions} (including eleuthericity\footnote{In fact, the remarks in this paragraph apply when $\J$ is an arbitrary (possibly large) eleutheric subcategory of arities in an arbitrary $\V$-category $\C$ over an arbitrary closed category $\V$.}), there is an equivalence between $\J$-relative $\V$-monads and $\J$-ary $\V$-monads.  To see this, first note that $\V\CAT(\J, \C)$ is a monoidal category, whose monoidal structure is defined by transporting the monoidal structure of $\End_{\underJ}(\C)$ along the equivalence $\V\CAT(\J, \C) \simeq \End_{\underJ}(\C)$ discussed in \ref{Jary}. Thus, the unit of the monoidal category $\V\CAT(\J, \C)$ is the $\V$-functor $j : \J \hookrightarrow \C$, while the monoidal product of $H, H' : \J \to \C$ is defined by $H \ast H' := (\Lan_jH) \circ H'$. A monoid in the monoidal category $\V\CAT(\J, \C)$ is thus a triple $(H, e, m)$ consisting of a $\V$-functor $H : \J \to \C$ and $\V$-natural transformations $e : j \to H$ and $m : H \ast H \to H$ satisfying the associativity and unit equations.  Unpacking $m$, we have for each $K \in \ob\J$ a morphism $m_K : (H \ast H)K \to HK$, where $(H \ast H)K = \left(\Lan_jH\right)HK = \int^{J \in \J} \C(J, HK) \tensor HJ$ by the coend formula for left Kan extensions \cite[4.25]{Kelly}. Hence $m$ is equivalently given by a family of morphisms $m_{JK} : \C(J, HK) \tensor HJ \to HK$ or equivalently $m_{JK} : \C(J, HK) \to \C(HJ, HK)$, $\V$-natural in $J,K \in \J$.  Expressing $m$ in this way, the unit and associativity axioms for a monoid $(H,e,m)$ in $\V\CAT(\J, \C)$ translate into the axioms for a $\J$-relative $\V$-monad in \ref{relativemonad}.  Thus, $\J$-relative $\V$-monads may be identified with monoids in $\V\CAT(\J,\C)$ and so are the objects of a category that is equivalent to the category of $\J$-ary $\V$-monads on $\C$.  In particular, if $(H, e, m)$ is a $\J$-relative $\V$-monad on $\C$, then the induced $\V$-endofunctor $\Lan_j H : \C \to \C$ underlies a $\J$-ary $\V$-monad on $\C$, and if $\T$ is a $\J$-ary $\V$-monad on $\C$, then the $\V$-functor $Tj : \J \to \C$ underlies a $\J$-relative $\V$-monad.  \qed
\end{para}

\noindent The following is an enrichment (and a slight variation, \ref{relative_mnd_algebra_rmk}) of a definition in \cite{Altenkirchmonads}:

\begin{defn}
\label{relativemonadalgebra}
Let $\K \hookrightarrow \C$ be a full sub-$\V$-category, and let $H = (H, e, m)$ be a $\K$-relative $\V$-monad on $\C$. An \textbf{$H$-algebra} is an object $A$ of $\C$ equipped with a $\V$-natural family of morphisms $\alpha_J : \C(J, A) \to \C(HJ, A)$ ($J \in \K$) making the following diagrams commute for all $J, K \in \ob\K$:
\begin{equation}\label{eq:relative_alg_1}
\begin{tikzcd}
	{\C(J, A)} && {\C(HJ, A)} && {\C(J, A)}
	\arrow["{\alpha_J}", from=1-1, to=1-3]
	\arrow["{\C\left(e_J, 1\right)}", from=1-3, to=1-5]
	\arrow["1"', curve={height=12pt}, from=1-1, to=1-5]
\end{tikzcd}
\end{equation}
% https://q.uiver.app/?q=WzAsNixbMCwwLCJcXEMoSiwgSEspIFxcdGVuc29yIFxcQyhLLCBBKSJdLFsyLDAsIlxcQyhISiwgSEspIFxcdGVuc29yIFxcQyhLLCBBKSJdLFs0LDAsIlxcQyhISiwgSEspIFxcdGVuc29yIFxcQyhISywgQSkiXSxbNCwxLCJcXEMoSEosIEEpIl0sWzAsMSwiXFxDKEosIEhLKSBcXHRlbnNvciBcXEMoSEssIEEpIl0sWzIsMSwiXFxDKEosIEEpIl0sWzAsMSwibV97Skt9IFxcdGVuc29yIDEiXSxbMSwyLCIxIFxcdGVuc29yIFxcYWxwaGFfSyJdLFsyLDMsImMiXSxbMCw0LCIxIFxcdGVuc29yIFxcYWxwaGFfSyIsMl0sWzQsNSwiYyIsMl0sWzUsMywiXFxhbHBoYV9KIiwyXV0=
\begin{equation}
\label{eq:relative_alg_2}
\begin{tikzcd}
	{\C(J, HK) \tensor \C(K, A)} &  & {\C(HJ, HK) \tensor \C(HK, A)} \\
	{\C(J, HK) \tensor \C(HK, A)} & {\C(J, A)} & {\C(HJ, A)},
	\arrow["{m_{JK} \tensor \alpha_K}", from=1-1, to=1-3]
	\arrow["c", from=1-3, to=2-3]
	\arrow["{1 \tensor \alpha_K}"', from=1-1, to=2-1]
	\arrow["c"', from=2-1, to=2-2]
	\arrow["{\alpha_J}"', from=2-2, to=2-3]
\end{tikzcd}
\end{equation}
where each $c$ is a composition morphism of $\C$. \qed
\end{defn} 

\begin{rmk}
\label{relative_mnd_algebra_rmk}
In Definition \ref{relativemonadalgebra}, we may omit the requirement that the morphisms $\alpha_J$ are $\V$-natural in $J \in \K$, as this follows automatically from the commutativity of the diagrams in \ref{relativemonadalgebra} (cf.~\cite{Altenkirchmonads} for the same fact in the ordinary context).  Note also that the structural morphisms $\alpha_J$ of an $H$-algebra may equivalently be written as $\alpha_J : \C(J, A) \tensor HJ \to A$.  \qed    
\end{rmk}

\begin{para}\label{case_j_equals_k_halg}
In the special case of \ref{rel_mnds_case_k_equals_j} where $\K = \J \hookrightarrow \C$, the equivalence of monoidal categories $\V\CAT(\J, \C) \simeq \End_{\underJ}(\C)$ induces an \textit{action} of the monoidal category $\V\CAT(\J,\C)$ on the $\V$-category $\C$, in the form of a strong monoidal functor $\Phi:\V\CAT(\J,\C) \rightarrow \V\CAT(\C,\C)$ given by left Kan extension.  If $H = (H,e,m)$ is a $\J$-relative $\V$-monad, regarded as a monoid in $\V\CAT(\J,\C)$ (\ref{rel_mnds_case_k_equals_j}), then the corresponding $\J$-ary $\V$-monad is the monoid $\Phi(H)$ in $\V\CAT(\C,\C)$ obtained canonically by applying $\Phi$.  An $H$-algebra is equivalently given by an algebra for the $\V$-monad $\Phi(H)$, which from the point of view of actions of monoidal categories we may call an \textit{$H$-module in $\C$} for the monoid $H$ in $\V\CAT(\J,\C)$. \qed
\end{para}

\noindent The algebras for a $\K$-relative $\V$-monad can be described in terms of diagrammatic $\K$-presentations in the sense of \ref{palg_arbitrary_j}:

\begin{prop}\label{rel_mnd_pres}
Let $H = (H, e, m)$ be a $\K$-relative $\V$-monad for a small full sub-$\V$-category $\K \hookrightarrow \C$. Then there is a diagrammatic $\K$-presentation $\cP_H$ such that $\cP_H$-algebras are precisely $H$-algebras.
\end{prop}

\begin{proof}
The free-form $\K$-signature $\cS$ for $H$-algebras will have an operation symbol $\alpha_J$ of arity $J \in \ob\K$ and parameter $HJ \in \ob\C$ for each $J \in \ob\K$, so that an $\cS$-algebra is an object $A$ of $\C$ equipped with structural morphisms $\alpha_J^A : \C(J, A) \to \C(HJ, A)$ for all $J \in \ob\K$. For each object $J$ of $\K$, the diagram \eqref{eq:relative_alg_1} describes a  diagrammatic $\cS$-equation $\C\left(J, A\right) \rightrightarrows \C\left(J, A\right)$ ($A \in \cS\Alg$). For each pair $J, K$ of objects of $\K$, the diagram \eqref{eq:relative_alg_2} describes a diagrammatic $\cS$-equation $\C\left(K, A\right) \rightrightarrows \C\left(\C(J, HK) \tensor HJ, A\right)$ ($A \in \cS\Alg$). We thus obtain a diagrammatic $\K$-presentation $\cP_H = (\cS, \sfE)$, and $\cP_H$-algebras are precisely $H$-algebras, in view of \ref{relative_mnd_algebra_rmk}.
\end{proof}

\begin{defn}\label{defn_vcat_halg}
Given a $\K$-relative $\V$-monad $H$ for a small full sub-$\V$-category $\K \hookrightarrow \C$, the \textbf{$\V$-category of $H$-algebras} is by definition the $\V$-category $H\Alg := \cP_H\Alg$ of $\cP_H$-algebras (\ref{palg_arbitrary_j}, \ref{rel_mnd_pres}).
\end{defn}

\begin{rmk}\label{halg_phih_alg}
In the special case of \ref{rel_mnds_case_k_equals_j} and \ref{case_j_equals_k_halg}, where $\K = \J$, the $\V$-category $H\Alg$ is isomorphic to the $\V$-category $\Phi(H)\Alg$ of algebras for the associated $\J$-ary $\V$-monad $\Phi(H)$, which we may equally call the \textit{$\V$-category of $H$-modules in $\C$} for the monoid $H$ in $\V\CAT(\J,\C)$. \qed
\end{rmk}

\begin{theo}\label{rel_mnd_thm}
Let $\C$ and $\V$ be as in \ref{Vcategoryassumptions}, and let $\K \hookrightarrow \C$ be a full sub-$\V$-category that is contained in some bounded and eleutheric subcategory of arities $\J \hookrightarrow \C$. For each $\K$-relative $\V$-monad $H$, the $\V$-category $H\Alg$ of $H$-algebras is strictly $\J$-monadic over $\C$ (as it is a $\J$-ary variety).
\end{theo}
\begin{proof}
Under the given hypotheses, we may regard the diagrammatic $\K$-presentation $\cP_H$ also as a diagrammatic $\J$-presentation, and the result follows from Theorem \ref{vmnd_pres_by_ff_jpres}.
\end{proof}

\begin{egg}\label{lpres}
Let $\V$ be a locally $\alpha$-presentable closed category and $\C$ a locally $\alpha$-presentable $\V$-category.  Then any small full sub-$\V$-category $\K \hookrightarrow \C$ is contained in a bounded and eleutheric subcategory of arities $\J \hookrightarrow \C$.  Indeed, by \cite[7.4]{Kellystr} we may choose some regular cardinal $\beta \geq \alpha$ such that every object of $\K$ is $\beta$-presentable in the enriched sense, and hence $\K$ is contained in the bounded and eleutheric subcategory of arities $\J = \C_\beta \hookrightarrow \C$ discussed in \ref{runningexamples}(1), where we noted that these data satisfy the assumptions of \ref{Vcategoryassumptions} and \ref{subcategoryaritiesassumptions}.  Hence, Theorem \ref{rel_mnd_thm} entails the following:
\end{egg}

\begin{cor}
Let $\V$ be a locally $\alpha$-presentable closed category, let $\C$ be a locally $\alpha$-presentable $\V$-category, and let $\K \hookrightarrow \C$ be a small full sub-$\V$-category.  Then there is a regular cardinal $\beta \geq \alpha$ such that for each $\K$-relative $\V$-monad $H$ on $\C$, the $\V$-category of $H$-algebras $H\Alg$ is strictly $\beta$-ary monadic over $\C$. \qed
\end{cor}

\begin{defn}
\label{standardizedpresentation}
Returning to the setting of a given subcategory of arities $\J \hookrightarrow \C$ satisfying the assumptions in \ref{Vcategoryassumptions} and \ref{subcategoryaritiesassumptions}, let $\T$ be a $\J$-ary $\V$-monad on $\C$, recalling that $Tj:\J \rightarrow \C$ then carries the structure of a $\J$-relative $\V$-monad that we denote simply by $Tj$ (\ref{rel_mnds_case_k_equals_j}).  By \ref{defn_vcat_halg} and \ref{halg_phih_alg}, we have a diagrammatic $\J$-presentation $\cP_{Tj}$ with $\cP_{Tj}\Alg = Tj\Alg \cong \T\Alg$ in $\V\CAT/\C$. We call $\cP_{Tj}$ the \textbf{standardized $\J$-presentation of $\T$}. Note that a $\cP_{Tj}$-algebra is an object $A$ of $\C$ equipped with morphisms $\C(J,A) \rightarrow \C(TJ,A)$ or $\C(J,A) \otimes TJ \rightarrow A$ $(J \in \ob\J)$ satisfying associativity and unit axioms (\ref{relativemonadalgebra}, \ref{relative_mnd_algebra_rmk}).\qed   
\end{defn}

\begin{para}
\label{standardizedpresentationJtheory}
{
Let $\C = \V$, and suppose that $\J \hookrightarrow \V$ is a \emph{system} of arities.  By \ref{sys_ar}, $\Mnd_{\underJ}(\V)$ is equivalent to the category $\Th_{\underJ}$ of $\J$-theories. Let $\scrT$ be a $\J$-theory with associated identity-on-objects $\J$-cotensor-preserving $\V$-functor $\tau : \J^\op \to \scrT$. Writing $\T$ for the $\J$-ary $\V$-monad on $\V$ corresponding to $\scrT$, we may consider also the $\J$-relative $\V$-monad on $\V$ corresponding to $\T$, whose underlying $\V$-functor is $\scrT(\tau-, I) : \J \to \V$ (see \cite[10.3]{Commutants}). By \ref{standardizedpresentation}, this $\J$-relative $\V$-monad determines a diagrammatic $\J$-presentation $\cP$ for which $\cP\Alg$ is isomorphic to the $\V$-category of $\T$-algebras and hence to the $\V$-category of normal $\scrT$-algebras, which is in turn equivalent to the $\V$-category of (all) $\scrT$-algebras (see \ref{sys_ar}). A $\cP$-algebra is then given by an object $A$ of $\V$ equipped with structural morphisms $\alpha_J^A : \scrT(J, I) \to \V(\V(J, A), A)$ or $\alpha_J^A:\scrT(J,I) \otimes \V(J,A) \rightarrow A$ ($J \in \ob\J$) satisfying associativity and unit axioms. \qed 
}
\end{para}

\begin{egg}
\label{internalaffinespaces}
As in \ref{Rmodules}, let $\V$ be a complete and cocomplete cartesian closed category, and let $R$ be an internal rig in $\V$. Employing the system of arities $\SF(\V) \hookrightarrow \V$ (\ref{sys_ar}), in \ref{Rmodules} we exhibited a diagrammatic $\SF(\V)$-presentation whose algebras are the (left) $R$-modules in $\V$.  We may also regard $R$-modules equivalently as the normal algebras of an $\SF(\V)$-theory, namely the \textit{$\V$-category of $R$-matrices} $\mathsf{Mat}_R$, which has hom-objects $\mathsf{Mat}_R(m,n) = R^{n \times m}$ $(m,n \in \N)$, with composition given by internal matrix multiplication \cite[6.4.6]{functional}.  Invoking \ref{standardizedpresentationJtheory}, we thus obtain another diagrammatic $\SF(\V)$-presentation for $R$-modules. As in \ref{standardizedpresentationJtheory}, an algebra for this diagrammatic $\SF(\V)$-presentation is an object $A$ of $\V$ equipped with structural morphisms $\alpha_n^A : \mathsf{Mat}_R(n, 1) \times A^n \to A$, i.e. $\alpha_n^A : R^n \times A^n \to A$, for all $n \in \N$, satisfying associativity and unit axioms, since $\mathsf{Mat}_R(n, 1) = R^{1 \times n} = R^n$. The structural morphism $\alpha_n^A : R^n \times A^n \to A$ can be regarded as the parametrized operation of taking $n$-ary $R$-linear combinations in $A$.     

As in \cite[8.5]{functional}, there is a (non-full) sub-$\V$-category $\mathsf{Mat}_R^\aff$ of $\mathsf{Mat}_R$ whose objects are again the natural numbers and whose hom-objects are subobjects $\mathsf{Mat}_R^\aff(m,n) \hookrightarrow R^{n \times m}$ that describe those matrices in which each row has sum $1$.  By \cite[8.5]{functional}, $\mathsf{Mat}_R^\aff$ is an $\SF(\V)$-theory whose normal algebras in $\V$ are called \emph{(internal left) $R$-affine spaces}; also see \cite[11.2.7]{Pres}. Hence, by \ref{standardizedpresentationJtheory}, we obtain a diagrammatic $\SF(\V)$-presentation $\cP$ for which the associated notion of $\cP$-algebra provides an equivalent way of defining the notion of internal $R$-affine space, as an object $A$ of $\V$ equipped with structural morphisms $\alpha_n^A : R^{n,\aff} \times A^n \to A$, satisfying associativity and unit axioms, where $R^{n,\aff} := \mathsf{Mat}_R^\aff(n,1) \hookrightarrow R^n$ is the subobject of $n$-tuples of $R$ whose sum is $1$.  The structural morphism $\alpha_n^A$ can thus be regarded as the parametrized operation of taking $n$-ary \textit{$R$-affine combinations} in $A$. We call $R\textnormal{-}\mathsf{Aff} := \cP\Alg$ \textit{the $\V$-category of (left) $R$-affine spaces}. \qed     
\end{egg}

\section{A more general monad-pretheory adjunction}
\label{monadpretheory}

In the recent paper \cite{BourkeGarner}, Bourke and Garner fix a small subcategory of arities $k : \K \hookrightarrow \C$ in a locally presentable $\V$-category $\C$ over a locally presentable closed category $\V$, and they consider the notion of a \emph{$\K$-pretheory}, which is by definition a $\V$-category $\scrT$ equipped with a $\V$-functor $\tau : \K^\op \to \scrT$ that is identity-on-objects\footnote{In \cite{BourkeGarner}, the dual notion is used, while in the present paper we employ a notion of pretheory that accords with the concept previously studied (in the unenriched context) by Linton \cite[p. 20]{Lintonoutline} under the name \emph{clone}, and later by Diers \cite{Diers} under the name \emph{theory}.}, so that $\K$-pretheories are the objects of a full subcategory $\Preth_\K(\C)$ of the coslice $\K^\op\slash\V\textnormal{-}\mathsf{Cat}$.  Letting $\Mnd(\C)$ be the category of $\V$-monads on $\C$, they define a functor $\Phi : \Mnd(\C) \to \Preth_\K(\C)$ that sends a $\V$-monad $\T$ on $\C$ to the \emph{Kleisli $\K$-pretheory} $\K_\T^\op$, i.e. the opposite of the full sub-$\V$-category $\K_\T$ of the Kleisli $\V$-category $\C_\T$ on the objects of $\K$, equipped with the $\V$-functor $\K^\op \rightarrow \K_\T^\op$ obtained by restricting the Kleisli left adjoint $F_\T:\C \rightarrow \C_\T$. By \cite[Theorem 6]{BourkeGarner}, $\Phi$ has a left adjoint $\Psi : \Preth_\K(\C) \to \Mnd(\C)$ that sends each $\K$-pretheory $\scrT$ to a $\V$-monad $\Psi(\scrT)$ whose $\V$-category of algebras $\Psi(\scrT)\Alg$ is isomorphic (in $\V\CAT/\C$) to the $\V$-category $\MOD_c(\scrT)$ of \emph{concrete $\scrT$-models}. This $\V$-category is defined\footnote{This definition (in the unenriched context) goes back to Linton \cite{Lintonoutline} and Diers \cite{Diers}.} as the following pullback in $\V\CAT$
% https://q.uiver.app/?q=WzAsNCxbMCwwLCJcXE1PRF9jKFxcc2NyVCkiXSxbMiwwLCJbXFxzY3JULCBcXFZdIl0sWzAsMiwiXFxDIl0sWzIsMiwiW1xcc2NyS15cXG9wLCBcXFZdIl0sWzAsMSwiUF5cXHNjclQiXSxbMCwyLCJVXlxcc2NyVCIsMl0sWzEsMywiW1xcdGF1LCAxXSJdLFsyLDMsIk5fayIsMl0sWzAsMywiIiwxLHsic3R5bGUiOnsibmFtZSI6ImNvcm5lciJ9fV1d
\[\begin{tikzcd}
	{\MOD_c(\scrT)} && {[\scrT, \V]} \\
	\\
	\C && {[\K^\op, \V]},
	\arrow[from=1-1, to=1-3]
	\arrow["{U^\scrT}"', from=1-1, to=3-1]
	\arrow["{[\tau, 1]}", from=1-3, to=3-3]
	\arrow["{N_k}"', from=3-1, to=3-3]
\end{tikzcd}\]
where $N_k : \C \to [\K^\op, \V]$ is the \emph{$k$-nerve} $\V$-functor defined by $N_kC = \C(k-, C)$ ($C \in \C$). A concrete $\scrT$-model is therefore an object $A$ of $\C$ equipped with a $\V$-functor $M : \scrT \to \V$ that extends the $k$-nerve $\C(k-, A) : \K^\op \to \V$ along $\tau : \K^\op \to \scrT$, i.e. $M \circ \tau = \C(k-, A)$. Since $\tau$ is identity-on-objects, $M$ is then given on objects by $MK = \C(K, A)$ $(K \in \ob\scrT = \ob\K)$. 

Our objective is now to employ diagrammatic $\J$-presentations to generalize this monad-pretheory adjunction from the locally presentable context of \cite{BourkeGarner} to the more general context of a small subcategory of arities $\K \hookrightarrow \C$ that is contained in some bounded and eleutheric subcategory of arities $\J \hookrightarrow \C$ \eqref{subcategoryaritiesassumptions}, where $\C$ and $\V$ need only satisfy the weaker hypotheses of the present paper \eqref{Vcategoryassumptions}. We shall then recover the monad-pretheory adjunction of \cite{BourkeGarner} as a special case, since in the locally presentable setting of \cite{BourkeGarner}, \textit{every} small subcategory of arities $\K \hookrightarrow \C$ is contained in a bounded and eleutheric subcategory of arities (\ref{lpres}). Towards these ends, we first prove the following result, now transporting the above notions of pretheory and concrete model to the general setting of the present paper, under just the assumptions of \ref{Vcategoryassumptions}: 

\begin{theo}
\label{BGpresentation}
Let $\C$ and $\V$ be as in \ref{Vcategoryassumptions}, and let $k : \K \hookrightarrow \C$ be a subcategory of arities that is contained in some bounded and eleutheric subcategory of arities $j : \J \hookrightarrow \C$. For any $\K$-pretheory $\tau : \K^\op \to \scrT$, there is a diagrammatic $\J$-presentation $\cP_\scrT$ with $\cP_\scrT\Alg \cong \MOD_c(\scrT)$ in $\V\CAT/\C$, so $\MOD_c(\scrT)$ is a strictly $\J$-monadic $\V$-category over $\C$. 
\end{theo}
 
\begin{proof}
A concrete $\scrT$-model $A$ consists of an object $A$ of $\C$ equipped with a $\V$-functor $M^A : \scrT \to \V$ satisfying certain conditions, including the requirement that $M^AK = \C(K, A)$ for all $K \in \ob\scrT = \ob\K$. In particular, $M^A$ has structural morphisms $M_{\highm{KK'}}^A : \scrT(\highm{K, K'}) \to \V(\highm{\C(K, A), \C(K', A)})$ ($\highm{K, K'} \in \ob\K$), which are equivalently given by morphisms $M_{\highm{KK'}}^A : \highm{\C(K, A) \tensor \scrT(K, K') \tensor K' \to A}$, since $\C$ is tensored. We now define a free-form $\J$-signature $\cS$ consisting of an operation symbol $M_{\highm{KK'}}$ of arity $\highm{K} \in \ob\K \subseteq \ob\J$ and parameter $\highm{\scrT(K, K') \tensor K'} \in \ob\C$ for each pair $\highm{(K, K')} \in \ob\K \times \ob\K$. An $\cS$-algebra is thus an object $A$ of $\C$ equipped with a $\V$-graph morphism $M^A:\scrT \rightarrow \V$ given on objects by $M^AK = \C(K, A)$ ($K \in \ob\scrT = \ob\K$), where we write simply $\scrT$ and $\V$ for the $\V$-graphs underlying $\scrT$ and $\V$.

We now define a diagrammatic $\J$-presentation $\cP_\scrT = (\cS, \sfE)$ for which $\cP_\scrT$-algebras will be precisely the concrete $\scrT$-models. We need diagrammatic $\cS$-equations expressing preservation of composition and identities by a concrete $\scrT$-model. For an $\cS$-algebra $A$ to be a concrete $\scrT$-model, we need the following diagram to commute for each triple \highl{$K, K', K''$} of objects of $\K$:
\[\begin{tikzcd}
	{\highm{\C(K, A) \tensor \scrT(K, K') \tensor \scrT(K', K'') \tensor K''}} && {\highm{\C(K, A) \tensor \scrT(K, K'') \tensor K''}} \\
	&& A \\
	{\highm{\C\left(K', \C(K, A) \tensor \scrT(K, K') \tensor K'\right) \tensor \scrT(K', K'') \tensor K''}} && {\highm{\C(K', A) \tensor \scrT(K', K'') \tensor K''}},
	\arrow["{\mathsf{coev} \tensor 1}"', from=1-1, to=3-1]
	\arrow["{\highm{\C\left(1, M_{KK'}^A\right) \tensor 1}}"', from=3-1, to=3-3]
	\arrow["{\highm{1 \tensor c_{KK'K''} \tensor 1}}", from=1-1, to=1-3]
	\arrow["{\highm{M_{KK''}^A}}", from=1-3, to=2-3]
	\arrow["{\highm{M_{KK''}^A}}"', from=3-3, to=2-3]
\end{tikzcd}\]
where \highl{$c_{KK'K''}$} is a composition morphism of $\scrT$ and $\mathsf{coev}$ is the unit morphism at $\C(\highm{K}, A) \tensor \scrT(\highm{K}, \highm{K'}) \in \ob\V$ of the tensor-hom adjunction $\highm{(-) \tensor K' \dashv \C(K', -)} : \C \to \V$. This diagram describes a diagrammatic $\cS$-equation $\highm{\C\left(K, A\right) \tensor \scrT(K, K') \tensor \scrT(K', K'') \tensor K''} \rightrightarrows A$ ($A \in \cS\Alg$) with arity \highl{$K$} and parameter \highl{$\scrT(K, K') \tensor \scrT(K', K'') \tensor K''$}.  For each object $K$ of $\K$, we also require the commutativity of the following diagram, which describes a diagrammatic $\cS$-equation $\C\left(K, A\right) \tensor K \rightrightarrows A$ ($A \in \cS\Alg$):
% https://q.uiver.app/?q=WzAsNCxbMCwwLCJcXEMoSywgQSkgXFx0ZW5zb3IgSyJdLFsxLDAsIlxcQyhLLCBBKSBcXHRlbnNvciBJIFxcdGVuc29yIEsiXSxbMywwLCJcXEMoSywgQSkgXFx0ZW5zb3IgXFxzY3JUKEssIEspIFxcdGVuc29yIEsiXSxbMywxLCJBIl0sWzAsMSwiXFxzaW0iXSxbMSwyLCIxIFxcdGVuc29yIHVfSyBcXHRlbnNvciAxIl0sWzIsMywiTV97S0t9XkEiXSxbMCwzLCJcXG1hdGhzZntldn0iLDJdXQ==
\[\begin{tikzcd}
	{\C(K, A) \tensor K} & {\C(K, A) \tensor I \tensor K} && {\C(K, A) \tensor \scrT(K, K) \tensor K} \\
	&&& A,
	\arrow["\sim", from=1-1, to=1-2]
	\arrow["{1 \tensor u_K \tensor 1}", from=1-2, to=1-4]
	\arrow["{M_{KK}^A}", from=1-4, to=2-4]
	\arrow["{\mathsf{ev}}"', from=1-1, to=2-4]
\end{tikzcd}\]
where $u_K : I \to \scrT(K, K)$ is an identity morphism of $\scrT$ and $\mathsf{ev}$ is the counit morphism at $A \in \ob\C$ of the tensor-hom adjunction $(-) \tensor K \dashv \C(K, -) : \C \to \V$.

Lastly, we need diagrammatic $\cS$-equations expressing that, for a concrete $\scrT$-model $A$, the $\V$-functor $M^A : \scrT \to \V$ extends $\C(k-, A) : \K^\op \to \V$ along $\tau : \K^\op \to \scrT$. So for an $\cS$-algebra $A$ to be a concrete $\scrT$-model, we need the following diagram to commute for all objects \highl{$K, K'$} of $\K$:
% https://q.uiver.app/?q=WzAsNSxbMCwwLCJcXEMoSiwgQSkgXFx0ZW5zb3IgXFxzY3JLKEssIEopIFxcdGVuc29yIEsiXSxbMywwLCJcXEMoSiwgQSkgXFx0ZW5zb3IgXFxzY3JUKEosIEspIFxcdGVuc29yIEsiXSxbMywxLCJBIl0sWzAsMSwiXFxDKEssIEopIFxcdGVuc29yIFxcQyhKLCBBKSBcXHRlbnNvciBLIl0sWzIsMSwiXFxDKEssIEEpIFxcdGVuc29yIEsiXSxbMCwxLCIxIFxcdGVuc29yIFxcdGF1X3tKS30gXFx0ZW5zb3IgMSJdLFsxLDIsIk1fe0pLfV5BIl0sWzAsMywiXFx3ciIsMl0sWzMsNCwiY197S0pBfSBcXHRlbnNvciAxIiwyXSxbNCwyLCJcXG1hdGhzZntldn0iLDJdXQ==
\[\begin{tikzcd}
	{\highm{\C(K, A) \tensor \K(K', K) \tensor K'}} &&& {\highm{\C(K, A) \tensor \scrT(K, K') \tensor K'}} \\
	{\highm{\C(K', K) \tensor \C(K, A) \tensor K'}} && {\highm{\C(K', A) \tensor K'}} & A,
	\arrow["{1 \tensor \tau_{\highm{KK'}} \tensor 1}", from=1-1, to=1-4]
	\arrow["{M_{\highm{KK'}}^A}", from=1-4, to=2-4]
	\arrow["\wr"', from=1-1, to=2-1]
	\arrow["{c_{\highm{K'K}A} \tensor 1}"', from=2-1, to=2-3]
	\arrow["{\mathsf{ev}}"', from=2-3, to=2-4]
\end{tikzcd}\]
where $c_{\highm{K'K}A}$ is a composition morphism of $\C$. This diagram describes a diagrammatic $\cS$-equation $\highm{\C\left(K, A\right) \tensor \K(K', K) \tensor K'} \rightrightarrows A$ ($A \in \cS\Alg$). We thus obtain a diagrammatic $\J$-presentation $\cP_\scrT$ for which $\cP_\scrT$-algebras are in bijective correspondence with concrete $\scrT$-models, and it is straightforward to verify (using \ref{universal_property_PAlg} and the definition of $\mathsf{Mod}_c(\scrT)$ as a pullback in $\V\CAT$) that this bijection on objects extends to an isomorphism $\cP_\scrT\Alg \cong \mathsf{Mod}_c(\scrT)$ in $\V\CAT/\C$.  Hence $\MOD_c(\scrT)$ is strictly $\J$-monadic over $\C$ by \ref{vmnd_pres_by_ff_jpres}.
\end{proof}

\noindent In the current general context, where $\C$ and $\V$ satisfy the background hypotheses of \ref{Vcategoryassumptions} and $\K \hookrightarrow \C$ is a small subcategory of arities, we can again define a functor $\Phi : \Mnd(\C) \to \Preth_\K(\C)$ exactly as described at the beginning of this section.  Moreover, under the hypotheses of \ref{BGpresentation}, we can prove the following generalization of \cite[Theorem 6]{BourkeGarner}:

\begin{theo}
\label{adjunctiontheo}
Let $\C$ and $\V$ be as in \ref{Vcategoryassumptions}, and let $k : \K \hookrightarrow \C$ be a subcategory of arities that is contained in some bounded and eleutheric subcategory of arities $j : \J \hookrightarrow \C$. Then $\Phi : \Mnd(\C) \to \Preth_\K(\C)$ has a left adjoint $\Psi : \Preth_\K(\C) \to \Mnd(\C)$ that sends each $\K$-pretheory $\scrT$ to a $\V$-monad $\Psi(\scrT)$ such that $\Psi(\scrT)\Alg \cong \MOD_c(\scrT)$ in $\V\CAT/\C$.  Furthermore, $\Psi(\scrT)$ is a $\J$-ary $\V$-monad for each $\K$-pretheory $\scrT$.
\end{theo}

\begin{proof}
By \cite[Theorem 2]{BourkeGarner}, which (as Bourke and Garner emphasize) does \emph{not} require the local presentability of $\C$ or $\V$, and (hence) still holds under the more general hypotheses of \ref{Vcategoryassumptions}, \textit{if} the forgetful $\V$-functor $U^\scrT : \MOD_c(\scrT) \to \C$ has a left adjoint for each $\K$-pretheory $\scrT$, \textit{then} $\Phi$ has a left adjoint $\Psi$ that sends $\scrT$ to the $\V$-monad $\Psi(\scrT)$ induced by this adjunction. But, under the present hypotheses, $U^\scrT$ is strictly $\J$-monadic by Theorem \ref{BGpresentation}.
\end{proof}

\noindent Theorem \ref{adjunctiontheo} significantly generalizes \cite[Theorem 6]{BourkeGarner} by removing the assumptions that $\C$ and $\V$ are locally presentable.  Thus, Theorem \ref{adjunctiontheo} is applicable to all of the examples of \ref{runningexamples}, in most of which $\V$ need not be locally presentable.  In view of the discussion preceding \ref{BGpresentation}, the hypotheses of Theorem \ref{adjunctiontheo} are automatically satisfied in the locally presentable context by taking $\J = \C_\beta$ for a suitable regular cardinal $\beta$, so that as a special case of \ref{adjunctiontheo} we recover \cite[Theorem 6]{BourkeGarner}: \textit{If $\K \hookrightarrow \C$ is a small subcategory of arities in a locally presentable $\V$-category $\C$ over a locally presentable closed category $\V$, then $\Phi : \Mnd(\C) \to \Preth_\K(\C)$ has a left adjoint $\Psi$ such that $\Psi(\scrT)\Alg \cong \MOD_c(\scrT)$ in $\V\CAT/\C$ for each $\K$-pretheory $\scrT$.}

\begin{defn}
\label{Knervous}
{
Let $\C$ and $\V$ be as in \ref{Vcategoryassumptions}, and let $\K \hookrightarrow \C$ be a small subcategory of arities. A $\V$-monad $\T$ on $\C$ is \textbf{$\K$-nervous} if $\T \cong \Psi(\scrT)$ for some $\K$-pretheory $\scrT$. \qed
}
\end{defn}

\noindent This is not the original definition of $\K$-nervous $\V$-monad given in \cite[Definition 17]{BourkeGarner}, but it is equivalent to their definition (in their locally presentable context) by \cite[Corollary 21]{BourkeGarner}.  

\begin{cor}
\label{KnervousJary}
Let $\C$ and $\V$ be as in \ref{Vcategoryassumptions}, and let $\K \hookrightarrow \C$ be a subcategory of arities that is contained in some bounded and eleutheric subcategory of arities $\J \hookrightarrow \C$. Then every $\K$-nervous $\V$-monad $\T$ is $\J$-ary.  
\end{cor}

\begin{proof}
There is some $\K$-pretheory $\scrT$ with $\T \cong \Psi(\scrT)$, and $\Psi(\scrT)$ is $\J$-ary by \ref{adjunctiontheo}.
\end{proof}

\begin{cor}
\label{Knervousaccessible}
Let $\K \hookrightarrow \C$ be a small subcategory of arities in a locally $\alpha$-presentable $\V$-category $\C$ over a locally $\alpha$-presentable closed category $\V$. Then there is a regular cardinal $\beta \geq \alpha$ such that every $\K$-nervous $\V$-monad on $\C$ is a $\beta$-ary $\V$-monad (\ref{runningexamples}).
\end{cor}

\begin{proof}
In view of \ref{lpres}, this follows from \ref{KnervousJary}. 
\end{proof}

\section{A Birkhoffian Galois connection for \texorpdfstring{$\J$}{J}-ary varieties enriched in $\V$}
\label{Galoisconnection}

Throughout this final section, we fix a free-form $\J$-signature $\cS$.\footnote{Although we have chosen to focus on a free-form $\J$-signature $\cS$ and the $\J$-ary $\V$-monad $\T_\cS$ generated by $\cS$ \eqref{salg_jmonadic}, the results of this section also hold when replacing $\T_\cS$ by a general $\J$-ary $\V$-monad $\T$.}  If $\cP$ is a free-form (or, equivalently, diagrammatic) $\J$-presentation of the form $\cP = (\cS,\sfE)$, then we say that $\cP$ is a \textbf{free-form $\J$-presentation over $\cS$}, and we call $\cP\Alg$ a \textbf{$\J$-ary variety over $\cS$} (cf.~\ref{Jaryvariety}).  Hence, a $\J$-ary variety over $\cS$ is equivalently given by a (possibly large) \textit{set} of $\cS$-algebras of the form $\ob(\cP\Alg)$ for a free-form $\J$-presentation $\cP$ over $\cS$, and we refer to such sets also as $\J$-ary varieties over $\cS$.

For the remainder of the section, we also fix a generating set $\G$ for $\C_0$, i.e. a set of objects $\G \subseteq \ob\C$ such that the functors $\C_0(G, -) : \C_0 \to \Set$ ($G \in \G$) are jointly faithful; we do not assume that $\G$ is small (e.g. we allow $\G = \ob\C$). Finally, we assume for this section that (in addition to the assumptions of \ref{Vcategoryassumptions}) the category $\C_0$ is complete, so that the $\V$-category $\C$ is complete (being tensored and cotensored).

In this section, we establish and study a Galois connection between sets of $\cS$-algebras and sets of \emph{$\G$-parametrized $\cS$-equations} \eqref{powerclasses} whose fixed points on one side are precisely the $\J$-ary varieties over $\cS$, thus generalizing the classical Galois connection of Birkhoff \cite{Birkhoff:Lattice}. We shall conclude the section by investigating conditions under which a $\J$-ary variety over $\cS$ is generated by a given $\cS$-algebra or a given set of $\cS$-algebras.  

\begin{defn}
\label{powerclasses}
{
A \textbf{$\G$-parametrized $\cS$-equation} is a parametrized $\cS$-equation $t \doteq u : G \rightrightarrows T_\cS J$ \eqref{param_Seqn} whose parameter $G$ is an object of $\G$. We then have the (possibly large) set $\Eqn_\G(\cS)$ of $\G$-parametrized $\cS$-equations, which may also be described as the disjoint union \[ \Eqn_\G(\cS) \;\;= \coprod_{J \in \ob\J,\;G \in \G} \C_0\left(G, T_\cS J\right)^2. \] We let $\Pow\left(\Eqn_\G(\cS)\right)$ be the (possibly large) set of all subsets of $\Eqn_\G(\cS)$ and $\Pow(\cS\Alg)$ the (possibly large) set of all subsets of $\ob\left(\cS\Alg\right)$. \qed 
} 
\end{defn}

\begin{para}
\label{class_vs_Jpres}
Note that if $\G = \ob\C$, then an $\ob\C$-parametrized $\cS$-equation is just a parametrized $\cS$-equation \eqref{param_Seqn}.  We write
$$\Eqn(\cS) := \Eqn_{\ob\C}(\cS)$$
to denote the set of \textit{all} parametrized $\cS$-equations.  Note also that if $\calE \subseteq \Eqn(\cS)$ is small, then $\calE$ may be regarded as a system of parametrized $\J$-ary equations over $\T_\cS$ \eqref{defn_ff_sys_jary_eqns}, so that $\cP = (\cS, \calE)$ is a free-form $\J$-presentation \eqref{ff_Jpres}. \qed
\end{para}

\begin{defn}
\label{varietytheory}
We now introduce some notation related to the satisfaction relation $\models$ for $\cS$-algebras and parametrized $\cS$-equations (\ref{interp_in_salg}), recalling that $\models$ is a binary relation from $\ob\left(\cS\Alg\right)$ to $\Eqn(\cS)$.  If $\calA \in \Pow(\cS\Alg)$ and $\calE \in \Pow(\Eqn(\cS))$, then we write $\calA \models \calE$ to mean that $A \models t \doteq u$ \eqref{interp_in_salg} for all $A \in \calA$ and all $(t, u) \in \calE$. If $\calA = \{A\}$, then we write $A \models \calE$ rather than $\{A\} \models \calE$, and if $\calE = \{(t, u)\}$, then we write $\calA \models t \doteq u$ rather than $\calA \models \{(t, u)\}$. Given a (possibly large) set $\calE$ of parametrized $\cS$-equations, we write
\begin{equation}\label{sat_upper_star} \Sat^*(\calE) \;\;:=\;\; \left\{ A \in \ob(\cS\Alg) \mid A \models \calE\right\}.\end{equation}
Given a (possibly large) set of $\cS$-algebras $\calA$, we write
\[ \Sat_*(\mathcal{A}) \;\;:=\;\; \left\{ (t, u) \in \Eqn(\cS) \mid \mathcal{A} \models t \doteq u\right\},\]
\begin{equation}\label{eq:models_lowerstarg} \Sat_*^\G(\mathcal{A}) \;\;:=\;\; \left\{ (t, u) \in \Eqn_\G(\cS) \mid \mathcal{A} \models t \doteq u\right\}.\end{equation}
\qed
\end{defn}

\begin{rmk}\label{rmk:basic_var_charn}
In view of \ref{class_vs_Jpres}, a set of $\cS$-algebras $\calA$ is a $\J$-ary variety over $\cS$ if and only if $\calA = \Sat^*(\calF)$ for some \textit{small} set $\calF$ of parametrized $\cS$-equations. \qed
\end{rmk}

\noindent As far as the satisfaction relation $\models$ is concerned, sets of parametrized $\cS$-equations can be represented faithfully as sets of \textit{$\G$-parametrized $\cS$-equations}, by way of the following construction:

\begin{defn}
\label{E_G}
Given a parametrized $\cS$-equation $t \doteq u:C \rightrightarrows T_\cS J$, we write $[t \doteq u]_\G$ to denote the set of all $\G$-parametrized $\cS$-equations of the form $t \circ x \doteq u \circ x:G \rightrightarrows T_\cS J$ where $G$ is an object in $\G$ and $x : G \to C$ is a morphism in $\C$.  Given instead a set $\calE$ of parametrized $\cS$-equations, we define $[\calE]_\G$ to be the union of the sets $[t \doteq u]_\G$ associated to the parametrized $\cS$-equations $t \doteq u$ in $\calE$. \qed
\end{defn}

\begin{prop}
\label{factor_through_satisf}
Let $\calA$ be a set of $\cS$-algebras, and let $\calE$ be a set of parametrized $\cS$-equations. Then $\calA \models \calE$ if and only if $\calA \models [\calE]_\G$.
\end{prop}
\begin{proof}
Since $\G$ is a generating set for $\C_0$, this follows readily from \ref{interp_morph_vs_interp_of_t}.
\end{proof}

\begin{cor}
\label{variety_equality}
Let $\calE$ be a set of parametrized $\cS$-equations. Then $\Sat^*(\calE) = \Sat^*\bigl([\calE]_\G\bigr)$. \qed
\end{cor}

\begin{rmk}\label{recover_satstar_from_satstarg}
Given a set $\calA$ of $\cS$-algebras, each of the sets $\Sat_*(\calA)$ and $\Sat_*^\G(\calA)$ can be expressed in terms of the other, as follows: firstly, $\Sat_*^\G(\calA)$ is clearly the intersection of $\Sat_*(\calA)$ with $\Eqn_\G(\cS)$, and secondly, it follows from \ref{factor_through_satisf} that the $\Sat_*(\calA)$ is the set of all parametrized $\cS$-equations $t \doteq u$ such that $\calA \models [t \doteq u]_\G$, i.e. such that 
$[t \doteq u]_\G \subseteq \Sat_*^\G(\calA)$. \qed
\end{rmk}

\noindent Regarding $\Pow(\Eqn_\G(\cS))$ and $\Pow(\cS\Alg)$ as (possibly large) partially ordered sets under inclusion, the following result now follows immediately from the definitions and is an instance of the well-known fact that any relation between sets induces a Galois connection between their power sets:

\begin{prop}
\label{Galoisconnectionprop}
There is a Galois connection
% https://q.uiver.app/?q=WzAsMixbMCwwLCJcXFBvdyhcXEVxbl9cXEcoXFxTaWdtYSkpIl0sWzMsMCwiXFxQb3coXFxTaWdtYVxcQWxnKV5cXG9wIl0sWzAsMSwiXFxtYXRoYmJ7Vn0iLDAseyJjdXJ2ZSI6LTJ9XSxbMSwwLCJcXFRoZXRhIiwwLHsiY3VydmUiOi0yfV0sWzMsMiwiIiwwLHsibGV2ZWwiOjEsInN0eWxlIjp7Im5hbWUiOiJhZGp1bmN0aW9uIn19XV0=
\[\begin{tikzcd}
	{\Pow(\Eqn_\G(\cS))} &&& {\Pow(\cS\Alg)^\op}
	\arrow[""{name=0, anchor=center, inner sep=0}, "\Sat^*", curve={height=-12pt}, from=1-1, to=1-4]
	\arrow[""{name=1, anchor=center, inner sep=0}, "\Sat_*^\G", curve={height=-12pt}, from=1-4, to=1-1]
	\arrow["\vdash"{anchor=center, rotate=90}, draw=none, from=1, to=0]
\end{tikzcd}\]
in which $\Sat^*$ and $\Sat_*^\G$ are given by \eqref{sat_upper_star} and \eqref{eq:models_lowerstarg}, respectively.  That is, $\Sat^*$ and $\Sat_*^\G$ are order-reversing maps, and for all $\calA \in \Pow(\cS\Alg)$ and $\calE \in \Pow(\Eqn_\G(\cS))$ we have \[ \calA \subseteq \Sat^*(\calE) \  \Longleftrightarrow \ \calE \subseteq \Sat_*^\G(\calA). \tag*{\qed} \]
\end{prop}

\noindent In the case where $\G = \ob\C$, \ref{Galoisconnectionprop} specializes to yield a Galois connection $\Sat^* \dashv \Sat_*:\Pow(\cS\Alg)^\op \rightarrow \Pow(\Eqn(\cS))$.

We recover the classical Galois connection of Birkhoff \cite{Birkhoff:Lattice} from \ref{Galoisconnectionprop} by taking $\C = \V = \Set$, $\J = \SF(\Set)$ (the finite cardinals), and $\G = \{1\}$.  In the classical setting, varieties over a given signature $\cS$ are the fixed points of the idempotent monad on $\Pow(\cS\Alg)$ induced by this Galois connection.  The same is true in our current general setting \textit{if the generator $\G$ is small} (\ref{small_gen_gcx}).  But in general we do not assume that $\G$ is small, and while the fixed points of the idempotent monad on $\Pow(\cS\Alg)$ induced by \ref{Galoisconnectionprop} are those of the form $\Sat^*(\calE)$ for a set of $\G$-parametrized $\cS$-equations $\calE$, the set $\calE$ need not be small, so in view of \ref{rmk:basic_var_charn} it is not clear that ${\Sat}^*(\calE)$ would be a $\J$-ary variety.  Indeed, we intend to accommodate, for example, the case where $\G = \ob\C$.

For this reason, we now define certain mild size conditions on sets of  parametrized $\cS$-equations and $\cS$-algebras, and we proceed to show that the Galois connection of \ref{Galoisconnectionprop} restricts to sets satisying these conditions:

\begin{defn}
\label{tame}
A set $\mathcal{E}$ of parametrized $\cS$-equations is \textbf{tame} if there is a \emph{small} set $\calF$ of parametrized $\cS$-equations such that $\Sat^*(\mathcal{E}) = \Sat^*\left(\calF\right)$. A set $\mathcal{A}$ of $\cS$-algebras is \textbf{tame} if there is a \emph{small} set $\mathcal{B}$ of $\cS$-algebras such that $\Sat_*(\mathcal{A}) = \Sat_*\left(\mathcal{B}\right)$. \qed
\end{defn}

\begin{rmk}\label{rmk:var_charn_tame}
In view of \ref{rmk:basic_var_charn}, a set $\calE$ of parametrized $\cS$-equations is tame iff $\Sat^*(\calE)$ is a $\J$-ary variety over $\cS$. \qed
\end{rmk}

\noindent Note that Definition \ref{tame} does not involve $\G$ at all.  But \ref{recover_satstar_from_satstarg} immediately entails the following characterization of tame sets of $\cS$-algebras in terms of \textit{$\G$-parametrized} $\cS$-equations:

\begin{prop}\label{gtame_iff_tame}
A set $\calA$ of $\cS$-algebras is tame if and only if there is a small set $\mathcal{B}$ of $\cS$-algebras with $\Sat_*^\G(\mathcal{A}) = \Sat_*^\G\left(\mathcal{B}\right)$.
\end{prop}

\begin{lem}
\label{square_brackets_small_tame}
If $\calF$ is a small set of parametrized $\cS$-equations, then $[\calF]_\G$ is a tame set of $\G$-parametrized $\cS$-equations.
\end{lem}
\begin{proof}
This follows immediately from \ref{variety_equality}.
\end{proof}

\begin{lem}
\label{presentationtamelem}
Let $\mathcal{A}$ be a small set of $\cS$-algebras. Then $\Sat_*^\G(\calA)$ is a tame set of parametrized $\cS$-equations.
\end{lem} 

\begin{proof}
For each $J \in \ob\J$, let $\bar{\ii}_J : T_\cS J \to \prod_{A \in \mathcal{A}} \ \langle A, A\rangle J$ be the morphism in $\C$ induced by the interpretation morphisms $\ii_J^A :  T_\cS J \to \langle A, A\rangle J$ ($A \in \calA$) of \ref{algebrasforpresentation}, where we regard each $A \in \calA$ equivalently as a $\T_\cS$-algebra (\ref{salg_jmonadic}), and let $p^1_J, p^2_J : C_J \rightrightarrows T_\cS J$ be the kernel pair of $\bar{\ii}_J$. Now let $\calF$ be the set consisting of all the parametrized $\cS$-equations $p^1_J \doteq p^2_J$ with $J \in \ob\J$.  Then $\calF$ is small, so by \ref{square_brackets_small_tame} it suffices to show that $\Sat_*^\G(\calA) = [\calF]_\G$. For each $\G$-parametrized $\cS$-equation $t, u : G \rightrightarrows T_\cS J$ ($J \in \ob\J, G \in \G$), we have that $(t, u) \in [\calF]_\G$ iff $t, u$ factors through the kernel pair of $\bar{\ii}_J$, iff $\bar{\ii}_J \circ t = \bar{\ii}_J \circ u$, iff $\ii^A_J \circ t = \ii^A_J \circ u$ for all $A \in \calA$, iff $A \models t \doteq u$ for all $A \in \calA$ (by \ref{interp_morph_vs_interp_of_t}), iff $(t, u) \in \Sat_*^\G(\calA)$, as desired.          
\end{proof}

\begin{prop}
\label{tameprop1}
Let $\calA$ be a tame set of $\cS$-algebras. Then $\Sat_*^\G(\mathcal{A})$ is a tame set of parametrized $\cS$-equations.
\end{prop}

\begin{proof}
By \ref{gtame_iff_tame}, there is a small set $\mathcal{B}$ of $\cS$-algebras with $\Sat_*^\G(\mathcal{A}) = \Sat_*^\G\left(\mathcal{B}\right)$, and the latter is tame by  \ref{presentationtamelem}.  
\end{proof} 

\begin{para}\label{prelims_free_alg_satis_prop}
Let $\cP = (\cS,\sfE)$ be a free-form $\J$-presentation over $\cS$, let $\T_\cP$ be the $\J$-ary $\V$-monad presented by $\cP$ (\ref{vmnd_pres_by_ff_jpres}), and let $q : \T_\cS \to \T_\cP$ be the regular epimorphism in $\Mnd_{\underJ}(\C)$ that presents $\T_\cP = \T_\cS / \sfE$ as a quotient (\ref{algebrasforpresentation}, \ref{ae_jary_monadic}, \ref{vmnd_pres_by_ff_jpres}).   Then each $\T_\cP$-algebra $A = (A,a)$ corresponds via  \ref{vmnd_pres_by_ff_jpres} to a $\cP$-algebra that we denote also by $A$, but then $A$ is in particular an $\cS$-algebra and corresponds via \ref{salg_jmonadic} to the $\T_\cS$-algebra $(A,a \circ q_A)$, by \ref{algebrasforpresentation}.  Hence, \ref{teqn_equalized_by_mnd_morph} immediately entails the following:
\end{para}

\begin{prop}
\label{free_alg_satis_prop}
In the situation of \ref{prelims_free_alg_satis_prop}, let $t, u : C \rightrightarrows T_\cS J$ be a parametrized $\cS$-equation (where $J \in \ob\J$ and $C \in \ob\C$). Then the following are equivalent: (1) every $\cP$-algebra satisfies $t \doteq u$; (2) every free $\cP$-algebra on an object of $\J$ satisfies $t \doteq u$; (3) the free $\cP$-algebra $T_\cP J$ on $J$ satisfies $t \doteq u$; (4) $q_J \circ t = q_J \circ u$.
\end{prop}

\begin{cor}
\label{presentationtheoryprop}
Let $\cP$ be a free-form $\J$-presentation over $\cS$. Then $\Sat_*^\G(\ob(\cP\Alg)) = \Sat_*^\G\left(\{T_\cP J \mid J \in \ob\J\}\right)$. \qed
\end{cor}

\noindent Since $\J$ is small, \ref{presentationtheoryprop} entails the following (in view of \ref{gtame_iff_tame}):

\begin{cor}
\label{Jary_variety_tame}
Every $\J$-ary variety over $\cS$ is a tame set of $\cS$-algebras. 
\end{cor}

\begin{cor}
\label{tameprop2}
Let $\calE$ be a tame set of parametrized $\cS$-equations. Then $\Sat^*(\calE)$ is a tame set of $\cS$-algebras (as it is a $\J$-ary variety over $\cS$, by \ref{rmk:var_charn_tame}). \qed 
\end{cor}

\begin{defn}\label{defn_v_theta}
If $\calE$ is a tame set of parametrized $\cS$-equations, then we call $$\bbV(\calE) := \Sat^*(\calE)$$
the \textbf{variety described by $\calE$}.  If $\calA$ is a tame set of $\cS$-algebras, then we call
$$\Theta_\G(\calA) := \Sat_*^\G(\calA)$$
the \textbf{equational theory of $\calA$} (over $\cS$ with parameters in $\G$).  An \textbf{equational theory over $\cS$} (with parameters in $\G$) is a set $\calE$ of $\G$-parametrized $\cS$-equations such that $\calE = \Theta_\G(\calA)$ for some tame set $\calA$ of $\cS$-algebras (noting that $\calE$ itself is then tame by \ref{tameprop1}). \qed
\end{defn}

\noindent From \ref{tameprop1} and \ref{tameprop2} we now immediately obtain the following result.

\begin{theo}
\label{tame_Galois_connection}
The Galois connection of \ref{Galoisconnectionprop} restricts to a Galois connection between the poset $\Pow(\Eqn_\G(\cS))_{\mathsf{tame}}$ of all tame sets of $\G$-parametrized $\cS$-equations and the poset $\Pow(\cS\Alg)_{\mathsf{tame}}$ of all tame sets of $\cS$-algebras:
% https://q.uiver.app/?q=WzAsMixbMCwwLCJcXFBvdyhcXEVxbl9cXEcoXFxTaWdtYSkpIl0sWzMsMCwiXFxQb3coXFxTaWdtYVxcQWxnKV5cXG9wIl0sWzAsMSwiXFxtYXRoYmJ7Vn0iLDAseyJjdXJ2ZSI6LTJ9XSxbMSwwLCJcXFRoZXRhIiwwLHsiY3VydmUiOi0yfV0sWzMsMiwiIiwwLHsibGV2ZWwiOjEsInN0eWxlIjp7Im5hbWUiOiJhZGp1bmN0aW9uIn19XV0=
\[\begin{tikzcd}
	{\Pow(\Eqn_\G(\cS))_{\mathsf{tame}}} &&& {\Pow(\cS\Alg)_{\mathsf{tame}}^\op}
	\arrow[""{name=0, anchor=center, inner sep=0}, "{\bbV}", curve={height=-12pt}, from=1-1, to=1-4]
	\arrow[""{name=1, anchor=center, inner sep=0}, "\Theta_\G", curve={height=-12pt}, from=1-4, to=1-1]
	\arrow["\vdash"{anchor=center, rotate=90}, draw=none, from=1, to=0]
\end{tikzcd} \tag*{\qed}\]
\end{theo}

\noindent We now show that the $\J$-ary varieties over $\cS$ are precisely the fixed points on one side of the Galois connection \eqref{tame_Galois_connection}:

\begin{theo}
\label{Jary_variety_fixedpoint}
Let $\calA$ be a (possibly large) set of $\cS$-algebras. Then the following are equivalent: (1) $\calA$ is a $\J$-ary variety over $\cS$; (2) $\calA$ is tame and $\calA = \bbV(\Theta_\G(\calA))$; (3) there is a tame set $\calE$ of $\G$-parametrized $\cS$-equations such that $\calA = \bbV(\calE)$. 
\end{theo}

\begin{proof}
(2) and (3) are equivalent by \ref{tame_Galois_connection}, and the standard theory of Galois connections, and (3) implies (1) by \ref{rmk:var_charn_tame}. If (1) holds, then there is a small set $\calF$ of parametrized $\cS$-equations with $\calA = \Sat^*(\calF) = \Sat^*([\calF]_\G)$, by \ref{rmk:basic_var_charn} and \ref{variety_equality}, but $[\calF]_\G$ is a tame set of $\G$-parametrized $\cS$-equations since $\calF$ is small (\ref{square_brackets_small_tame}).
\end{proof}

\noindent Invoking \ref{tame_Galois_connection} and \ref{Jary_variety_fixedpoint} in the case where $\G = \ob\C$, we obtain the following:

\begin{cor}\label{galois_cx_obc}
Writing $\Theta := \Theta_{\ob\C}$, there is a Galois connection
$$\bbV \dashv \Theta:\Pow(\cS\Alg)_{\mathsf{tame}}^\op \rightarrow \Pow(\Eqn(\cS))_{\mathsf{tame}}\;,$$
and the following are equivalent for a set of $\cS$-algebras $\calA$: (1) $\calA$ is a $\J$-ary variety over $\cS$; (2) $\calA$ is tame and $\calA = \bbV(\Theta(\calA))$; (3) there is a tame set $\calE$ of parametrized $\cS$-equations such that $\calA = \bbV(\calE)$. \qed
\end{cor}

\begin{para}
\label{closureoperators}
The Galois connection of \ref{tame_Galois_connection} induces a closure operator (i.e. an idempotent monad) $\bbV\Theta_\G$ on $\Pow(\cS\Alg)_{\mathsf{tame}}$ whose fixed points are precisely the $\J$-ary varieties over $\cS$, by \ref{Jary_variety_fixedpoint}.  Hence, given a tame set $\calA$ of $\cS$-algebras, $\bbV\Theta_\G(\calA)$ is the smallest $\J$-ary variety over $\cS$ that contains $\calA$.  As a special case, similar conclusions are obtained with regard to the idempotent monad $\bbV\Theta$ on $\Pow(\cS\Alg)_{\mathsf{tame}}$ obtained via \ref{galois_cx_obc}, and we thus deduce that the monads $\bbV\Theta$ and $\bbV\Theta_\G$ are identical.  For each tame set $\calA$ of $\cS$-algebras, we call $\bbV\Theta_\G(\calA) = \bbV\Theta(\calA)$ the \textbf{$\J$-ary variety over $\cS$ generated by $\calA$}.  We may also consider the idempotent monad $\Theta_\G\bbV$ on $\Pow(\Eqn_\G(\cS))_{\mathsf{tame}}$, whose fixed points are the equational theories of \ref{defn_v_theta}.  Given a tame set $\calE$ of $\G$-parametrized $\cS$-equations, we call $\Theta_\G\bbV(\calE)$ \textbf{the equational theory over $\cS$ generated by $\calE$} (with parameters in $\G$), as it is the smallest equational theory over $\cS$ that contains $\calE$. \qed
\end{para}

\noindent If the generating set $\G$ of $\C_0$ is small, then we can show that \textit{every} set of parametrized $\cS$-equations is tame, and that \textit{every} set of $\cS$-algebras is tame.  We begin with the following observation:

\begin{lem}
\label{smallgeneratorbasiclem}
Suppose that $\G$ is small.  Then every set of $\G$-parametrized $\cS$-equations is small.
\end{lem}
\begin{proof}
$\Eqn_\G(\cS)$ is small since $\G$ is small, $\J$ is small, and $\C_0$ is locally small, and the result follows.
\end{proof}

\begin{prop}
\label{smallgeneratorlem}
Suppose that $\G$ is small. Then every set $\calE$ of parametrized $\cS$-equations is tame.
\end{prop}
\begin{proof}
By \ref{variety_equality}, $[\calE]_\G$ is a set of $\G$-parametrized $\cS$-equations with $\Sat^*(\calE) = \Sat^*\left([\calE]_\G\right)$, and $[\calE]_\G$ is small by \ref{smallgeneratorbasiclem}.
\end{proof}

\begin{lem}\label{tameness_reflection_lemma}
Let $\calA$ be a set of $\cS$-algebras, and suppose that $\Sat_*^\G(\calA)$ is a tame set of parametrized $\cS$-equations.  Then $\calA$ is tame.
\end{lem}
\begin{proof}
By \ref{tameprop2}, $\Sat^*\Sat_*^\G(\calA)$ is a tame set of $\cS$-algebras, so by \ref{gtame_iff_tame} there is a small set $\mathcal{B}$ of $\cS$-algebras such that $\Sat_*^\G\Sat^*\Sat_*^\G(\calA) = \Sat_*^\G(\mathcal{B})$.  But $\Sat_*^\G\Sat^*\Sat_*^\G(\calA) = \Sat_*^\G(\calA)$ by \ref{tame_Galois_connection} and a general property of Galois connections, and the result follows, by \ref{gtame_iff_tame}.
\end{proof}

\begin{prop}
\label{smallgeneratortame}
Suppose that $\G$ is small. Then every set $\calA$ of $\cS$-algebras is tame.
\end{prop}

\begin{proof}
By \ref{smallgeneratorlem},
$\Sat_*^\G(\calA)$ is a tame set of parametrized $\cS$-equations, so $\calA$ is tame by \ref{tameness_reflection_lemma}.
\end{proof}

\begin{theo}\label{small_gen_gcx}
Suppose that $\C_0$ has a small generating set.  Then, for any (possibly large) generating set $\G$ for $\C_0$, the Galois connections of \ref{Galoisconnectionprop} and of \ref{tame_Galois_connection} are identical and so may be written as $\bbV \dashv \Theta_\G:\Pow(\cS\Alg)^\op \rightarrow \Pow(\Eqn_\G(\cS))$.  With this notation, the following are equivalent, for a (possibly large) set $\calA$ of $\cS$-algebras: (1) $\calA$ is a $\J$-ary variety over $\cS$; (2) $\calA = \bbV(\Theta_\G(\calA))$; (3) there is a (possibly large) set $\calE$ of $\G$-parametrized $\cS$-equations such that $\calA = \bbV(\calE)$.
\end{theo}
\begin{proof}
By hypothesis, $\C_0$ has some small generating set $\mathscr{H}$, which need not coincide with $\G$. Nevertheless, we may apply \ref{smallgeneratorlem} and \ref{smallgeneratortame} with respect to $\mathscr{H}$ in order to deduce that every set of parametrized $\cS$-equations is tame and that every set of $\cS$-algebras is tame.  Hence, the result now follows from \ref{Jary_variety_fixedpoint}.
\end{proof}

\begin{cor}
Suppose that $\C_0$ has a small generating set.  Then there is a Galois connection $\bbV \dashv \Theta:\Pow(\cS\Alg)^\op \rightarrow \Pow(\Eqn(\cS))$ for which the fixed points of the associated idempotent monad $\bbV\Theta$ on $\Pow(\cS\Alg)$ are precisely the $\J$-ary varieties over $\cS$.
\end{cor}
\begin{proof}
This follows from \ref{small_gen_gcx} in the case where $\G = \ob\C$.
\end{proof}

\begin{defn}\label{defn:gen}
Let $\calA$ be a tame set of $\cS$-algebras.  We say that a $\J$-ary variety $\calB$ over $\cS$ is \textbf{generated by} $\calA$ if $\calB = \bbV\Theta(\calA)$ (equivalently, if $\calB = \bbV\Theta_\G(\calA)$, \ref{closureoperators}).  As a special case, given an $\cS$-algebra $A$, we say that the $\J$-ary variety $\calB$ over $\cS$ is \textbf{generated by} $A$ if it is generated by $\{A\}$. \qed
\end{defn}

\noindent With this terminology, \ref{presentationtheoryprop} entails the following:

\begin{cor}\label{var_gen_jgen_free}
Every $\J$-ary variety $\calB = \ob(\cP\Alg)$ over $\cS$ is generated by the set $\{T_\cP J \mid J \in \ob\J\}$ consisting of the free $\cP$-algebras $T_\cP J$ on objects $J$ of $\J$. \qed
\end{cor}

\noindent We shall conclude this section with some examples of $\J$-ary varieties generated by a single $\cS$-algebra, for which we shall need the following results:

\begin{lem}\label{cl_prop_var}
Let $\cP$ be a free-form $\J$-presentation over $\cS$.  Then (1) $\cP\Alg$ is closed under (weighted) limits in $\cS\Alg$, and (2) $\cP\Alg$ is closed under subobjects in $\cS\Alg$, in the sense that if $m:A \rightarrow B$ is a monomorphism in $\cS\Alg$ and $B$ is a $\cP$-algebra, then $A$ is a $\cP$-algebra.
\end{lem}
\begin{proof}
(1) follows from the fact that both $U^\cS:\cS\Alg \rightarrow \C$ and $U^\cP:\cP\Alg \rightarrow \C$ are strictly monadic $\V$-functors (by \ref{salg_jmonadic} and \ref{vmnd_pres_by_ff_jpres}) and so preserve and create all limits.  If $m:A \rightarrow B$ is a monomorphism in $\cS\Alg$ and $B$ satisfies a given parametrized $\cS$-equation $t \doteq u$, then it is straightforward to verify that $A$ satisfies $t \doteq u$, using the naturality of $\llb t\rrb$ and $\llb u\rrb$.  (2) now follows.
\end{proof}

\begin{prop}\label{sufficient_cond_gen_var}
Let $\cP$ be a free-form $\J$-presentation over $\cS$, and let $\calA$ be a tame set of $\cP$-algebras.  Suppose that for each $J \in \ob\J$, the free $\cP$-algebra $T_\cP J$ is a subobject of some (weighted) limit, in $\cS\Alg$, of objects in $\calA$.  Then $
\calA$ generates the $\J$-ary variety $\ob(\cP\Alg)$ over $\cS$.
\end{prop}
\begin{proof}
Let $\calB$ be the $\J$-ary variety over $\cS$ generated by $\calA$.  Then $\calB \subseteq \ob(\cP\Alg)$, since $\calA$ is a subset of the $\J$-ary variety $\ob(\cP\Alg)$ over $\cS$.  By \ref{cl_prop_var}, $\calB$ is closed in $\cS\Alg$ under limits and subobjects.  Hence, our hypothesis entails that $\{T_\cP J \mid J \in \ob\J\} \subseteq \calB$, so by \ref{var_gen_jgen_free} we deduce that $\ob(\cP\Alg) \subseteq \calB$.
\end{proof}

\begin{egg}
Let $R$ be an internal rig in a complete and cocomplete cartesian closed category $\V$, and let $\cP = (\cS,\sfE)$ be the diagrammatic (or, equivalently, free-form) $\SF(\V)$-presentation of internal (left) $R$-modules discussed in \ref{Rmodules}, so that $\cP\Alg = R\Mod$.  Given a finite cardinal $n$, let us write also $n$ to denote the $n$th copower of $1$ in $\V$ and note that the free internal $R$-module on $n$ is the (conical) power $R^n$ in $R\Mod$ (e.g. by \cite[6.4.5]{functional}), where $R$ itself is regarded as an $R$-module by multiplication on the left.  Hence, the single $R$-module $R$ generates the $\SF(\V)$-ary variety $\ob(R\Mod)$ over $\cS$, by \ref{sufficient_cond_gen_var}. \qed
\end{egg}

\begin{egg}
Again let $R$ be an internal rig in a complete and cocomplete cartesian closed category $\V$, but this time let $\cP = (\cS,\sfE)$ be the diagrammatic $\SF(\V)$-presentation of (internal left) $R$-affine spaces discussed in \ref{internalaffinespaces}, so that $\cP\Alg = R\text{-}\mathsf{Aff}$.  Let us regard $R$ itself as a left $R$-affine space (taking $R$-affine combinations in $R$ in the usual way, \cite[8.6]{functional}).  Given a finite cardinal $n$, the free $R$-affine space on $n$ (or rather, on the $n$th copower of $1$ in $\V$) is the subobject $R^{n,\aff} \hookrightarrow R^n$ in $R\text{-}\mathsf{Aff}$ of $n$-tuples with sum $1$, by \cite[\S 8]{functional} (or by \ref{internalaffinespaces}).  Hence, the single $R$-affine space $R$ generates the $\SF(\V)$-ary variety $\ob(R\text{-}\mathsf{Aff})$ over $\cS$, by \ref{sufficient_cond_gen_var}.\qed
\end{egg}

\noindent Noting that \ref{cl_prop_var} involves closure properties for $\J$-ary varieties over $\cS$ that are related to the Birkhoff variety theorem \cite{Birkhoff:Var}, it is natural to ask whether a generalization of the latter theorem can be proved in the present general setting, perhaps under additional hypotheses; we leave this question for future work.

\bibliographystyle{amsplain}
\bibliography{mybib}

\providecommand{\bysame}{\leavevmode\hbox to3em{\hrulefill}\thinspace}
\providecommand{\MR}{\relax\ifhmode\unskip\space\fi MR }
\providecommand{\MRhref}[2]{%
  \href{http://www.ams.org/mathscinet-getitem?mr=#1}{#2}
}
\providecommand{\href}[2]{#2}
\begin{thebibliography}{10}

\bibitem{categoricalviewordered}
J.~Ad\'amek, M.~Dost\'al, and J.~Velebil, \emph{A categorical view of varieties
  of ordered algebras}, Math. Struct. Comput. Sci. (2022), 1--25.

\bibitem{finitarymonadsposets}
Ji\v{r}\'{i} Ad\'{a}mek, Chase Ford, Stefan Milius, and Lutz Schr\"{o}der,
  \emph{Finitary monads on the category of posets}, Math. Struct. Comput. Sci.
  (2021), 1--23.

\bibitem{Altenkirchmonads}
Thorsten Altenkirch, James Chapman, and Tarmo Uustalu, \emph{Monads need not be
  endofunctors}, Log. Methods Comput. Sci. \textbf{11} (2015), no.~1, 1:3, 40.

\bibitem{BMW}
C.~Berger, P.-E. Melli{\`e}s, and M.~Weber, \emph{Monads with arities and their
  associated theories}, J. Pure Appl. Algebra \textbf{216} (2012), no.~8-9,
  2029--2048.

\bibitem{Birkhoff:Var}
Garrett Birkhoff, \emph{On the structure of abstract algebras}, Proc. Camb.
  Phil. Soc. \textbf{31} (1935), 433--454.

\bibitem{Birkhoff:Lattice}
\bysame, \emph{Lattice {T}heory}, American Mathematical Society, New York,
  1940.

\bibitem{Borceux1}
Francis Borceux, \emph{Handbook of categorical algebra 1}, Encyclopedia of
  Mathematics and its Applications, vol.~50, Cambridge University Press,
  Cambridge, 1994.

\bibitem{Borceux2}
\bysame, \emph{Handbook of categorical algebra 2}, Encyclopedia of Mathematics
  and its Applications, vol.~51, Cambridge University Press, Cambridge, 1994.

\bibitem{BorceuxDay}
Francis Borceux and Brian Day, \emph{Universal algebra in a closed category},
  J. Pure Appl. Algebra \textbf{16} (1980), no.~2, 133--147.

\bibitem{BourkeGarner}
John Bourke and Richard Garner, \emph{Monads and theories}, Adv. Math.
  \textbf{351} (2019), 1024--1071.

\bibitem{Burroni}
Albert Burroni, \emph{Alg\`ebres graphiques: sur un concept de dimension dans
  les langages formels}, Cahiers Topologie G\'{e}om. Diff\'{e}rentielle
  \textbf{22} (1981), no.~3, 249--265.

\bibitem{Cohn}
P.~M. Cohn, \emph{Universal algebra}, Harper \& Row, Publishers, New
  York-London, 1965.

\bibitem{Diers}
Yves Diers, \emph{Foncteur pleinement fid\'{e}le dense classant les
  alg\'{e}bres}, Cahiers Topologie G\'{e}om. Diff\'{e}rentielle Cat\'{e}g.
  \textbf{17} (1976), no.~2, 171--186.

\bibitem{Dubucbook}
Eduardo~J. Dubuc, \emph{Kan extensions in enriched category theory}, Lecture
  Notes in Mathematics, Vol. 145, Springer-Verlag, Berlin-New York, 1970.

\bibitem{Dubucsemantics}
\bysame, \emph{Enriched semantics-structure (meta) adjointness}, Rev. Un. Mat.
  Argentina \textbf{25} (1970/71), 5--26.

\bibitem{Fiore:MES}
Marcelo Fiore, \emph{An equational metalogic for monadic equational systems},
  Theory Appl. Categ. \textbf{27} (2012), No. 18, 464--492.

\bibitem{Fioreclones}
\bysame, \emph{On the concrete representation of discrete enriched abstract
  clones}, Tbilisi Math. J. \textbf{10} (2017), no.~3, 297--328.

\bibitem{termequationalsystems}
Marcelo Fiore and Chung-Kil Hur, \emph{Term equational systems and logics},
  Electron. Notes Theor. Comput. Sci. \textbf{218} (2008), 171--192.

\bibitem{Freydtensor}
P.~Freyd, \emph{Algebra valued functors in general and tensor products in
  particular}, Colloq. Math. \textbf{14} (1966), 89--106.

\bibitem{Kellystr}
G.~M. Kelly, \emph{Structures defined by finite limits in the enriched context
  {I}}, Cahiers Topologie G\'{e}om. Diff\'{e}rentielle Cat\'{e}g. \textbf{23}
  (1982), no.~1, 3--42.

\bibitem{Kelly}
\bysame, \emph{Basic concepts of enriched category theory}, Repr. Theory Appl.
  Categ. (2005), no.~10, Reprint of the 1982 original [Cambridge Univ. Press,
  Cambridge].

\bibitem{KellyLackstronglyfinitary}
G.~M. Kelly and Stephen Lack, \emph{Finite-product-preserving functors, {K}an
  extensions and strongly-finitary {$2$}-monads}, Appl. Categ. Structures
  \textbf{1} (1993), no.~1, 85--94.

\bibitem{KellyPower}
G.~M. Kelly and A.~J. Power, \emph{Adjunctions whose counits are coequalizers,
  and presentations of finitary enriched monads}, J. Pure Appl. Algebra
  \textbf{89} (1993), no.~1-2, 163--179.

\bibitem{Kock}
Anders Kock, \emph{Monads on symmetric monoidal closed categories}, Arch. Math.
  (Basel) \textbf{21} (1970), 1--10.

\bibitem{Lackmonadicity}
Stephen Lack, \emph{On the monadicity of finitary monads}, J. Pure Appl.
  Algebra \textbf{140} (1999), no.~1, 65--73.

\bibitem{LR}
Stephen Lack and Ji\v{r}\'{\i} Rosick\'{y}, \emph{Notions of {L}awvere theory},
  Appl. Categ. Structures \textbf{19} (2011), no.~1, 363--391.

\bibitem{Law:Phd}
F.~W. Lawvere, \emph{Functorial semantics of algebraic theories}, Dissertation,
  Columbia University, New York. Available in: \emph{Repr. Theory Appl. Categ.}
  \textbf{5} (2004), 1963.

\bibitem{Linton:Eq}
F.~E.~J. Linton, \emph{Some aspects of equational categories}, Proc. {C}onf.
  {C}ategorical {A}lgebra ({L}a {J}olla, {C}alif., 1965), Springer, New York,
  1966, pp.~84--94.

\bibitem{Lintonoutline}
\bysame, \emph{An outline of functorial semantics}, Sem. on {T}riples and
  {C}ategorical {H}omology {T}heory ({ETH}, {Z}\"{u}rich, 1966/67), Springer,
  Berlin, 1969, pp.~7--52.

\bibitem{enrichedfact}
Rory B.~B. Lucyshyn-Wright, \emph{Enriched factorization systems}, Theory Appl.
  Categ. \textbf{29} (2014), No. 18, 475--495.

\bibitem{EAT}
\bysame, \emph{Enriched algebraic theories and monads for a system of arities},
  Theory Appl. Categ. \textbf{31} (2016), No. 5, 101--137.

\bibitem{functional}
\bysame, \emph{Functional distribution monads in functional-analytic contexts},
  Adv. Math. \textbf{322} (2017), 806--860.

\bibitem{Commutants}
\bysame, \emph{Commutants for enriched algebraic theories and monads}, Appl.
  Categ. Structures \textbf{26} (2018), no.~3, 559--596.

\bibitem{locbd}
Rory B.~B. Lucyshyn-Wright and Jason Parker, \emph{Locally bounded enriched
  categories}, Theory Appl. Categ. \textbf{38} (2022), 684--736.

\bibitem{Pres}
\bysame, \emph{Presentations and algebraic colimits of enriched monads for a
  subcategory of arities}, Preprint, arXiv:2201.03466, 2022.

\bibitem{Manesbook}
Ernest~G. Manes, \emph{Algebraic theories}, Graduate Texts in Mathematics, No.
  26, Springer-Verlag, New York-Heidelberg, 1976.

\bibitem{McCrudden}
Paddy McCrudden, \emph{Balanced coalgebroids}, Theory Appl. Categ. \textbf{7}
  (2000), No. 6, 71--147.

\bibitem{NishizawaPower}
Koki Nishizawa and John Power, \emph{Lawvere theories enriched over a general
  base}, J. Pure Appl. Algebra \textbf{213} (2009), no.~3, 377--386.

\bibitem{PlotkinPowernotions}
Gordon Plotkin and John Power, \emph{Notions of computation determine monads},
  Foundations of software science and computation structures ({G}renoble,
  2002), Lecture Notes in Comput. Sci., vol. 2303, Springer, Berlin, 2002,
  pp.~342--356.

\bibitem{Robinson}
Edmund Robinson, \emph{Variations on algebra: Monadicity and generalisations of
  equational theories}, Form. Asp. Comput. \textbf{13} (2002), 308--326.

\bibitem{Statonpredicate}
Sam Staton, \emph{An algebraic presentation of predicate logic (extended
  abstract)}, Foundations of software science and computation structures,
  Lecture Notes in Comput. Sci., vol. 7794, Springer, Heidelberg, 2013,
  pp.~401--417.

\bibitem{Statoninstances}
\bysame, \emph{Instances of computational effects: an algebraic perspective},
  28th {A}nnual {ACM}/{IEEE} {S}ymposium on {L}ogic in {C}omputer {S}cience
  ({LICS} 2013), 2013, pp.~519--528.

\bibitem{WolffVcat}
Harvey Wolff, \emph{{$V$}-cat and {$V$}-graph}, J. Pure Appl. Algebra
  \textbf{4} (1974), 123--135.

\end{thebibliography}

\end{document}